\documentclass[11pt,letterpaper]{amsart}

\pdfoutput=1

\usepackage{amsmath}
\usepackage{amsfonts}
\usepackage{amssymb}
\usepackage{enumerate}
\usepackage{amstext}
\usepackage{amsbsy}
\usepackage{amsopn}
\usepackage{bbm,amsthm}
\usepackage{amscd}
\usepackage[pdftex]{color}
\usepackage{amsxtra}
\usepackage{upref}
\usepackage{epstopdf}
\usepackage{graphicx,color}
\usepackage{hyperref}
\usepackage{longtable}
\usepackage{graphicx}
\usepackage[english]{babel}
\usepackage{todonotes}

\DeclareFontFamily{OML}{rsfs}{\skewchar\font'177}
\DeclareFontShape{OML}{rsfs}{m}{n}{ <5> <6> rsfs5 <7> <8> <9> rsfs7
  <10> <10.95> <12> <14.4> <17.28> <20.74> <24.88> rsfs10 }{}
\DeclareMathAlphabet{\mathfs}{OML}{rsfs}{m}{n}

\newtheorem{theorem}{Theorem}
\newtheorem{lemma}[theorem]{Lemma}
\newtheorem{proposition}[theorem]{Proposition}

\theoremstyle{definition}

\theoremstyle{remark}
\newtheorem{remark}[theorem]{\bf Remark}

\numberwithin{equation}{section}
\numberwithin{theorem}{section}

\newcommand{\intav}[1]{\mathchoice {\mathop{\vrule width 6pt height 3 pt depth  -2.5pt
\kern -8pt \intop}\nolimits_{\kern -6pt#1}} {\mathop{\vrule width
5pt height 3  pt depth -2.6pt \kern -6pt \intop}\nolimits_{#1}}
{\mathop{\vrule width 5pt height 3 pt depth -2.6pt \kern -6pt
\intop}\nolimits_{#1}} {\mathop{\vrule width 5pt height 3 pt depth
-2.6pt \kern -6pt \intop}\nolimits_{#1}}}

\newcommand{\intavl}[1]{\mathchoice {\mathop{\vrule width 6pt height 3 pt depth  -2.5pt
\kern -8pt \intop}\limits_{\kern -6pt#1}} {\mathop{\vrule width 5pt
height 3  pt depth -2.6pt \kern -6pt \intop}\nolimits_{#1}}
{\mathop{\vrule width 5pt height 3 pt depth -2.6pt \kern -6pt
\intop}\nolimits_{#1}} {\mathop{\vrule width 5pt height 3 pt depth
-2.6pt \kern -6pt \intop}\nolimits_{#1}}}

\newcommand{\plus}{{-}}
\newcommand{\minus}{{+}}
\newcommand{\SSigma}{{\Sigma}}

\newcommand{\ve}{\varepsilon}
\newcommand{\wt}{\widetilde}
\newcommand{\wh}{\widehat}

\newcommand{\R}{\mathbb{R}}
\newcommand{\N}{\mathbb{N}}

\newcommand{\Z}{\mathbb{Z}}

\newcommand\Tau{\mathrm{T}}
\newcommand\x{{\bf x}}

\begin{document}

\title[Polynomial decay of correlations in nonpositive curvature]{Polynomial decay of correlations for \\
nonpositively curved surfaces}

\author{Yuri Lima, Carlos Matheus, and Ian Melbourne}
\date{\today}
\thanks{The research of YL and IM was supported in part by a Newton Advanced Fellowship, grant
NAF\textbackslash R1\textbackslash 180108, from The Royal Society. YL was also supported by CNPq and Instituto Serrapilheira,
grant ``Jangada Din\^{a}mica: Impulsionando Sistemas Din\^{a}micos na Regi\~{a}o Nordeste''.}

\address{Yuri Lima, Departamento de Matem\'atica, Centro de Ci\^encias, Campus do Pici,
Universidade Federal do Cear\'a (UFC), Fortaleza -- CE, CEP 60455-760, Brasil}
\email{yurilima@gmail.com}
\address{Carlos Matheus, CNRS \& \'Ecole Polytechnique, CNRS (UMR 7640), 91128, Palaiseau, France}
\email{carlos.matheus@math.cnrs.fr}
\address{Ian Melbourne, Mathematics Institute, University of Warwick, Coventry, CV4 7AL, UK}
\email{i.melbourne@warwick.ac.uk}

\begin{abstract}
We prove polynomial decay of correlations for geodesic flows on a class of nonpositively curved
surfaces where zero curvature only occurs along one closed geodesic.
We also prove that various statistical limit laws, including the central limit theorem, are satisfied by this class of geodesic flows.
\end{abstract}

\maketitle

\begin{flushright}
{\it Dedicated to the memory of Todd Fisher}
\end{flushright}

\tableofcontents

\section{Introduction}\label{sec:introduction}

The goal of this work is to provide examples of geodesic flows
on nonpositively curved manifolds with polynomial decay of correlations.

\begin{theorem}\label{thm:main} Let $r\in [4,\infty)$, and let $S$ be a closed Riemannian surface of nonpositive curvature obtained by isometrically gluing two negatively curved surfaces with boundaries to the boundaries of the surface of revolution with profile $1+|s|^r$, $|s|\leq 1$. Let $M=T^1S$ and $a=\tfrac{r+2}{r-2}\in(1,3]$. Then the geodesic flow
$g_t:M\to M$ has polynomial decay of correlations with respect to the normalized Riemannian
volume $\mu$: for all $\epsilon>0$ and all sufficiently smooth observables $\phi,\psi:M\to\R$, there is a constant
$C(\phi,\psi)$ such that 
$$
\left|\int_{M} \phi \cdot (\psi\circ g_t)\,d\mu-\int_{M}\phi\,d\mu\int_{M}\psi\,d\mu\right|
\leq C(\phi,\psi) \frac{1}{t^{a-\epsilon}}\, \ \text{for all $ t> 0$}.
$$
\end{theorem}

The precise meaning of ``sufficiently smooth'' is explained in Section~\ref{sec:conclusion}.
In addition, we state and prove statistical limit laws such as the central limit theorem (CLT) for H\"older observables $\phi:M\to\R$. 

Figure~\ref{fig:surface-intro} depicts the surfaces $S$ considered in Theorem \ref{thm:main}: 
the region between two curves $\alpha$ and $\beta$ is a surface of revolution 
with profile $1+|s|^r$, $|s|\leq 1$, thus the curve $\gamma$ with $s=0$ is a closed geodesic with zero
curvature; outside this region, $S$ has negative curvature.
We call $\gamma$ the {\em degenerate closed geodesic}.
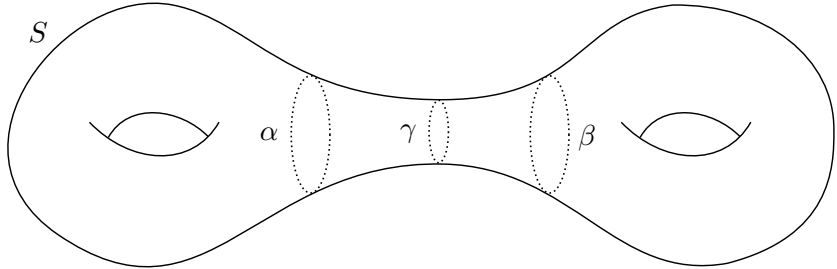
\begin{figure}[hbt!]
\centering
\def\svgwidth{11cm}
\begingroup%
  \makeatletter%
  \providecommand\color[2][]{%
    \errmessage{(Inkscape) Color is used for the text in Inkscape, but the package 'color.sty' is not loaded}%
    \renewcommand\color[2][]{}%
  }%
  \providecommand\transparent[1]{%
    \errmessage{(Inkscape) Transparency is used (non-zero) for the text in Inkscape, but the package 'transparent.sty' is not loaded}%
    \renewcommand\transparent[1]{}%
  }%
  \providecommand\rotatebox[2]{#2}%
  \newcommand*\fsize{\dimexpr\f@size pt\relax}%
  \newcommand*\lineheight[1]{\fontsize{\fsize}{#1\fsize}\selectfont}%
  \ifx\svgwidth\undefined%
    \setlength{\unitlength}{430.99340532bp}%
    \ifx\svgscale\undefined%
      \relax%
    \else%
      \setlength{\unitlength}{\unitlength * \real{\svgscale}}%
    \fi%
  \else%
    \setlength{\unitlength}{\svgwidth}%
  \fi%
  \global\let\svgwidth\undefined%
  \global\let\svgscale\undefined%
  \makeatother%
  \begin{picture}(1,0.3204114)%
    \lineheight{1}%
    \setlength\tabcolsep{0pt}%
    \put(0,0){\includegraphics[width=\unitlength,page=1]{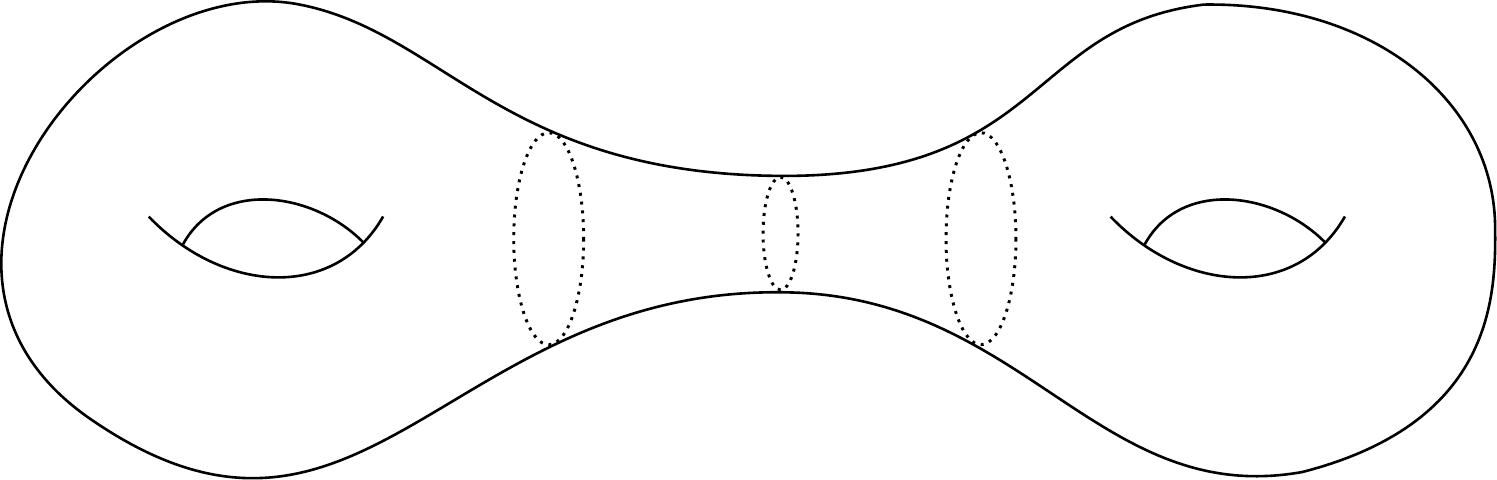}}%
    \put(0.47485944,0.15742031){\color[rgb]{0,0,0}\makebox(0,0)[lt]{\lineheight{1.25}\smash{\begin{tabular}[t]{l}$\gamma$\end{tabular}}}}%
    \put(0.30522619,0.15473895){\color[rgb]{0,0,0}\makebox(0,0)[lt]{\lineheight{1.25}\smash{\begin{tabular}[t]{l}$\alpha$\end{tabular}}}}%
    \put(0.69022657,0.14860366){\color[rgb]{0,0,0}\makebox(0,0)[lt]{\lineheight{1.25}\smash{\begin{tabular}[t]{l}$\beta$\end{tabular}}}}%
    \put(0.02584672,0.27273526){\color[rgb]{0,0,0}\makebox(0,0)[lt]{\lineheight{1.25}\smash{\begin{tabular}[t]{l}$S$\end{tabular}}}}%
  \end{picture}%
\endgroup%

\caption{Surfaces 
with degenerate closed geodesic $\gamma$ considered in Theorem \ref{thm:main}.}
\label{fig:surface-intro}
\end{figure}

\begin{remark} \label{rmk:main}
We expect that the mixing rate in Theorem~\ref{thm:main} is almost sharp.
Indeed, as part of the proof of Theorem~\ref{thm:main}, we construct a piecewise smooth Poincar\'e map $g:\SSigma\to \SSigma$ with piecewise smooth first hit time that is bounded above and below (away from zero), and $g$ has the sharp polynomial mixing rate $n^{-a}$.

More precisely, let $\mu_{\SSigma}$ denote the corresponding Liouville probability measure on $\SSigma$.
In Section~\ref{sec:conclusion}, we apply \cite{Young-polynomial} to show that
for all H\"older observables $\phi,\psi:\SSigma\to\R$, there is a constant
$C(\phi,\psi)$ such that for all integers $n\ge2$,
$$
\left|\int_{\SSigma} \phi \cdot (\psi\circ g^n)\,d\mu_{\SSigma}-\int_{\SSigma}\phi\,d\mu_{\SSigma}\int_{\SSigma}\psi\,d\mu_{\SSigma}\right|
\leq C(\phi,\psi) \frac{(\log n)^{1+a}}{n^a}\cdot
$$
Moreover, by~\cite{BMTapp},
for all H\"older observables $\phi,\psi:\SSigma\to\R$  with nonzero mean and support bounded away from $T^1\gamma$, there is a constant $c(\phi,\psi)$ such that for all integers $n\ge2$,
\begin{equation} \label{eq:lower}
\left|\int_{\SSigma} \phi \cdot (\psi\circ g^n)\,d\mu_{\SSigma}-\int_{\SSigma}\phi\,d\mu_{\SSigma}\int_{\SSigma}\psi\,d\mu_{\SSigma}\right|
\geq c(\phi,\psi) \frac{1}{(\log n)n^a}\cdot
\end{equation}

Obtaining rates for the geodesic flow $g_t$ is more subtle due to the neutral flow direction. 
We obtain the upper bound in Theorem~\ref{thm:main} by applying
the recent work of \cite{Balint-et-al-2019}.
Since the first hit time is bounded
away from zero, there is no reason to expect the flow $g_t$ to decay faster than $g$;
a precise statement of the form~\eqref{eq:lower} for the flow $g_t$ is the subject of work in progress~\cite{BMTprep}.
(We note however
that the bounds obtained in Lemmas \ref{lem:t1-new} and \ref{lem:diffpsi} below
combined with the arguments of \cite{BMMW-1} show that $g_t$ is certainly not exponentially mixing.)
\end{remark}

\begin{remark}\label{rmk:main2}
We believe that the optimal mixing rate for $g_t$ is $t^{-a}$ (similarly $n^{-a}$ for $g$). Indeed, we
expect that the arguments in \cite{Chernov-Zhang-08} can be used to remove the multiplicative
factor $(\log n)^{1+a}$ in Remark~\ref{rmk:main} and the same argument would allow us to take $\epsilon=0$ in Theorem~\ref{thm:main}. However, to focus on the main ideas introduced in this paper,
we do not pursue such an improvement here. 
\end{remark}

In general, the dynamical and statistical properties of geodesic flows in closed Riemannian manifolds is
a fascinating topic whose origin goes back to Artin, Hadamard, Hedlund, Hopf, Klein, Poincar\'e, among others.
Indeed, geodesic flows on manifolds with negative or nonpositive sectional curvature
were the motivation of various breakthroughs in ergodic theory.
One of them was given by Hopf: the nowadays called \emph{Hopf argument} was
used to prove that geodesic flows on 
negatively curved compact surfaces are ergodic with respect to their Liouville volume measure \cite{Hopf-1939}.
Anosov extended Hopf's argument to prove ergodicity of geodesic flows in negative curvature to arbitrary dimensions
\cite{Anosov-Geodesic-Flows}.

Moreover, geodesic flows in manifolds with negative sectional curvature are Bernoulli 
\cite{Ornstein-Weiss-Geodesic-Flows,Ratner-Flows-Bernoulli}, which is the ultimate chaotic
property from a measure-theoretic point of view. After this was established, efforts were made
to understand finer statistical properties such as decay of correlations. Among the developments,
we mention the work of Chernov \cite{Chernov-1998}, Dolgopyat \cite{Dolgopyat-1998} and Liverani
\cite{Liverani-2004} on the exponential decay of correlations for contact Anosov flows
(and in particular geodesic flows on compact manifolds with negative curvature),
and the work of Burns et al. \cite{BMMW-1,BMMW-2} on the rates of mixing of the 
Weil-Petersson geodesic flows on moduli spaces of Riemann surfaces.
We mention that, beside its intrinsic interest, exponential mixing for geodesic (and frame) flows
has applications to other fields such as the geometry of lattices \cite{Eskin-McMullen},
number theory \cite{Kleinbock-Margulis}, and the topology of $3$--manifolds \cite{Kahn-Markovic-Annals2012}.

While geodesic flows in negative curvature are the prototypical examples of uniformly hyperbolic
flows, geodesic flows in nonpositive curvature are the prototypical examples of nonuniformly
hyperbolic flows, and are much harder to study. For instance, the ergodicity of the Liouville
measure is still an open problem. Pesin developed a global theory
for nonuniformly hyperbolic systems, nowadays called {\em Pesin theory}, and used it to derive
the ergodicity of the Liouville measure in various contexts
\cite{Pesin-Characteristic-1977,Pesin-geodesic-flows}, 
see also the discussion in \cite[p.\ 5]{Ballmann}.
In particular, the geodesic flows considered in Theorem \ref{thm:main} are ergodic for the Liouville measure.
There has been recent progress on the measure-theoretic
properties of these flows, including the uniqueness of the measure of maximal entropy
\cite{Knieper-Rank-One-Entropy}, the uniqueness of other classes of equilibrium states \cite{Burns-et-al},
and the Bernoulli property \cite{Ledrappier-Lima-Sarig}, among others.
Previously, nothing was known about decay of correlations or the CLT for such flows, and this paper gives the first
contribution in this direction. 
(Although this paper gives the first ``classical'' CLT, an asymptotic version of the CLT for measures converging to the measure of maximal entropy was recently obtained by~\cite{ThompsonWang}.)

To prove Theorem \ref{thm:main} we use an axiomatic approach, nowadays called
{\em Chernov axioms}, first developed by Chernov to prove exponential decay of correlations for
dispersing billiard maps \cite{Chernov-1999}.
We actually follow a simplification of this work given by Chernov and Zhang \cite{Chernov-Zhang}.
We apply these works to a uniformly hyperbolic map with singularities $f$, equal to the return
map of a Poincar\'e section that does not intersect the degenerate closed geodesic. 
Then the method of Markarian~\cite{Markarian-polynomial, Chernov-Zhang, BMTapp} enables us to establish polynomial decay of correlations for $h$ and $g_t$.

\begin{remark}
When establishing mixing rates for billiards, the main step is to verify complexity bounds, due to the fact that the remaining Chernov axioms had already been verified for many classes of examples in the previous twenty years starting with~\cite{BSC91}, as discussed in~\cite[Section~4]{Chernov-Zhang}. However, this is not the case for geodesic flows in nonpositive curvature, so the current paper aims to lay the groundwork for verifying all of the Chernov axioms for general classes of geodesic flows, in addition to treating the specific example in Theorem~\ref{thm:main}.
\end{remark}

One of the Chernov axioms is that invariant manifolds have uniformly bounded curvature. 
This is a delicate point for the surfaces we consider. For instance, Ballmann, Brin and Burns
showed that in a surface of revolution with profile $1+s^4$ (i.e. $r=4$ in 
Theorem \ref{thm:main}) the invariant manifolds of the degenerate closed geodesic are not $C^2$,
hence the curvature is not even defined \cite{BBB-1987}. To avoid this, we verify that 
\cite{Chernov-1999,Chernov-Zhang} works under the weaker assumption that the invariant manifolds
have uniformly bounded $C^{1+\rm{Lip}}$ norms, and we exploit the fact that this latter property is satisfied
in the class of surfaces we consider, by a result of Gerber and Wilkinson \cite{Gerber-Wilkinson},
see Theorem \ref{thm:Gerber-Wilkinson}.

Some of the Chernov axioms are related to hyperbolicity properties of the uniformly hyperbolic
map $f$ mentioned above. In negative curvature, these properties
are usually obtained by estimating solutions of the Riccati equation. Unfortunately, the presence of zero
curvature weakens such estimates, and we were not able to use them to establish the required axioms.
Instead, we follow a different approach and use a system of coordinates in the unit tangent bundle
of the surface of revolution, called {\em Clairaut coordinates}. In these coordinates, estimates
for~$f$ are almost sharp.

In addition, the Poincar\'e section has to satisfy some geometrical and dynamical properties.
One of them is the absence of triple intersections for a sufficiently large number of pre-iterates
under $f$ of a finite family of compact curves. This task would be a simple application
of perturbative methods if these pre-iterates remained compact. However, $f$ has unbounded
derivative, and the pre-images of some compact curves have infinite length. This phenomenon
is related to the homoclinic points of the degenerate closed geodesic, and a thorough analysis of 
the dynamics of $f$ is required to implement the perturbative methods successfully.

\begin{remark} \label{rmk:CZ05b}
Chernov and Zhang considered an analogous class of dispersing billiard maps~\cite{ChernovZhang05b}
for which the obstacles are convex with nonvanishing curvature
except at two flat points where the obstacles have profile $\pm(1+|s|^r)$, $r>2$.
They obtained the same upper bound $(\log n)^{a+1} n^{-a}$ as in Remark~\ref{rmk:main} and the lower bound $(\log n)^{-1} n^{-a}$ was proved in~\cite{BMTapp}.
The associated semiflow has the same polynomial
decay of correlations \cite{M07,Melbourne-2018}, but the analogue of Theorem
\ref{thm:main} for the billiard flow remains unproved (this is the final open question in~\cite[Section~9]{Melbourne-2018}).

Contrary to~\cite{ChernovZhang05b}, we require $r\geq 4$ rather than just $r > 2$ because this provides the usual $C^4$ regularity  required to apply several tools from the theory of geodesic flows described in Section~\ref{sec:surfaces} (from the properties of elementary ordinary differential equations
like the Jacobi and Riccati equations to more recent results such as Theorem~\ref{thm:Gerber-Wilkinson}).
\end{remark}

The paper is organized as follows. In Section \ref{sec:surfaces} we review known facts about the
geometry and dynamics of geodesic flows on surfaces, with special attention to the class of surfaces
considered in Theorem \ref{thm:main}. In particular, we state the main results of Gerber and Wilkinson
\cite{Gerber-Wilkinson} that we will use. Section \ref{sec:exponential.mixing} presents the axiomatic approach
of \cite{Chernov-1999,Chernov-Zhang}, and includes the justification that uniform bounds
on the $C^{1+{\rm Lip}}$ norms of invariant manifolds are enough, see Remark \ref{rmk:A5}.
In Section \ref{sec:neck-dynamics} we make a systematic study of the dynamics of the geodesic flow near 
the degenerate closed geodesic, which is related to explosion of the derivative of $f$. Here we make substantial
use of the Clairaut coordinates. In Section \ref{sec:construction-Poincare-section} we construct
the Poincar\'e section.
We also prove
that the roof function of the constructed Poincar\'e section has polynomial tails (Lemma \ref{lem:tails}),
and prove some hyperbolicity estimates for $f$, see Section \ref{ss:hyperbolicity-f}.
Using these results, we prove in Section \ref{sec:CZ-scheme} that $f$ indeed satisfies 
the Chernov axioms. Finally, we prove Theorem \ref{thm:main} and various statistical limit laws in Section~\ref{sec:conclusion}.

\section{Surfaces with nonpositive curvature}\label{sec:surfaces}

In this section, we recall some known facts on differential geometry, most specifically
on geodesic flows in nonpositively curved surfaces and surfaces of revolution.
We also give a precise description
of the class of surfaces we consider in this article, and describe the properties of these
surfaces that will be used in the sequel.

\subsection{Geodesic flows}\label{sec:geodesic-flows}

Let $S$ be a closed Riemannian surface. Let $M=T^1S$ be its unit tangent bundle,
which is a closed three dimensional Riemannian manifold. There is a natural metric on $M$,
called the {\em Sasaki metric}, which is the product of horizontal and vertical vectors,
see e.g.\ \cite[Chapter 3, Exercise 2]{Manfredo}. We write $\|\cdot\|_{\rm Sas}$ for the norm induced by 
the Sasaki metric. The volume form on $S$ induces a smooth probability measure $\mu$ on $M$.

\medskip
\noindent
{\sc Geodesic flow $\{g_ t\}$:} The {\em geodesic flow} on $S$ is the flow $\{g_t\}_{t\in\R}:M\to M$ defined by
$g_t(x)=\gamma'_x(t)$, where $\gamma_x:\R\to S$ is the unique geodesic such that $\gamma'_x(0)=x$.
For simplicity, we denote the geodesic flow by $g_t$. 
The probability measure $\mu$ is invariant under $g_t$.

\medskip
For $p\in S$, let $K(p)$ be the Riemannian curvature at $p$.
We assume that $S$ has {\em nonpositive} Riemannian curvature: $K(p)\leq 0$ for all $p\in S$.
The dynamical properties of $g_t$ are intimately related to the curvature of $S$. 

\medskip
\noindent
{\sc Degenerate and Regular sets:} The {\em degenerate set} of $g_t$ 
is defined by
$$
{\rm Deg}=\{x\in M:K(\gamma_x(t))=0\text{ for all }t\in\R\},
$$
and the {\em regular set} of $g_t$ is defined by
$$
{\rm Reg}=M\setminus {\rm Deg}=\{x\in M:K(\gamma_x(t))<0\text{ for some }t\in\R\}.
$$

\medskip
Clearly, ${\rm Deg}$ and ${\rm Reg}$ form a partition of $M$, with ${\rm Deg}$
closed and ${\rm Reg}$ open.

\begin{remark} The classical literature uses the terminology
``singular set'' instead of ``degenerate set'', but here we reserve the term
``singular'' for the dynamical setting of the Chernov axioms.
\end{remark}

\begin{theorem}[Pesin \cite{Pesin-geodesic-flows}]\label{thm:ergodicity}
If $\mu[{\rm Deg}]=0$, then the flow $(g_t,\mu)$ is ergodic.
\end{theorem}
\noindent See also the discussion in \cite[p.\ 5]{Ballmann}.

\medskip
The dynamical properties of $g_t$ are usually studied via Jacobi fields.

\medskip
\noindent
{\sc Jacobi field:} A vector field $J:t\mapsto J(t)\in T_{\gamma(t)}S$
along a geodesic $\gamma$ is called a {\em Jacobi field}
if it satisfies the {\em Jacobi equation}
$$
J''(t)+K(\gamma(t))J(t)=0.
$$
If $J(t),J'(t)$ are perpendicular to $\gamma'(t)$ for
some (and hence all) $t$, then $J$ is called a {\em perpendicular Jacobi field}.

\medskip
For every $x\in M$ there is an isomorphism
$$
T_xM \leftrightarrow \{(J(0),J'(0)):J\text{ is a Jacobi field along }\gamma_x\text{ with }J'(0)\perp \gamma_x'(0)\}. 
$$
Under this identification, the Sasaki metric is equal to $\|(J,J')\|_{\rm Sas}^2=\|J\|^2+\|J'\|^2$.
Additionally, we have $dg_t(J(0),J'(0))=(J(t),J'(t))$, and this is one of the reasons
why Jacobi fields provide dynamical information of $g_t$.
Under our curvature assumptions, we can characterize stable and unstable subspaces.

\medskip
\noindent
{\sc Stable and unstable Jacobi fields:} A Jacobi field $J$ is called {\em stable} if $\|J(t)\|$
is uniformly bounded for all $t\geq 0$, and {\em unstable} if $\|J(t)\|$ is uniformly
bounded for all $t\leq 0$.

\medskip
\noindent
{\sc Stable and unstable subspaces:} The {\em stable subspace}
of $x\in M$ is
$$
\wh E^s_x=\{(J(0),J'(0)):J\text{ is a stable perpendicular Jacobi field along }\gamma_x\},
$$
and the {\em unstable subspace} of $x\in M$ is
$$
\wh E^u_x=\{(J(0),J'(0)):J\text{ is an unstable perpendicular Jacobi field along }\gamma_x\}.
$$

\begin{remark}
We reserve the notation $E^{s/u}_x$ for the stronger notion of stable/unstable
subspace in the sense of hyperbolic dynamics, as described in Section
\ref{sec:construction-Poincare-section}.
\end{remark}

\medskip
Let $Z_x$ denote the one-dimensional subspace of $T_xM$ tangent to the geodesic flow.
The following are known facts of these subspaces, see e.g. \cite{Eberlein}. 

\begin{lemma}\label{lem:subspaces}
The families $\{\wh E^s_x\},\{\wh E^u_x\}$ satisfy the following properties:
\begin{enumerate}[{\rm (1)}]
\item {\sc Invariance:} $\{\wh E^s_x\},\{\wh E^u_x\}$ are $dg_t$--invariant for all $t\in\R$.
\item {\sc Dimension:} $\wh E^s_x,\wh E^u_x$ have dimension one and are orthogonal to $Z_x$.
\item {\sc Continuity:} The maps $x\mapsto \wh E^s_x,\wh E^u_x$ are continuous.
\item {\sc Relation with ${\rm Deg,Reg}$:} $\wh E^s_x=\wh E^u_x$ if and only if $x\in {\rm Deg}$, hence
$\wh E^s_x\oplus Z_x \oplus \wh E^u_x=T_xM$ if and only if $x\in{\rm Reg}$.
\end{enumerate}
\end{lemma}

Next, we consider the invariant manifolds of $g_t$. We first
define the invariant manifolds for the geodesic flow $\wt g_t$ on the universal
cover $\wt S$ of $S$. For that, we consider Busemann functions and horospheres.
Our discussion follows \cite[Section~IV.A]{Eberlein}. For each $v\in T^1\wt S$, let
$\wt\gamma_v$ be the unique geodesic of $\wt S$ with $\wt \gamma'_v(0)=v$.
Given $t\in\R$, define the function $B_{v,t}:\wt S\to\R$ by $B_{v,t}(x)=d(x,\wt\gamma_v(t))-t$. 
(Here, $d$ is the unique metric on $\wt S$ making the covering map a local isometry.)
The {\em Busemann function} of $v$ is the limit function $B_v:\wt S\to\R$ defined
by $B_v=\lim\limits_{t\to+\infty}B_{v,t}$. Since $S$ has nonpositive curvature, it follows
from Eberlein that
each $B_v$ is $C^2$  \cite[Section~IV.A]{Eberlein}, see also \cite[Prop.\ 3.1]{Heintze-Im-Hof}.

\medskip
\noindent
{\sc Horospheres:} 
The {\em stable horosphere} at $v\in T^1\wt S$
is the set $H^s(v)\subset\wt S$ defined as $H^s(v)=B_v^{-1}(0)$.
The {\em unstable horosphere} at $v\in T^1\wt S$ is the set $H^u(v)\subset\wt S$
defined by $H^u(v)=(B_{-v})^{-1}(0)$.

\medskip 
Each $H^{s/u}(v)$ is a $C^2$ curve of $\wt S$, and $v\mapsto H^{s/u}(v)$ is continuous, see
e.g. \cite[p.\ 25]{Ballmann}. These curves define invariant foliations for $\wt g_t$.

\medskip
\noindent
{\sc Invariant manifolds for $\wt g_t$:} The {\em stable manifold} for $\wt g_t$ at $v\in T^1\wt S$
is the graph over $H^s(v)$ defined by
$$
\wt W^s_v=\left\{w\in T^1\wt S:\begin{array}{l}w\text{ is perpendicular to and has basepoint at }\\
H^s(v),\text{ pointing in the same direction of }v\end{array}\right\}.
$$  
The {\em unstable manifold} for $\wt g_t$ at $v\in T^1\wt S$ is defined analogously.

\medskip
Since $H^{s/u}(v)$ is $C^2$, its normal subbundle is $C^1$, i.e.\ each leaf $\wt W^{s/u}_v$
is $C^1$. 

\medskip
\noindent
{\sc Invariant manifolds for $g_t$:} The stable/unstable manifolds of
$g_t$ at $x\in M$ are the projections to $M$ of the stable/unstable manifolds of $\wt g_t$ 
at some (every) $v\in T^1\wt S$ that projects to $x$. We denote them by $\wh W^{s/u}_x$. 

\medskip
By the above discussion
the curves $\wh W^{s/u}_x$ are $C^1$ for all $x\in M$. 
Under additional conditions, Gerber \& Wilkinson proved a stronger regularity \cite{Gerber-Wilkinson}
and also a property about the tangent distributions $\wh E^{s/u}$, see 
Theorem \ref{thm:Gerber-Wilkinson} below.

Next, we discuss the link between Jacobi fields and horospheres.
Fix $x\in M$, and let $J_-(t)$ be a stable perpendicular Jacobi field along $\gamma_x$.
If $E(t)$ is a unitary parallel vector field orthogonal to $\gamma_x$, then 
$J_-(t)=j_-(t)E(t)$ where $j_-(t)=\|J_-(t)\|$ satisfies the {\em scalar Jacobi equation}
$$
j''(t)+K(\gamma_x(t))j(t)=0.
$$
The logarithmic derivative $u_{x,-}(t)=\tfrac{j'_-(t)}{j_-(t)}=[\log j_-(t)]'$ satisfies the {\em Riccati equation}
$$
u'(t)+u(t)^2+K(\gamma_x(t))=0.
$$
Define $u_-:M\to \R$ by $u_-(x)=u_{x,-}(0)$. It is known that $u_-(x)$ is the geodesic curvature of
the curve $H^s(v)$ at the basepoint of $v$ for some (every) $v\in T^1\wt S$ that projects to $x$.
Similarly, for each $x\in M$ define a function $u_{x,+}$; then $u_+:M\to \R$ defined by
$u_+(x)=u_{x,+}(0)$ is the geodesic curvature of $H^u(v)$ at the basepoint of $v$
for some (every) $v\in T^1\wt S$ that projects to $x$.
The following properties of $u_\pm$ will be essential to us.

\begin{proposition}\label{prop:u-continuous}
The functions $u_\pm$ are continuous, $u_-\leq 0\leq u_+$ and
$u_{\pm}(x)=0$ if and only if $x\in {\rm Deg}$.
\end{proposition}

The continuity of $u_\pm$ follows from the regularity of the Busemann
functions mentioned above. 
The functions $u_\pm$ also provide the
growth rate of the derivative of $g_t$, as follows.
Fix $x\in M$. For an unstable perpendicular Jacobi field $J_+(t)$ along $\gamma_x$, we have
\begin{align*}
j_+(t) & =j_+(0)\exp\left[\displaystyle\int_0^t u_+(g_sx)ds\right]\\
j'_+(t) & =j_+(0)\exp\left[\displaystyle\int_0^t u_+(g_sx)ds\right]u_+(g_tx)
\end{align*}
and so
\begin{align*}
\frac{\|dg_t(J_+(0),J_+'(0))\|_{\rm Sas}}{\|(J_+(0),J_+'(0))\|_{\rm Sas}}&=\frac{\|(J_+(t),J_+'(t))\|_{\rm Sas}}{\|(J_+(0),J_+'(0))\|_{\rm Sas}}\\
&=\sqrt{\dfrac{1+u_+(g_tx)^2}{1+u_+(x)^2}}\exp\left[\displaystyle\int_0^t u_+(g_sx)ds\right].
\end{align*}
By Proposition \ref{prop:u-continuous}, if the orbit segment $g_{[0,t]}x$ is far from ${\rm Deg}$,
then $\wh E^u_x$ is indeed an expanding direction. A similar calculation holds
for stable perpendicular Jacobi fields. We actually work with a variant of the Sasaki metric,
as in \cite[\S 17.6]{Katok-Hasselblatt-Book}.
Let $\delta>0$.

\medskip
\noindent
{\sc $\delta$--Sasaki metric:} The {\em $\delta$--Sasaki metric} is the metric $\|\cdot\|_{\delta-{\rm Sas}}$
satisfying the equality
$$
\|(J,J')\|_{\delta-{\rm Sas}}^2=\|J\|^2+\delta\|J'\|^2
$$
for all Jacobi field $J$.

\medskip
The Sasaki metric is the 1--Sasaki metric. In our calculations, we will fix a $\delta$--Sasaki
metric for $\delta$ small enough and denote it simply by $\|\cdot\|$. Hence 
\begin{equation*}
\frac{\|dg_t(J_+(0),J_+'(0))\|}{\|(J_+(0),J_+'(0))\|}=\sqrt{\dfrac{1+\delta u_+(g_tx)^2}{1+\delta u_+(x)^2}}\exp\left[\displaystyle\int_0^t u_+(g_sx)ds\right],
\end{equation*}
thus by the continuity of $u_+$ we get that 
\begin{equation}\label{eq:delta-Sasaki}
C_{\delta}^{-1} \exp\left[\displaystyle\int_0^t u_+(g_sx)ds\right]\leq \frac{\|dg_t(J_+(0),J_+'(0))\|}{\|(J_+(0),J_+'(0))\|}\leq C_{\delta} \exp\left[\displaystyle\int_0^t u_+(g_sx)ds\right]
\end{equation} 
where $\lim\limits_{\delta\to0} C_{\delta}=1$. Similar considerations apply to $J_-$.

\subsection{Surfaces of revolution}\label{sec:revolution}

Let $I\subset\R$ be a compact interval, and let $\xi:I\to\R$ be a positive $C^4$ function.
The {\em surface of revolution}
defined by $\xi$ around the $x$ axis is the surface $S$ with global chart
$\Xi:I\times[0,2\pi]\to \R^3$ given by
$\Xi(s,\theta)=(s,\xi(s)\cos\theta,\xi(s)\sin\theta)$.
We collect some known facts about these surfaces, see \cite{Manfredo-superficies}.

\medskip
\noindent
{\sc Curvature:} The curvature at $p=\Xi(s,\theta)$ is equal to
$$
K(p)=-\frac{\xi''(s)}{\xi(s)[1+(\xi'(s))^2]^2}\cdot
$$
See \cite[Example 4, p.\ 161]{Manfredo-superficies}.

\medskip
Geodesics on surfaces of revolution have a simple description.
They satisfy the so-called {\em Clairaut relation}, which allows to reduce the second order 
ordinary differential equation (ODE) defining the geodesic to a first order ODE.
Let $\gamma(t)=\Xi(s(t),\theta(t))$ be a geodesic, and let $\psi(t)\in\mathbb S^1$ be the angle that
the circle $s=s(t)$, more precisely its image under $\Xi$, makes with $\gamma$ at $\gamma(t)$, see Figure~\ref{fig:clairaut}.

\medskip
\noindent
{\sc Clairaut relation:} The value
\begin{equation} \label{eq:Clairaut}
c=\xi(s(t))\cos\psi(t)=\xi(s(t))^2 \theta'(t)
\end{equation}
is constant along $\gamma$.
We call $c$ the {\em Clairaut constant} of $\gamma$ and of all of its tangent vectors.

\begin{figure}[hbt!]
\centering
\def\svgwidth{10cm}
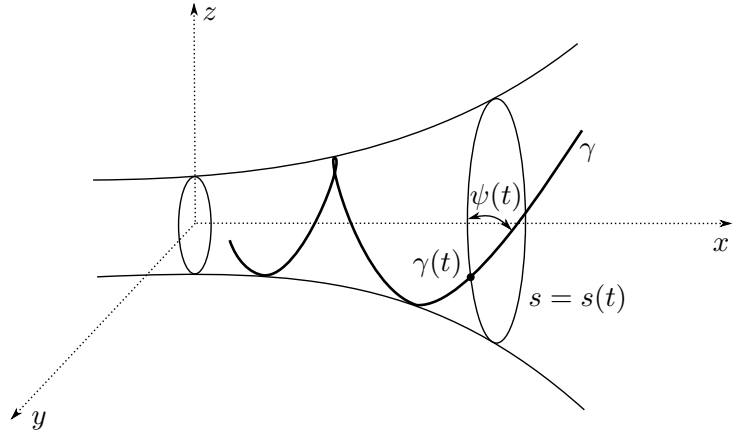
\caption{Clairaut relation: $c=\xi(s(t))\cos\psi(t)=\xi(s(t))^2 \theta'(t)$ is constant along $\gamma$.}
\label{fig:clairaut}
\end{figure}

\medskip
\noindent
{\sc Equation of geodesics:} If $\gamma(t)=\Xi(s(t),\theta(t))$ is a geodesic with Clairaut constant $c$, 
then $s(t)$ satisfies
\begin{equation}\label{eq:geodesics}
\left[1+\xi'(s)^2\right](s')^2+\frac{c^2}{\xi(s)^2}=1.
\end{equation}
See \cite[Example 5,  p. 255]{Manfredo-superficies} for the proof.

\medskip
For a fixed $s_0\in\R$, the curve $\Xi(s_0,\theta)$ is called a {\em meridian}.
The meridian $\Xi(s_0,\theta)$ is a geodesic if and only if $\xi'(s_0)=0$ 
(see \cite[p. 256]{Manfredo-superficies}).

Observe that $S$ is diffeomorphic to $I\times\mathbb S^1$ and $M$ is
diffeomorphic to $S\times \mathbb S^1\cong I\times\mathbb S^1\times\mathbb S^1$,
where $(p,\psi)\in S\times \mathbb S^1$ is identified to the unit tangent vector with basepoint $p$ that makes
an angle $\psi$ with the meridian passing though $p$. We can use this identification to define
another metric on $M$.

\medskip
\noindent
{\sc Clairaut coordinates and Clairaut metric:} The {\em Clairaut coordinates} on $M$ are 
$(s,\theta,\psi)\in I\times\mathbb S^1\times \mathbb S^1$, and the {\em Clairaut metric}
on $M$ is the Riemannian metric $\|\cdot\|_{\rm C}$ on $M$ given by the canonical product on
$I\times\mathbb S^1\times \mathbb S^1$.

\medskip
The Clairaut metric induces a distance, which we call the {\em Clairaut distance} and
denote by $d_{\rm C}$.
Above, the canonical metrics are the induced metrics of $I\subset\R$ and $\mathbb S^1\subset\R^2$.
Since $I$ is compact, the metrics $\|\cdot\|_{\delta-{\rm Sas}}$ and $\|\cdot\|_{\rm C}$ are equivalent.

The Clairaut relation (\ref{eq:Clairaut}) leads us to the following definition.

\medskip
\noindent
{\sc Clairaut function:} The {\em Clairaut function} is the function $c:M\to\R$ defined by
$c(s,\theta,\psi)=\xi(s)\cos\psi$.

\subsection{Surfaces with degenerate closed geodesic}\label{sec:our-surface}

We now define a class of surfaces that exhibit two special features:
the only region of zero curvature is a closed geodesic,
and on a neighborhood of this geodesic the surface is a particular surface of revolution.

\medskip
\noindent
{\sc Surface with degenerate closed geodesic:} A surface of nonpositive curvature $S$
is a {\em surface with degenerate closed geodesic $\gamma$} if there are $r\in [4,\infty)$ and $\ve_0>0$ such that:
\begin{enumerate}[$\circ$]
\item $S$ is $C^r$ with everywhere negative curvature except at a closed geodesic
$\gamma$.
\item There are two closed curves $\alpha,\beta$ defining a set $\mathcal N\subset S$
that contains $\gamma$ such that $\mathcal N$ is a surface of revolution with $\xi(s)=1+|s|^r$ for $|s|\leq \ve_0$.
Moreover, $\alpha=\Xi(-\ve_0,\theta)$, $\gamma=\Xi(0,\theta)$, $\beta=\Xi(\ve_0,\theta)$ are meridians.
We call $\mathcal N$ the {\em neck}, see Figure~\ref{fig:surface}. 
\end{enumerate}
\begin{figure}[hbt!]
\centering
\def\svgwidth{8cm}
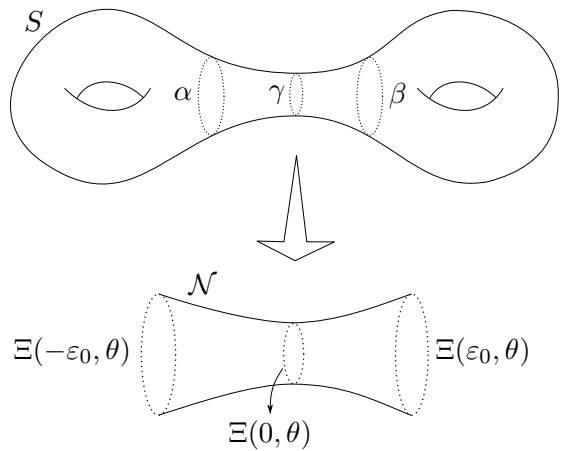
\caption{An example of a surface $S$ with degenerate closed geodesic.}\label{fig:surface}
\end{figure}

Such surfaces indeed exist, and can be obtained by interpolating the neck with 
a hyperbolic surface ($K\equiv -1$) with one cusp on each side. Since near the cusp a hyperbolic
surface is a surface of revolution, it is enough to interpolate its profile function with the function $\xi$,
in a way that the resulting function is strictly convex for $s\neq 0$. This can be made as in 
\cite[Appendix A.2]{Donnay}, using a partition of unity.

In the sequel, we fix a surface $S$ with degenerate closed geodesic $\gamma$.
Following the notation of Section \ref{sec:geodesic-flows},
let $M=T^1S$, $g_t:M\to M$ the geodesic flow on $S$, and $\mu$ the smooth
probability measure on $M$ induced by the Riemannian metric. 
Recall that the invariant manifolds $\wh W^{s/u}_x$
are $C^1$ for all $x\in M$. Gerber \& Wilkinson proved a stronger regularity \cite{Gerber-Wilkinson}
and also a property about the tangent distributions $\wh E^{s/u}$.

\begin{theorem}[Gerber \& Wilkinson \cite{Gerber-Wilkinson}]\label{thm:Gerber-Wilkinson}
Let $S$ be a surface with degenerate closed geodesic. 
Then the curves $\wh W^{s/u}_x$ are uniformly $C^{1+{\rm Lip}}$, and 
$x\mapsto \wh E^{s/u}_x$ is H\"older continuous. 
\end{theorem}

In other words, the $C^{1+{\rm Lip}}$ norms of all $\wh W^{s/u}_x$, $x\in M$,
are bounded by a uniform constant, and the tangent direction $\wh E^{s/u}_x$
is H\"older continuous as a function of $x\in M$. Actually, Gerber \& Wilkinson established
Theorem \ref{thm:Gerber-Wilkinson} in a context that does not cover surfaces with degenerate closed 
geodesic with non-integer $r$, but their proof can be easily adapted to prove the above theorem.
In the Appendix, we show how to make such changes.


In the Clairaut coordinates, let
$\gamma_0=\{0\}\times \mathbb S^1\times\{0\}$ and
$\gamma_\pi=\{0\}\times \mathbb S^1\times\{\pi\}$. We have ${\rm Deg}=\gamma_0\cup\gamma_\pi$
and so Theorem~\ref{thm:ergodicity} implies that the flow $g_t$ is ergodic.
Actually, $g_t$ is Bernoulli; see \cite{Pesin-Characteristic-1977} and
\cite[Thm.\ 12.2.13]{Barreira-Pesin-Non-Uniform-Hyperbolicity-Book} for the classical proofs,
and \cite{Ledrappier-Lima-Sarig} for a proof using symbolic dynamics.

Next, we
use the Clairaut function to distinguish some vectors of $M$ that will play a key role
in the next sections. The only meridian that is a geodesic is $\gamma=\Xi(0,\theta)$.
In $M$, this corresponds to the two geodesics $\gamma_0$ and
$\gamma_\pi$. The Clairaut constants are $c=1$ and $c=-1$ respectively. 
Let $x=(s,\theta,\psi)\in M$ with $s\neq 0$ such that $g_{[0,\ve]}(x)\subset [-|s|,|s|]\times \mathbb S^1\times\mathbb S^1$ for
some $\ve>0$, i.e. the geodesic starting at $x$ points towards $\gamma$.

\medskip
\noindent
{\sc Asymptotic, bouncing, crossing vectors and geodesics:} 
A vector $x\in M$ as above is called:
\begin{enumerate}[$\circ$]
\item {\em Asymptotic} if $c(x)=\pm 1$: the geodesic path $g_{[0,\infty)}(x)$ is
asymptotic to $\gamma$.
\item {\em Bouncing} if $|c(x)|>1$: there is $t>0$ such that $\psi(t)=0$ or $\pi$,
i.e.\ the geodesic path $g_{[0,t]}(x)$ spirals towards $\gamma$, $g_t(x)$
is tangent to a meridian and after that the geodesic path spirals away from $\gamma$.
In such cases, the geodesic does not reach $\gamma$. 
\item {\em Crossing} if $|c(x)|<1$: there is $t>0$ such that $s(t)=0$,
i.e.\ the geodesic path $g_{[0,t]}(x)$ spirals towards $\gamma$, $g_t(x)$ crosses $\gamma$ and
after that the geodesic path spirals away from $\gamma$.
\end{enumerate}
The corresponding geodesic with initial condition $x$ is called {\em asymptotic, bouncing, crossing} respectively.

\medskip
See Figure~\ref{fig:geodesics}.
The statements above are easily verified using the Clairaut relation~\eqref{eq:Clairaut}. For instance, if $c(x)>1$ then $s(t)$ never vanishes during the neck transition and $g_t(x)$ is bounded away from $\gamma$. It follows that $\psi(t_0)=0$ for some $t_0>0$, and $s(t_0)$ is uniquely determined by the equation $\xi(s(t_0))=c$. In particular, the value of $t_0$ is unique and we obtain bouncing as claimed.

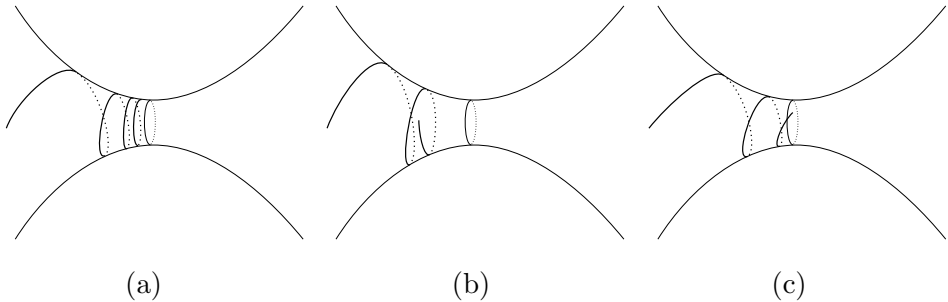
\begin{figure}[hbt!]
\centering
\def\svgwidth{12.5cm}
\begingroup%
  \makeatletter%
  \providecommand\color[2][]{%
    \errmessage{(Inkscape) Color is used for the text in Inkscape, but the package 'color.sty' is not loaded}%
    \renewcommand\color[2][]{}%
  }%
  \providecommand\transparent[1]{%
    \errmessage{(Inkscape) Transparency is used (non-zero) for the text in Inkscape, but the package 'transparent.sty' is not loaded}%
    \renewcommand\transparent[1]{}%
  }%
  \providecommand\rotatebox[2]{#2}%
  \newcommand*\fsize{\dimexpr\f@size pt\relax}%
  \newcommand*\lineheight[1]{\fontsize{\fsize}{#1\fsize}\selectfont}%
  \ifx\svgwidth\undefined%
    \setlength{\unitlength}{1104.99729974bp}%
    \ifx\svgscale\undefined%
      \relax%
    \else%
      \setlength{\unitlength}{\unitlength * \real{\svgscale}}%
    \fi%
  \else%
    \setlength{\unitlength}{\svgwidth}%
  \fi%
  \global\let\svgwidth\undefined%
  \global\let\svgscale\undefined%
  \makeatother%
  \begin{picture}(1,0.3114174)%
    \lineheight{1}%
    \setlength\tabcolsep{0pt}%
    \put(0,0){\includegraphics[width=\unitlength,page=1]{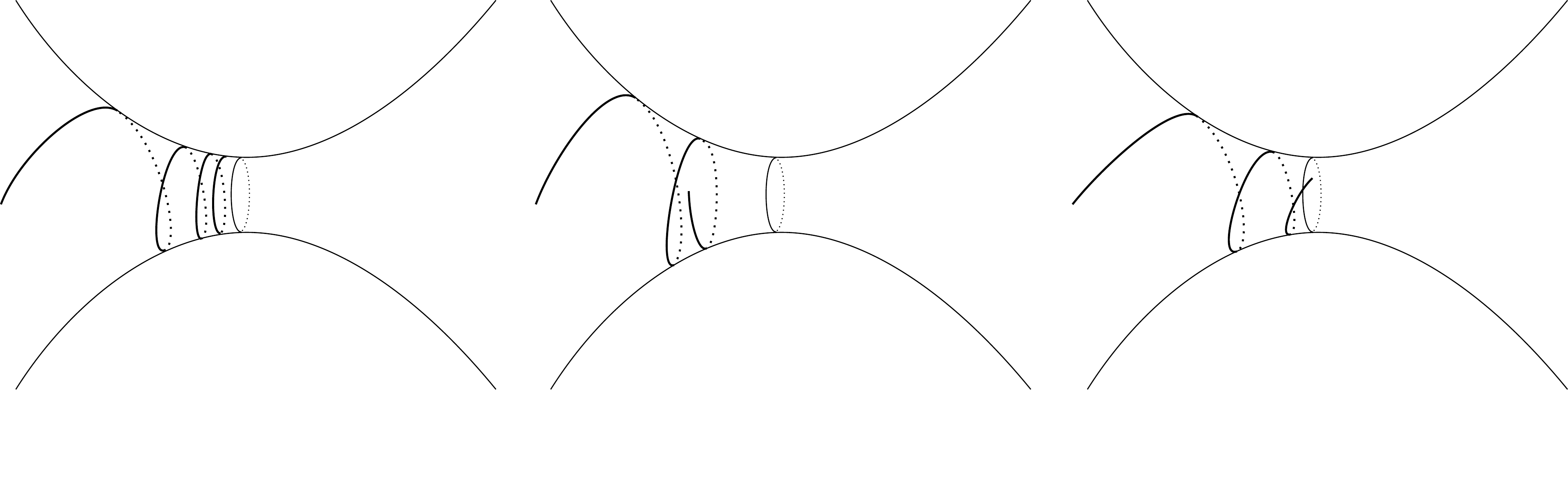}}%
    \put(0.12697338,0.00546168){\color[rgb]{0,0,0}\makebox(0,0)[lt]{\lineheight{1.25}\smash{\begin{tabular}[t]{l}(a)\end{tabular}}}}%
    \put(0.47445267,0.00546168){\color[rgb]{0,0,0}\makebox(0,0)[lt]{\lineheight{1.25}\smash{\begin{tabular}[t]{l}(b)\end{tabular}}}}%
    \put(0.81458876,0.00546168){\color[rgb]{0,0,0}\makebox(0,0)[lt]{\lineheight{1.25}\smash{\begin{tabular}[t]{l}(c)\end{tabular}}}}%
  \end{picture}%
\endgroup%

\caption{(a) Asymptotic vector. (b) Bouncing vector. (c) Crossing vector.}\label{fig:geodesics}
\end{figure}

We end this section making a comment on the number of closed geodesics.
Letting $P(T)$ denote the number of closed orbits of length $\leq T$,
Knieper \cite{Knieper-thesis} proved that 
\begin{equation} \label{eq:Knieper}
\lim\limits_{T\to\infty}\tfrac{1}{T}\log P(T)=h_{\rm top}(g_1).
\end{equation}

\section{Chernov axioms for exponential mixing}\label{sec:exponential.mixing}

Young introduced a powerful scheme \cite{Young-towers}, nowadays called {\em Young towers},
that implies exponential mixing for a vast class of dynamical systems, especially finite horizon dispersing billiards.
Based on some previous work \cite{Chernov-DCDC} and on Young's novel method,
Chernov introduced a set of axioms that implies exponential mixing, and applied it to numerous further
classes of planar dispersing billiards \cite{Chernov-1999}. These axioms are nowadays called
{\em Chernov axioms}. Using the ideas of Young and Chernov,
many authors proved exponential and polynomial decay of correlations for other classes of billiards
\cite{Young-polynomial,Markarian-polynomial,Chernov-Zhang,ChernovZhang05b,Balint-Toth}.

In this paper, we only require the existence of a Young tower together with its consequences; hence we omit the precise definition of Young tower and
instead refer the reader to \cite{Young-towers}.


We pay special attention to \cite{Chernov-Zhang}, where 
the presentation of Chernov axioms is more suitable to our context, as we now explain.
Firstly, they focus on two-dimensional
maps. Secondly, they give a simpler criterion on the axiom that is usually hardest to prove,
commonly called {\em growth of unstable manifolds}. The simpler criterion assumes, additionally to the
low dimension of the phase space, four facts:
\begin{enumerate}[$\circ$]
\item Alignment: $\mathfs S^+$ is tangent to stable cones and $\mathfs S^-$ is
tangent to unstable cones. Here, $\mathfs S^\pm$ are the singular sets, see
Section \ref{ss:Chernov-axioms-statements} below.
\item Structure of the singular set: control on the rate of accumulation of singularity curves.
\item Growth bound: control on the inverses of least expansions of smooth pieces of unstable
manifolds.
\item Complexity bound: control on the growth rate of self-intersections of {\em primary} singularities.
\end{enumerate}
These conditions are stated in axioms (A3) and (A8) below. Hence, in this work the Chernov axioms
consist of eight conditions for an abstract smooth hyperbolic map with singularities to have exponential
decay of correlations. Except for axioms (A3) and (A8), our presentation is based on
\cite[Appendix A]{Balint-Toth}, with one
crucial difference in axiom (A5): while the Chernov axioms require the invariant manifolds to be $C^2$ curves with
uniformly bounded curvature, we only require them to be $C^{1+{\rm Lip}}$ with uniformly bounded $C^{1+{\rm Lip}}$
norm. 

\subsection{Chernov axioms}\label{ss:Chernov-axioms-statements}

Here, we state the axioms (A1) to  (A8).

\medskip
\noindent
{\sc (A1) Dynamical system.} We consider $X_0$ an open subset of a $C^2$ Riemannian surface
$\widehat X$ such that its closure $X$ is compact. We let $\mathfs S^+,\mathfs S^-$
be closed subsets of $X$, and let 
$f:X_0\backslash \mathfs S^+\to X_0\backslash\mathfs S^-$ be a $C^2$ diffeomorphism.
We call $\mathfs S^+$ the {\em singular set} of $f$, and $\mathfs S^-$ the {\em singular set} of $f^{-1}$.
(Derivatives are allowed to blow up at the boundary of $X_0$ and at the singular sets.)

\medskip
For $n\ge 1$, define
\begin{align*}
\mathfs S_n&=\mathfs S^+\cup f^{-1}(\mathfs S^+)\cup\cdots\cup f^{-n+1}(\mathfs S^+)\\
\mathfs S_{-n}&=\mathfs S^-\cup f(\mathfs S^-)\cup\cdots\cup f^{n-1}(\mathfs S^-).
\end{align*}
Then $\mathfs S_n$ is the singular set of $f^n$, and $\mathfs S_{-n}$ is the singular set of $f^{-n}$.
Observe that $\mathfs S_1=\mathfs S^+$ and $\mathfs S_{-1}=\mathfs S^-$.

\medskip
\noindent
{\sc (A2) Uniform hyperbolicity.} There are two families of cones $\{C^u_x\},\{C^s_x\}$ in
the tangent planes $T_x\widehat X$, $x\in X$, called unstable and stable cones,
and a constant $\Lambda>1$ such that:
\begin{enumerate}[i \ \ \ \ ]
\item[(A2.1)] {\em Continuity}: $\{C^u_x\},\{C^s_x\}$ are continuous on $X$.
\item[(A2.2)] {\em Full hyperbolicity}: The axes of $\{C^u_x\},\{C^s_x\}$ are one-dimensional.
\item[(A2.3)] {\em Transversality:} $\min\limits_{x\in X}\angle (C^u_x,C^s_x)>0$. 
\item[(A2.4)] {\em Invariance}: $Df(C^u_x)\subset C^u_{fx}$ and $Df(C^s_x)\supset C^s_{fx}$
whenever $Df$ exists.
\item[(A2.5)] {\em Uniform hyperbolicity}: For all $v^u\in C^u_x$ and $v^s\in C^s_x$,
$$
\|Df(v^u)\|\geq\Lambda\|v^u\|
\quad \text{and} \quad
\|Df^{-1}(v^s)\|\geq\Lambda\|v^s\|.
$$
\end{enumerate}

\medskip
Let $x\in X_0$. For $x\not\in \bigcup_{n\geq 0}\mathfs S_{-n}$ let
$E^u_x=\bigcap_{n\geq 0}Df^n(C^u_{f^{-n}x})$, and for $x\not\in \bigcup_{n\geq 0}\mathfs S_n$
let $E^s_x=\bigcap_{n\geq 0}Df^{-n}(C^s_{f^nx})$. These are called the
{\em unstable} and {\em stable} subspaces, respectively. Axiom (A2) implies that
every $E^u_x,E^s_x$ is one-dimensional. Moreover,
if $x\not\in\bigcup_{n\in\Z}\mathfs S_n$ then $E^u_x\oplus E^s_x=T_x\widehat X$,
with $E^u_x$ being spanned by vectors with positive Lyapunov exponents and $E^s_x$ spanned
by vectors with negative Lyapunov exponents.

For the remaining axioms, we need to introduce some terminology.
Let $\rho,m$ be respectively the Riemannian metric and Lebesgue measure on $\widehat X$. Given 
a curve $W\subset \widehat X$, let $\rho_W,m_W$ be respectively the Riemannian metric and 
Lebesgue measure on $W$ induced by $\rho,m$. 

\medskip
\noindent
{\sc Local unstable manifold (LUM):} A {\em local unstable manifold} (LUM) is a curve $W\subset X$
such that:
\begin{enumerate}[i,]
\item[(i)] $f^{-n}$ is well-defined and smooth on $W$ for all $n\geq 0$.
\item[(ii)] $\rho(f^{-n}x,f^{-n}y)\to 0$ exponentially quickly as $n\to\infty$ for all $x,y\in W$.
\end{enumerate}

\medskip
We usually write $W^u_x$ to represent a LUM containing $x$.
The tangent space of $W^u_x$ at $x$ is $E^u_x$.
Similarly, we define the notion of local stable manifold (LSM)
and write $W^s_x$ to represent a LSM containing $x$.

Now let $W_1,W_2$ be sufficiently
small and close enough LUM's such that small LSM's intersect each of $W_1,W_2$ at most once.
let $W'_1=\{x\in W_1:W^s_x\cap W_2\neq\emptyset\}$, and let $H:W'_1\to W_2$ be the {\em holonomy map} obtained by sliding
along local stable manifolds, i.e.\ $H(x)$ is the unique intersection between $W^s_x$ and $W_2$.
Also, let $\Lambda(x)=|\det (Df\restriction_{E^u_x})|$ be the Jacobian of $f$ in the direction of $E^u_x$,
which is the factor of expansion on $W^u_x$ at $x$.\footnote{For simplicity, the original work of Chernov
also required an assumption called {\em nonbranching of unstable manifolds}, see \cite[p.\ 516]{Chernov-1999}.
As already remarked in \cite{Chernov-1999}, this assumption can be dropped \cite{vandenBedem}.}

%


\medskip
\noindent
{\sc (A3) Alignment.} The angle between $\mathfs S^+$ and LUM's is bounded
away from zero; the angle between $\mathfs S^-$ and LSM's is bounded
away from zero.

\medskip
\noindent
{\sc (A4) SRB measure.} The map $f$ preserves an ergodic volume measure $\mu$
such that a.e. $x\in X_0$ has a LUM $W^u_x$ and the conditional measure on $W^u_x$
induced by $\mu$ is absolutely continuous with respect to $m_{W^u_x}$. Furthermore,
$f^n$ is ergodic for all $n\geq 1$.

\medskip
\noindent
{\sc (A5) Uniformly bounded $C^{1+{\rm Lip}}$ norms.} The leaves $W^{u/s}_x$ are
uniformly $C^{1+{\rm Lip}}$.

\medskip
In other words, there exists a universal constant $K>0$ such that if $W^{u/s}_x$ is an LUM or LSM
then the graph representing $W^{u/s}_x$ locally at $x$ has $C^{1+{\rm Lip}}$ norm bounded
by $K$. Axiom (A5) is weaker than those required in the literature, and is discussed further 
at the end of this section, see Remark \ref{rmk:A5}.

\medskip
\noindent
{\sc (A6) Uniform distortion bounds.} There is a function $\psi:[0,\infty)\to[0,\infty)$ with $\lim_{x\to0}\psi(x)=0$
for which the following holds: if $W$ is a LUM, then for all $x,y$ belonging to the same connected component $V$
of $W\cap \mathfs S_{n-1}$,
$$
\log\left[\prod_{i=0}^{n-1}\dfrac{\Lambda(f^i x)}{\Lambda(f^i y)}\right]\leq \psi(\rho_{f^n(V)}(f^n x,f^n y)).
$$

\medskip
\noindent
{\sc (A7) Uniform absolute continuity.} There is a constant $C>0$ with the following property:
if $W_1,W_2$ are two sufficiently small and close enough LUM's,
then the holonomy map $H:W_1'\to W_2$ is absolutely continuous with respect to $m_{W_1},m_{W_2}$
and
$$
\frac{1}{C}\leq \frac{m_{W_2}(H[W_1'])}{m_{W_1}(W_1')} \leq C.
$$

\medskip
Now we proceed to the crucial axiom, which states that the expansion of the system
prevails over the fragmentation caused by the singularities.
Let us recall once more that condition (A8) below follows \cite{Chernov-Zhang}.
Indeed, we require the practical
scheme described in \cite[\S 6]{Chernov-Zhang}, for the following reasons:
\begin{enumerate}[$\circ$]
\item The singular set $\mathfs S_1=\mathfs S^+$ is usually decomposed into two components $\mathfs S_{\rm P}$ and
$\mathfs S_{\rm S}$. The set $\mathfs S_{\rm P}$ is made of intrinsic singularities, which
we call {\em primary}, and $\mathfs S_{\rm S}$ is made of artificial ones, which we call {\em secondary}.
The set $\mathfs S_{\rm S}$ is artificially added to guarantee bounded distortion,
near places where the derivative explodes.
\item Since the expansion factor $\Lambda$ might be close to 1, we often need to consider an iterate $f^n$
so that $\Lambda^n$ is large enough.
\end{enumerate}
Hence we assume that $\mathfs S_1=\mathfs S^+=\mathfs S_{\rm P}\cup \mathfs S_{\rm S}$,
and we let $\mathfs S_n=\mathfs S_{{\rm P},n}\cup \mathfs S_{{\rm S},n}$ be the corresponding 
decomposition of $\mathfs S_n$ for $n\geq 1$, where
\begin{align*}
\mathfs S_{{\rm P},n}&=\mathfs S_{\rm P}\cup f^{-1}(\mathfs S_{\rm P})\cup\cdots\cup f^{-n+1}(\mathfs S_{\rm P})\\
\mathfs S_{{\rm S},n}&=\mathfs S_{\rm S}\cup f^{-1}(\mathfs S_{\rm S})\cup\cdots\cup f^{-n+1}(\mathfs S_{\rm S}).
\end{align*}
 Since secondary singularities are accompanied
by strong hyperbolicity, the philosophy that ``expansion prevails over fragmentation" is guaranteed
if the expansion caused by $\mathfs S_{\rm S}$ prevails
over the fragmentation caused by $\mathfs S_{{\rm P},n}$. We only need 
to check this for some $n\geq 1$ for which  $\Lambda^n$ is large enough. The precise
conditions are the following.

\medskip
\noindent
{\sc (A8) Growth of unstable manifolds.}
\begin{enumerate}[i \ \ \ \ ]
\item[(A8.1)] {\em Structure of the singularity set}: There are constants $C,d>0$
such that if $W$ is a LUM, then $W\cap \mathfs S_1$ is at most countable. Furthermore,
$W\cap\mathfs S_1$ has at most one accumulation point $x_\infty$, and
if $\{x_n\}$ is the monotonic sequence in $W\cap\mathfs S_1$ coverging to $x_\infty$ then
$$
\rho(x_n,x_\infty)\leq \frac{C}{n^d}\, \quad\text{for all $n\geq 1$.}
$$
\item[(A8.2)] {\em Growth bound (assumption on secondary singularities)}:
$$
\theta_0:=\liminf_{\delta\to 0}\sup_{|W|<\delta}\sum_n \tfrac{1}{\Lambda_n}<1,
$$
where the supremum is taken over all LUM $W$ with $W\cap\mathfs S_{\rm P}=\emptyset$,
the connected components of $W\backslash \mathfs S_{\rm S}$ are $\{W_n\}$ and 
$\Lambda_n=\min\{\Lambda(x):x\in W_n\}$.
\item[(A8.3)] {\em Complexity bound (assumption on primary singularities)}: Let 
$$
K_{\rm P,n}:=\lim_{\delta\to 0}\sup_{|W|<\delta}K_{{\rm P},n}(W)
$$
where the supremum is taken over all LUM $W$ and $K_{{\rm P},n}(W)$ is the number
of connected components of $W\backslash \mathfs S_{{\rm P},n}$; then
$K_{{\rm P},n}<\min\{\theta_0^{-1},\Lambda\}^n$ for some $n\geq 1$.
\end{enumerate} 

We can now state the result of Chernov \& Zhang that will be important to us.

\begin{theorem}[Chernov \& Zhang \cite{Chernov-Zhang}]\label{thm:Chernov-Zhang}
If $f$ satisfies {\rm (A1)--(A8)}, then $f$ is modelled by a Young tower with exponential tails.
\end{theorem}

\noindent In particular, $f$ enjoys exponential decay of correlations by~\cite{Young-towers}.

\begin{remark}\label{rmk:A5}
 Previous work
requires invariant manifolds to have uniformly bounded sectional curvature in axiom (A5). This condition is solely
used to uniformly approximate pieces of invariant manifolds by hyperplanes.
More specifically, \cite[Estimate~(4.1)]{Chernov-1999}
states the existence of $C>0$ such that if $W$ is a $\delta$--LUM then $Z[W,W,0]\leq CZ[H,H,0]$, where
$H$ is the projection of $W$ to $T_xM$ for some (every) $x\in W$.
The function $Z[W,W,0]$ characterizes, in some sense, the ``size'' of $W$. The curvature
assumption is used to show that the ``size'' of $W$ is of the same order of its projection
to hyperplanes. The weaker assumption of uniformly bounded $C^{1+{\rm Lip}}$ norm ensures that
the jacobians of the projection map and of its inverse have uniformly bounded Lipschitz constants,
and this is enough to imply the above estimate.
In particular, Theorem~\ref{thm:Chernov-Zhang} remains valid under (A1)--(A8) with this slightly weakened version of (A5).
\end{remark}

\section{Dynamics of transitions in the neck $\mathcal N$}\label{sec:neck-dynamics}

We initiate the study of the geodesic flow on a surface with degenerate closed geodesic as defined in Section~\ref{sec:our-surface} with $r\ge4$ and $\ve_0>0$ fixed.
Recall that we are taking $\delta>0$ small enough and considering the $\delta$--Sasaki
metric $\|\cdot\|=\|\cdot\|_{\delta-{\rm Sas}}$.
By Proposition \ref{prop:u-continuous} and equation~(\ref{eq:delta-Sasaki}),
the loss of uniform hyperbolicity of $g_t$ occurs when a geodesic approaches ${\rm Deg}$, i.e.\ when it 
spends a large amount of time spiraling close to the degenerate closed 
geodesic~$\gamma$.
Hence in this section we focus on 
giving a detailed description of the dynamics of the transitions 
in the neck $\mathcal N$.

In Subsection~\ref{ss:Omega} we construct a two-dimensional section $\Omega$, the {\em transition section},
and an associated {\em transition map} $f_0$.
In Subsection~\ref{ss:explicit} we obtain an explicit formula for $f_0$.
It is this representation that allows us to avoid using the Riccati equation and its (not so precise)
estimates and, instead, to perform more accurate calculations.
In Subsection~\ref{ss:times} we estimate transition times and derivatives of $f_0$.

In the rest of the paper, we use the following notation:
\begin{enumerate}[$\circ$]
\item $a(u)\ll b(u)$ or $a(u)=O(b(u))$ if there is a constant $C$
such that $a(u)\le Cb(u)$ for all $u$ small enough.
\item $a(u)\approx b(u)$ if $a(u)\ll b(u)$ and $b(u)\ll a(u)$.
\item $a(u)\sim b(u)$ if $\lim\limits_{u\to 0}\tfrac{a(u)}{b(u)}=1$.
\end{enumerate}

The calculations in this section use the Clairaut coordinates 
$x=(s,\theta,\psi)$ on $T^1A\cong[-\ve_0,\ve_0]\times\mathbb S^1\times\mathbb S^1$.

\subsection{Transition section $\Omega$ and map $f_0$}\label{ss:Omega}

The transition section $\Omega$ we construct allows a very simple description of the transitions of  geodesics in the neck.
As detailed below, we define $\Omega=\Omega_+\cup\Omega_-\cup\Omega_0$ where:
\begin{enumerate}[$\circ$]
\item $\Omega_+\subset\{\pm\ve_0\}\times\mathbb S^1\times\mathbb S^1$ is a neighborhood of four families of geodesics entering the neck and asymptotic to $\gamma$;
\item $\Omega_- \subset\{\pm\ve_0\}\times\mathbb S^1\times\mathbb S^1$ 
is a neighborhood of four families of geodesics exiting the neck and asymptotic to $\gamma$ in backwards time;
\item Each of $\Omega_+$ and $\Omega_-$ is a disjoint union of four annular regions (diffeomorphic to $\mathbb S^1\times(-1,1)$);
\item $\Omega_0$ is defined in a neighborhood of $\gamma$ and
is a disjoint union of two open disks.
\end{enumerate}

To construct $\Omega_\pm$, we use the Clairaut function $c$ from Section~\ref{sec:revolution}.
Recall that $\xi(s)=1+|s|^r$.
Hence there is a unique $\psi_0\in(0,\frac{\pi}{2})$ such that $\xi(\ve_0)\cos\psi_0=\xi(-\ve_0)\cos\psi_0=1$.
Recalling the notation of asymptotic vector introduced in 
Section~\ref{sec:our-surface},
the vectors $(\pm\ve_0,\theta,\pm\psi_0)$, $\theta\in\mathbb S^1$, constitute four families of asymptotic
vectors with Clairaut constant $c=1$ and asymptotic to $\gamma$.
Similarly, $(\pm\ve_0,\theta,\pm(\pi-\psi_0))$, $\theta\in\mathbb S^1$, constitute four families of asymptotic
vectors with $c=-1$ and asymptotic to $\gamma$.
Of these, $(-\ve_0,\theta,\psi_0)$, $(\ve_0,\theta,-\psi_0)$,
$(-\ve_0,\theta,\pi-\psi_0)$, $(\ve,\theta,-(\pi-\psi_0))$ correspond to the four families of trajectories that enter
$\mathcal N$ and are asymptotic to $\gamma$ as $t\to\infty$.
The remaining four families of trajectories exit $\mathcal N$ and are asymptotic to $\gamma$ as $t\to-\infty$.

Focusing momentarily on $x=(-\ve_0,\theta,\psi_0)$, we define
$$
\Omega_1=\{x=(-\ve_0,\theta,\psi)\in T^1A:|c(x)-1|<\chi\}.
$$
Shrinking $\chi=\chi(\ve_0,r)>0$, we can ensure that
$\Omega_1= \{-\ve_0\}\times \mathbb S^1\times I$
where $I$ is an open interval containing $\psi_0$ with $\overline{I}\subset(0,\frac{\pi}{2})$.
Treating the other three families of entering asymptotic trajectories similarly, we obtain $\Omega_+$ as the union of four sets isomorphic to $\Omega_1$.

Similarly, the set $\Omega_-$, isomorphic to $\Omega_+$, is obtained by considering the four 
families of exiting asymptotic trajectories.

Finally, let
$$
\Omega_0=(-\chi,\chi)\times\{0\}\times\left((-\chi,\chi)\cup (\pi-\chi,\pi+\chi)\right)\subset T^1A.
$$

\medskip
\noindent
{\sc Transition section $\Omega$:} Define $\Omega=\Omega_-\cup\Omega_0\cup\Omega_+$.

\medskip
We now prove that $\Omega$ is transverse to the flow direction.
\begin{enumerate}[$\circ$]
\item The tangent space at every point of $\Omega_0$ is $\R\times \{0\}\times\R$. Since the 
flow directions at $(0,0,0)$ and $(0,\pi,0)$ are spanned by $(0,1,0)$ and $(0,-1,0)$ respectively,
transversality holds at $(0,0,0)$ and $(0,\pi,0)$.
Since $\chi>0$ is small enough, transversality holds at every point of $\Omega_0$.
\item The tangent space at every $x=(\pm\ve_0,\theta,\psi)\in\Omega_+$ is $\{0\}\times\R\times\R$.
If $(s(t),\theta(t),\psi(t))$ is the geodesic defined by $x$, then the flow direction at $x$ is $(s'(0),\theta'(0),\psi'(0))$.
Since $\psi(0)\neq 0$, we have $\xi(s(0))>c(x)$ and so the equation of geodesics
(\ref{eq:geodesics}) implies that
$s'(0)\neq 0$. Again for $\chi>0$ small enough, $\Omega_+$ is transverse to the flow direction at $x$.
\item Analogously, $\Omega_-$ is transverse to the flow direction.
\end{enumerate}

The transition section $\Omega$ captures all trajectories that approach ${\rm Deg}$. To better
understand them, consider the partition of
$\Omega_+=\Omega_+^{=}\cup\Omega_+^{>}\cup
\Omega_+^{<}$ induced by $c$:
\begin{align*}
&\Omega_+^{=}=\{x\in\Omega_+:|c(x)|=1\},\;\text{(asymptotic)}, \\
&\Omega_+^{>}=\{x\in\Omega_+:|c(x)|>1\},\;\text{(bouncing)}, \\
&\Omega_+^{<}=\{x\in\Omega_+:|c(x)|<1\},\;\text{(crossing)}.
\end{align*}
Similarly, we define $\Omega_-=\Omega_-^{=}\cup\Omega_-^{>}\cup
\Omega_-^{<}$.

We have the following transitions of segments of trajectories that enter $\mathcal N$
and approach ${\rm Deg}$.
\begin{enumerate}[$\circ$]
\item Starting at asymptotic vectors: $\Omega_+^{=}\to \Omega_0\to\Omega_0\to\cdots$,
and these trajectories get trapped in $\Omega_0$.
\item Starting at bouncing vectors: $\Omega_+^{>}\to \Omega_0\to\cdots\to\Omega_0\to \Omega_-^{>}$.
\item Starting at crossing vectors: $\Omega_+^{<}\to \Omega_0\to\cdots\to\Omega_0\to \Omega_-^{<}$.
\end{enumerate}
We have thus understood the transition in the neck of every $x\in \Omega_+$.
Next, we introduce the map that performs the transitions from $\Omega_+$ to $\Omega_-$.

\medskip
\noindent
{\sc Transition map $f_0$:} Let $f_0:\Omega_+\setminus \Omega_+^{=}\to \Omega_-\setminus\Omega_-^{=}$
be the map induced by the flow, i.e.
$f_0(x)=g_t(x)$ where $t>0$ is least such that $g_t(x)\in\Omega_-$.

\subsection{Formula for the transition map $f_0$}\label{ss:explicit}
In this subsection, we obtain an explicit formula for $f_0$. 

Denote geodesics in the neck by $\x=\x(t)=(s(t),\theta(t),\psi(t))$. As we have seen, they are characterized by the
equations 
\begin{align}\label{eq:Clairaut-relations-neck1}
& c=\xi(s)\cos\psi=\xi(s)^2 \theta', 
\\ & [1+\xi'(s)^2](s')^2+\frac{c^2}{\xi(s)^2}=1 ,
\label{eq:Clairaut-relations-neck2}
\end{align}
where $c=c(\x)$ is the Clairaut constant of $\x$. 
We parametrize bouncing geodesics~$\x$ taking
$s'(0)=0$ and $\psi(0)=0$ or $\pi$, i.e.\ $\x$ bounces back exactly at time $t=0$. Similarly,
we parametrize crossing geodesics $\x$ taking
$s(0)=0$, i.e.\ $\x$ crosses $\gamma$ exactly at time $t=0$.
The next result describes the symmetry properties of bouncing/crossing geodesics.

\begin{lemma}\label{lem:symmetry} 
The following are true.
\begin{enumerate}[{\rm (1)}]
\item A bouncing geodesic $(s(t),\theta(t),\psi(t))$ with $s'(0)=0$ and $\psi(0)=0$
satisfies  
$$
(s(-t),\theta(-t),\psi(-t))=(s(t),-\theta(t)+2\theta(0),-\psi(t))
$$
for all $t$.
\item A crossing geodesic $(s(t),\theta(t),\psi(t))$ with $s(0)=0$ satisfies  
$$
(s(-t),\theta(-t),\psi(-t))=(-s(t),-\theta(t)+2\theta(0),\psi(t))
$$
for all $t$. 
\end{enumerate}
\end{lemma} 

\begin{proof} (1) Define
$$
\overline\x=(\overline s(t),\overline\theta(t),\overline \psi(t))=
(s(-t),-\theta(-t)+2\theta(0),-\psi(-t)).
$$
We show that $\overline\x$ is a geodesic with the same initial conditions of $\x$.
Start by observing that $\overline\x(0)=(s(0),\theta(0),0)=\x(0)$.
Since $\overline s'(t)=-s'(-t)$, $\overline\theta'(t)=\theta'(-t)$
and $\overline\psi'(t)=\psi'(-t)$, 
\begin{align*}
& c(\overline\x)=\xi(\overline s(t))^2\overline\theta'(t)=\xi(s(-t))^2\theta'(-t)=c(\x), \\  
& [1+\xi'(\overline s(t))^2](\overline s'(t))^2+\tfrac{c(\overline \x)^2}{\xi(\overline s(t))^2}=
[1+\xi'(s(-t))^2](s'(-t))^2+\tfrac{c(\x)^2}{\xi(s(-t))^2}=1.
\end{align*}
Hence $\overline \x=\x$, which proves part (1).

\medskip
\noindent
(2) Similarly, define
$$
\overline\x=(\overline s(t),\overline\theta(t),\overline \psi(t))=
(-s(-t),-\theta(-t)+2\theta(0),\psi(-t)).
$$
Then $\overline\x(0)=(0,\theta(0),\psi(0))=\x(0)$ and,
since $\overline s'(t)=s'(-t)$, $\overline\theta'(t)=\theta'(-t)$
and $\overline\psi'(t)=-\psi'(-t)$, 
\begin{align*}
& c(\overline\x)=\xi(\overline s(t))^2\overline\theta'(t)=\xi(s(-t))^2\theta'(-t)=c(\x), \\
& [1+\xi'(\overline s(t))^2](\overline s'(t))^2+\tfrac{c(\overline \x)^2}{\xi(\overline s(t))^2}=
[1+\xi'(s(-t))^2](s'(-t))^2+\tfrac{c(\x)^2}{\xi(s(-t))^2}=1.
\end{align*}
We obtain again that $\overline \x=\x$, which proves part (2).
\end{proof}

Let $\x=(s(t),\theta(t),\psi(t))$ be a bouncing/crossing geodesic parametrized as above.

\medskip
\noindent
{\sc Transition time $\Upsilon_0$:} 
Define $\Upsilon_0(\x)=\min\{t>0: |s(t)|=\ve_0\}$.

\medskip
Since $\xi$ is an even function, we can assume that $\x$ enters the neck $\mathcal N$ at $\{s=-\ve_0\}$.
Lemma~\ref{lem:symmetry} implies the identities 
\[
x=g_{-\Upsilon_0(\x)}(s(0),\theta(0),\psi(0))\in \Omega_+, \qquad
f_0(x)=g_{\Upsilon_0(\x)}(s(0),\theta(0),\psi(0))\in \Omega_-. 
\]
In particular, the transition time of $\x$
from $\Omega_+$ to $\Omega_-$ is actually equal to $2\Upsilon_0(\x)$ (but it is convenient to call $\Upsilon_0$ the transition time). Moreover:
\begin{enumerate}[$\circ$]
\item If $\x$ is bouncing then $s\restriction_{[-\Upsilon_0(\x),0]}$ is strictly increasing and
$s\restriction_{[0,\Upsilon_0(\x)]}$ is strictly decreasing. 
We have $s(-\Upsilon_0(\x))=-\ve_0=s(\Upsilon_0(\x))$.
\item If $\x$ is crossing then $s\restriction_{[-\Upsilon_0(\x),\Upsilon_0(\x)]}$ is strictly increasing. 
We have $s(-\Upsilon_0(\x))=-\ve_0$, $s(\Upsilon_0(\x))=\ve_0$.
\end{enumerate}

The next proposition gives the remaining ingredient to obtain the explicit formula for $f_0$.

\begin{proposition}\label{prop:PW}
Let $\x$ be a bouncing/crossing geodesic. Then 
$$
|\theta(\Upsilon_0(\x))-\theta(-\Upsilon_0(\x))|=2\int_{|s(0)|}^{\ve_0} \frac{|c(\x)|}{\xi(s)}\left[\frac{1+\xi'(s)^2}{\xi(s)^2-c(\x)^2}\right]^{\frac{1}{2}}ds.
$$
\end{proposition}

\begin{proof}
Equations~\eqref{eq:Clairaut-relations-neck1} 
and~\eqref{eq:Clairaut-relations-neck2} can be rewritten as
$$
\theta'=c(\x)\xi(s)^{-2}\ \ \text{ and }\ \
|s'|=\xi(s)^{-1}\left[\frac{\xi(s)^2-c(\x)^2}{1+\xi'(s)^2}\right]^{\frac{1}{2}}.
$$
Dividing one equation by the other yields
$$
\left|\frac{d\theta}{ds}\right|=
\frac{|c(\x)|}{\xi(s)}\left[\frac{1+\xi'(s)^2}{\xi(s)^2-c(\x)^2}\right]^{\frac{1}{2}}.
$$
Since $s\restriction_{[0,\Upsilon_0(\x)]}$ is monotone,
$$
|\theta(\Upsilon_0(\x))-\theta(0)|=\int_{|s(0)|}^{\ve_0} \frac{|c(\x)|}{\xi(s)}\left[\frac{1+\xi'(s)^2}{\xi(s)^2-c(\x)^2}\right]^{\frac{1}{2}}ds.
$$
By Lemma~\ref{lem:symmetry}, we have that
$\theta(\Upsilon_0(\x))-\theta(-\Upsilon_0(\x))=2[\theta(\Upsilon_0(\x))-\theta(0)]$ and the proof is complete.
\end{proof}

By the rotational symmetry, the integral in Proposition~\ref{prop:PW} only depends on $\psi$.
Hence we define
\begin{align*}
\zeta(\psi) & =2\int_{|s(0)|}^{\ve_0} \frac{|c(\x)|}{\xi(s)}\left[\frac{1+\xi'(s)^2}{\xi(s)^2-c(\x)^2}\right]^{\frac{1}{2}}ds. 
\end{align*}
With this notation, 
\[
f_0(-\ve_0,\theta,\psi)=\begin{cases} (-\ve_0,\theta\pm \zeta(\psi),-\psi)
& \text{for bouncing vectors,} \\
(\ve_0,\theta\pm \zeta(\psi),\psi)
& \text{for crossing vectors.} \end{cases}
\]

\subsection{Transition times and derivatives of $f_0$}
\label{ss:times}

Before proceeding further, we derive two elementary integral estimates.

\begin{proposition}\label{prop:int} {\color{white}a}
\begin{enumerate}[{\rm (1)}]
\item Let $r,\alpha,\ve>0$ with $\alpha r>1$.  Then
\[
\int_0^\ve (s^r+b)^{-\alpha}\,ds
\sim C_1b^{-\alpha+\frac{1}{r}} \quad\text{as $b\to0^+$,}
\]
where $C_1=\int_0^\infty (x^r+1)^{-\alpha}\,dx$.
\item
Let $q,r,\alpha,\ve>0$ and $\beta\ge0$ with $\alpha r-\beta q>1$.  Then
\[
\int_b^\ve (s^r-b^r)^{-\alpha}(s^q-b^q)^\beta\,ds
\sim C_2b^{\beta q-\alpha r+1}
\quad\text{as $b\to0^+$,}
\]
and
\[
\int_b^{\ve + b} (s^r-b^r)^{-\alpha}(s^q-b^q)^\beta\,ds
\sim C_2b^{\beta q-\alpha r+1}
\quad\text{as $b\to0^+$,}
\]
where $C_2=\int_1^\infty (x^r-1)^{-\alpha}(x^q-1)^\beta\,dx$.
\end{enumerate}
\end{proposition}

\begin{proof}
Note that $\alpha r>1$ implies that $C_1$ is finite.
By direct computation and the change of variables $x=b^{-\frac{1}{r}}s$,
\[
\int_0^\ve (s^r+b)^{-\alpha}\,ds=
b^{-\alpha}\int_0^\ve (b^{-1}s^r+1)^{-\alpha}\,ds
=b^{{-\alpha}+\frac{1}{r}}\int_0^{\ve b^{-\frac{1}{r}}} (x^r+1)^{-\alpha}\,dx
\]
and part~(1) follows.

Proceeding similarly to the proof of part~(1),
\begin{align*}
\int_b^\ve (s^r-b^r)^{-\alpha}(s^q-b^q)^\beta ds & =
b^{\beta q-\alpha r}\int_b^\ve \left[\left(\tfrac{s}{b}\right)^r-1\right]^{-\alpha}\left[\left(\tfrac{s}{b}\right)^q-1\right]^\beta ds
\\ & =b^{\beta q-\alpha r+1}\int_1^{\ve b^{-1}} (x^r-1)^{-\alpha}(x^q-1)^\beta dx
\end{align*}
and the first estimate in part~(2) follows. The argument for the second estimate is identical.
\end{proof}

To better analyse $f_0$ near the set of asymptotic vectors,
we introduce a partition of $\Omega_+$, as follows.

\medskip
\noindent
{\sc Homogeneity bands on $\Omega_+$:} For each $n\geq 1$, the {\em homogeneity band} with 
index~$n$ is $\mathfs C_n=\mathfs{C}_n^{>}\cup\mathfs{C}_n^{<}$ where 
\begin{align*}
\mathfs{C}_{n}^{>}&=\left\{x\in \Omega_+: 1+\tfrac{1}{(n+1)^2}<|c(x)|< 1+\tfrac{1}{n^2}\right\}\\
\mathfs{C}_n^{<}&=\left\{x\in \Omega_+: 1-\tfrac{1}{n^2}< |c(x)|<1-\tfrac{1}{(n+1)^2}\right\}.
\end{align*} 

\medskip

The next result estimates $\Upsilon_0$ in the homogeneity bands.

\begin{lemma}\label{lem:t1-new}
Let $\x$ be a geodesic with entry vector in $\mathfs C_n$. Then
$\Upsilon_0(\x)\approx n^{\frac{r-2}{r}}$.  
\end{lemma}

\begin{proof}
We continue to assume without loss that $s(-\Upsilon_0(\x))=-\ve_0$.
Also, we suppose without loss that 
the Clairaut constant $c=c(\x)$ is positive.
By assumption, 
$ \tfrac{1}{(n+1)^2}<|c-1|< \tfrac{1}{n^2}$.
By~\eqref{eq:Clairaut-relations-neck2},
\begin{equation} \label{eq:s'approx}
(s')^2  = (1+(\xi')^2)^{-1}\xi^{-2}(\xi+c)(\xi-c)\approx \xi(s)-c.
\end{equation}
We have two cases:
\begin{enumerate}[$\circ$]
\item If $\x$ is bouncing, then $c=1+|s(0)|^r$ where $s(0)\sim -n^{-2/r}$. By~\eqref{eq:s'approx}, 
$(s')^2 \approx |s|^r- |s(0)|^r$ and so 
$(|s|^r-|s(0)|^r)^{-\frac{1}{2}}s'\approx 1$ on the interval $[-\Upsilon_0(\x),0]$. 
Applying Proposition~\ref{prop:int}(2) with $\alpha=\frac12$, $\beta=0$ and $b=|s(0)|$,
we conclude that 
\begin{align*}
&\, \Upsilon_0(\x)\approx-\int_{-\ve_0}^{s(0)}(|s|^r-|s(0)|^r)^{-\frac{1}{2}}ds
 =\int_{|s(0)|}^{\ve_0}(s^r-|s(0)|^r)^{-\frac{1}{2}}ds
\\ & \approx |s(0)|^{-\frac{r}{2}+1}\sim n^{\frac{r-2}{r}}.
\end{align*}
\item If $\x$ is crossing, then $c=\cos\psi(0)\sim 1-n^{-2}$. By~\eqref{eq:s'approx},
$(s')^2 \approx |s|^r+n^{-2}$ and so 
$(|s|^r+n^{-2})^{-\frac{1}{2}}s'\approx 1$ on the interval $[-\Upsilon_0(\x),\Upsilon_0(\x)]$. 
Applying Proposition~\ref{prop:int}(1) with $\alpha=\frac12$ and $b=n^{-2}$,
we conclude that 
$$\Upsilon_0(\x)\approx\int_0^{\ve_0}(s^r+n^{-2})^{-\frac{1}{2}}ds
\approx \left(n^{-2}\right)^{-\frac{1}{2}+\frac{1}{r}}=n^{\frac{r-2}{r}}.
$$
\end{enumerate}
This concludes the proof of the lemma.
\end{proof}

Now we estimate, in terms of $c$, how 
the $\psi$--coordinate varies under $f_0$.
Without loss, we restrict to positive values of $c$.

\begin{lemma}\label{lem:diffpsi}
Suppose that $\x,\overline{\x}$ are both bouncing or both crossing geodesics. Then
$$
|\psi(\Upsilon_0(\x))-\overline{\psi}(\Upsilon_0(\overline{\x}))|\approx |c(\x)-c(\overline{\x})|.
$$
In particular, if the entry vectors of $\x,\overline{\x}$ are both in the same connected component of
$\mathfs C_n^<$ or of $\mathfs C_n^>$ then
$$
|\psi(\Upsilon_0(\x))-\overline{\psi}(\Upsilon_0(\overline{\x}))|=2(\ve_0^{2r}+2\ve_0^r)^{-1/2} n^{-3}
+O(n^{-4}).
$$
\end{lemma}

\begin{proof}
Let $a=(1+\ve_0^r)^{-1}$. By~\eqref{eq:Clairaut-relations-neck1}, $(1+\ve_0^r)\cos\psi(\Upsilon_0(\x))=c(\x)$
and so $\cos\psi(\Upsilon_0(\x))=ac(\x)$. Similarly,
$\cos\overline\psi(\Upsilon_0(\overline{\x}))=ac(\overline{\x})$.
By the mean value theorem,
\begin{align*}
a|c(\x)-c(\overline{\x})|&=
|\cos\psi(\Upsilon_0(\x))-\cos\overline{\psi}(\Upsilon_0(\overline{\x}))|\\
&=|\sin\psi^*|\cdot |\psi(\Upsilon_0(\x))-\overline\psi(\Upsilon_0(\overline{\x}))|
\end{align*}
for some $\psi^*$ between $\psi(\Upsilon_0(\x))$ and $\overline{\psi}(\Upsilon_0(\overline{\x}))$.
Since $\cos\psi(\Upsilon_0(\x)),\cos\overline{\psi}(\Upsilon_0(\overline{\x}))\sim a$, we have
$|\sin\psi^*|\sim (1-a^2)^{1/2}$ and so
$$
|\psi(\Upsilon_0(\x))-\overline\psi(\Upsilon_0(\overline{\x}))|=a|\sin\psi^*|^{-1}|c(\x)-c(\overline{\x})|\sim a(1-a^2)^{-1/2}
|c(\x)-c(\overline{\x})|.
$$

Now suppose that the entry vectors of $\x,\overline{\x}$ are in the same connected component
of $\mathfs C_n^{>}$. Assuming without loss of generality that $c(\x),c(\overline{\x})>0$, then
$1+(n+1)^{-2}<c(\x),c(\overline{\x})\leq 1+n^{-2}$.
It follows that 
$|c(\x)-c(\overline{\x})|= 2n^{-3}+O(n^{-4})$. 
Also, $\sin\psi^*=(1-a^2)^{-1/2}+O(n^{-2})$.
Hence
$$|\psi(\Upsilon_0(\x))-\overline{\psi}(\Upsilon_0(\overline{\x}))|= 2a(1-a^2)^{-1/2} n^{-3}+O(n^{-4}).
$$
An analogous calculation
holds if the entry vectors are in $\mathfs C_n^<$.
\end{proof}


To conclude this section, we obtain estimates for the derivatives of $\zeta$.

\begin{lemma}\label{lem:pw}
The following are true.
\begin{enumerate}[{\rm (1)}]
\item If $(-\ve_0,\theta,\psi)\in \mathfs C_n^<$ then
$\zeta'(\psi)\approx -n^{3-\frac{2}{r}}$ and $\zeta''(\psi)\approx n^{5-\frac{2}{r}}$.    
\item If $(-\ve_0,\theta,\psi)\in \mathfs C_n^>$ then
$\zeta'(\psi)\approx n^{3-\frac{2}{r}}$ and $|\zeta''(\psi)| \ll n^{5-\frac{2}{r}}$.
\end{enumerate}
\end{lemma}

Note that the first three estimates in Lemma~\ref{lem:pw} give upper and lower bounds, while the fourth estimate gives only an upper bound. 

\begin{proof}
(1) Let $a=1+\ve_0^r$ (this notation is different from Lemma \ref{lem:diffpsi}).
We have $c(\x)=a\cos\psi$, hence $\cos\psi\sim a^{-1}\approx 1$ and
$\sin\psi\sim (1-a^{-2})^{\frac{1}{2}}\approx 1$. Also, $s(0)=0$, so
\begin{align*}
\zeta(\psi) & =
2a \cos\psi\int_0^{\ve_0}  A(s)\left[\xi(s)^2-a^2\cos^2\psi\right]^{-\frac{1}{2}}ds
\end{align*}
where $A(s)=\frac{\left[1+\xi'(s)^2\right]^{\frac{1}{2}}}{\xi(s)}\approx 1$. By direct calculation, 
\begin{align*}
\zeta'(\psi)&=-2a\sin\psi\int_0^{\ve_0} A(s)\xi(s)^2\left[\xi(s)^2-a^2\cos^2\psi\right]^{-\frac{3}{2}}ds\\
&\approx  -\int_0^{\ve_0} \left[\xi(s)-c(\x)\right]^{-\frac{3}{2}}ds
\end{align*}
and 
\begin{align*}
\zeta''(\psi) & =2a\cos\psi\int_0^{\ve_0} A(s)\xi(s)^2\left[3a^2\sin^2\psi+a^2\cos^2\psi-\xi(s)^2\right]
\left[\xi(s)^2-a^2\cos^2\psi\right]^{-\frac{5}{2}}ds\\
 &\approx  \int_0^{\ve_0}\left[\xi(s)-c(\x)\right]^{-\frac{5}{2}}ds.
\end{align*}
Noting that $\xi(s)-c(\x)\approx s^r+n^{-2}$, Proposition~\ref{prop:int}(1) applied with:
\begin{enumerate}[$\circ$]
\item $\alpha=\frac32$ and $b=n^{-2}$ gives that $\zeta'(\psi)\approx -(n^{-2})^{-\frac32+\frac{1}{r}}=
-n^{3-\frac{2}{r}}$.
\item $\alpha=\frac52$ and $b=n^{-2}$ gives that $\zeta''(\psi)\approx (n^{-2})^{-\frac52+\frac{1}{r}}=n^{5-\frac{2}{r}}$.
\end{enumerate}
This proves part (1). 

\vspace{1ex}
\noindent
(2) This part is more difficult, for two reasons: the interval of integration in $\zeta$ is not fixed,
and the denominator $\xi(s)^2-c(\x)^2$ is harder to control since both $\xi(s),c(\x)>1$.
Since $s'(0)=0$, we have $c(\x)=a\cos\psi=1+|s(0)|^r$ with, as before, $\cos\psi\approx 1$ and $\sin\psi\approx 1$.
Introduce the variable $y=|s(0)|$. Then $c(\x)-1=y^r$ and so $y\sim n^{-\frac{2}{r}}$.
Write $\zeta(\psi)=2I(y)$, where
\begin{align*}
&\, I(y)=\int_y^{\ve_0} \frac{(1+y^r)}{\xi(s)}\left[\frac{1+\xi'(s)^2}{\xi(s)^2-(1+y^r)^2}\right]^{\frac12}ds\\
&=\int_0^{\ve_0-y} \frac{(1+y^r)}{\xi(s+y)}\left[\frac{1+\xi'(s+y)^2}{\xi(s+y)^2-(1+y^r)^2}\right]^{\frac12}ds\\
&=\int_0^{\ve_0-y}A(s,y)\left[(s+y)^r-y^r\right]^{-\frac12}ds.
\end{align*}
Here, $A(s,y)=\tfrac{(1+y^r) }{\xi(s+y)}\left[\tfrac{1+\xi'(s+y)^2}{\xi(s+y)+(1+y^r)}\right]^{\frac12}$
is $C^2$ with $A(s,y)\approx 1$.

\vspace{1ex}
Next, write $I(y)=I_1(y)+I_2(y)$ where
\begin{align*}
&I_1(y)=\int_0^{\ve_0}A(s,y)\left[(s+y)^r-y^r\right]^{-\frac12}ds,\\
&I_2(y)  =-\int_{\ve_0-y}^{\ve_0}A(s,y)\left[(s+y)^r-y^r\right]^{-\frac12}ds.
\end{align*}
Now, $I_2$ is $C^2$ for $y$ small
and in particular $I_2'$ and $I_2''$ are bounded. Therefore, it remains to estimate $I_1'$ and $I_1''$.
We have $I_1'=Q_1+Q_2$ where
\begin{align*}
Q_1(y)&=\int_0^{\ve_0} \partial_y A(s,y)\left[(s+y)^r-y^r\right]^{-\frac12}ds,\\ 
Q_2(y)&= -\tfrac{r}{2}\int_0^{\ve_0} A(s,y)\left[(s+y)^r-y^r\right]^{-\frac32}\left[(s+y)^{r-1}-y^{r-1}\right]ds.
\end{align*}
Using that $A$ is $C^2$ and applying Proposition~\ref{prop:int}(2)
with $\alpha=\frac12$, $\beta=0$, $b=y$ gives that
$$
|Q_1(y)|\ll \int_0^{\ve_0}\left[(s+y)^r-y^r\right]^{-\frac12}ds
 =\int_y^{\ve_0+y} \left[s^r-y^r\right]^{-\frac12}ds\approx y^{-\frac{r}{2}+1}.
$$
Now, since $A(s,y)\approx 1$, applying Proposition~\ref{prop:int}(2)
with $\alpha=\frac32$, $\beta=1$, $q=r-1$, $b=y$ implies that
\begin{align*}
Q_2(y) &\approx -\int_0^{\ve_0}\left[(s+y)^r-y^r\right]^{-\frac32}\left[(s+y)^{r-1}-y^{r-1}\right]ds\\
&=-\int_y^{\ve_0+y}\left[s^r-y^r\right]^{-\frac32}\left[s^{r-1}-y^{r-1}\right]ds
\approx -y^{-\frac{r}{2}}.
\end{align*}
Hence,
$I'(y)\approx I_1'(y)\approx -y^{-\frac{r}{2}}$. Next we transform back to
the variable $\psi$. Differentiating $1+y^r=a\cos\psi$ with respect to $\psi$, we get that 
$ry^{r-1}\frac{dy}{d\psi}=-a\sin\psi\approx -1$ and so $\frac{dy}{d\psi}\approx -y^{1-r}$,
which implies that 
$$
\zeta'(\psi)=2I'(y)\tfrac{dy}{d\psi}\approx (-y^{-\frac{r}{2}})(-y^{1-r})=y^{1-\frac{3r}{2}}
\sim (n^{-\frac{2}{r}})^{1-\frac{3r}{2}}\sim n^{3-\frac{2}{r}}.
$$
This is the desired estimate for $\zeta'$.

Similarly, we can write $I_1''=Q_3+Q_4+Q_5+Q_6$ with
\begin{align*}
|Q_3(y)|&\ll \int_y^{\ve_0+y}\left[s^r-y^r\right]^{-\frac12}ds\approx y^{-\frac{r}{2}+1},\\
|Q_4(y)|&\ll \int_y^{\ve_0+y}\left[s^r-y^r\right]^{-\frac32}\left[s^{r-1}-y^{r-1}\right]ds\approx y^{-\frac{r}{2}}, \\
|Q_5(y)|&\approx \int_y^{\ve_0+y}\left[s^r-y^r\right]^{-\frac32}\left[s^{r-2}-y^{r-2}\right] ds\approx y^{-\frac{r}{2}-1},\\ 
|Q_6(x)|&\approx \int_y^{\ve_0+y}\left[s^r-y^r\right]^{-\frac52}\left[s^{r-1}-y^{r-1}\right]^2ds \approx y^{-\frac{r}{2}-1}.
\end{align*}
Here, the estimates of $Q_3,Q_4$ are the same as those of $Q_1,Q_2$, the estimate of
$Q_5$ follows from Proposition~\ref{prop:int}(2) with $\alpha=\frac32$, $\beta=1$, $q=r-2$, $b=y$
and the estimate of $Q_6$ follows from Proposition~\ref{prop:int}(2) with $\alpha=\frac52$, $\beta=2$,
$q=r-1$, $b=y$. Hence $|I''(y)|\ll y^{-\frac{r}{2}-1}$.
Differentiating $ry^{r-1}\frac{dy}{d\psi}=-a\sin\psi$ with respect to $\psi$ and recalling that
$\frac{dy}{d\psi}\approx -y^{1-r}$, we get that
\[
-1\approx -a\cos\psi=
r(r-1)y^{r-2}\big(\tfrac{dy}{d\psi}\big)^2+ry^{r-1}\tfrac{d^2y}{d\psi^2}
\approx r(r-1)y^{-r}+ry^{r-1}\tfrac{d^2y}{d\psi^2}.
\]
Since $y^{-r}$ is large, both terms on the right-hand side have the same order,
hence $\tfrac{d^2y}{d\psi^2}\approx -y^{-2r+1}$. We thus conclude that
\begin{align*}
|\zeta''(\psi)|  =\big|2I''(y)\big(\tfrac{dy}{d\psi}\big)^2+2I'(y)\tfrac{d^2y}{d\psi^2}\big|
& \ll y^{-\frac{r}{2}-1}y^{2(1-r)}+y^{-\frac{r}{2}}y^{-2r+1}\\
&\approx y^{-\frac{5r}{2}+1}\sim (n^{-\frac{2}{r}})^{-\frac{5r}{2}+1}=n^{5-\frac{2}{r}}
\end{align*}
which is the required estimate for $\zeta''$.
\end{proof}

\section{Poincar\'e section and first return map $f$}\label{sec:construction-Poincare-section}

In this section, we construct a suitable
Poincar\'e first return map $f:\Sigma_0\to\Sigma_0$ with
unbounded Poincar\'e return time $\tau:\Sigma_0\to(0,\infty]$ such that $\inf\tau>0$.  
Keeping in mind the Chernov axioms (Section~\ref{sec:exponential.mixing}), we require that 
the two-dimensional cross-section $\Sigma_0$ satisfies:
\begin{enumerate}[$\circ$]
\item $\Sigma_0$ is the disjoint union of finitely many codimension one submanifolds of $M$
each of which is almost orthogonal to the flow direction and hence, by Lemma \ref{lem:subspaces}(2),
almost parallel to $\wh E^s\oplus\wh E^u$. 
In particular, the flow projections of the one-dimensional stable/unstable directions $\wh E^{s,u}$ of
$g_t$ define stable/unstable directions $E^{s,u}$ for $f$. 
Moreover, 
the hyperbolicity of $f$ along $E^{s,u}$ is almost the same as that of $g_t$
along $\wh E^{s,u}$, as given by equation \eqref{eq:delta-Sasaki}. 
\item The boundary of $\Sigma_0$ is transverse to $E^s$ and $E^u$. This condition ensures (A3).
\item There are no triple intersections for a certain family of curves and iterates under~$f$.
This family of curves is the union of boundaries of $\Sigma_0$ and finitely
many curves of asymptotic vectors, and makes up the primary singular set
$\mathfs S_{\rm P}$. This condition is used in the verification of (A8.3).
\end{enumerate}
The construction of $\Sigma_0$ is rather technical, and is done as follows.
Recall the transition section $\Omega$ constructed in Section \ref{ss:Omega}.
Using $\Omega$, we construct a ``security'' section $\wt\Sigma$
that is almost parallel to $\wh E^s\oplus\wh E^u$. Then we choose $\wh\Sigma\subset\wt\Sigma$
so that its boundary is transverse
to the invariant directions and the ``no triple intersections'' requirement stated above holds.
Ideally, we would take $\Sigma_0=\wh\Sigma$, but
unfortunately this is not enough to guarantee axiom (A2), since the hyperbolicity rate of the
Poincar\'e return map (and its induced maps) depends on two effects that compete one against the other:
\begin{enumerate}[$\circ$]
\item The angle between the Poincar\'e section and $\wh E^{s,u}$: the smaller the angle is, the closer are 
the hyperbolicity rates of the flow and of the Poincar\'e return map.
\item The Poincar\'e return time of the section: the smaller the return time is,
the weaker is the hyperbolicity of the flow (and hence of the Poincar\'e return map).
\end{enumerate}
If the connected components of $\wh\Sigma$ are small to guarantee a small angle,
then the Poincar\'e return time is close to zero.
To bypass this difficulty, we divide each connected component of
$\wh\Sigma$ into small pieces, each of them still with boundary transverse to the invariant directions,
and displace each of them by a small amount in the flow direction so that the new angle with
$\wh E^{s,u}$ is as close to zero as we wish and no triple intersection appears.
Letting $\Sigma_0$ be the union of the displaced pieces,
its Poincar\'e return time is close to that of $\wh\Sigma$. In other words,
this division procedure decreases the angle whilst almost preserving the Poincar\'e return time.
This allows us to prove (A2) in Section \ref{sec:CZ-scheme}.

\subsection{Construction of $\wt \Sigma$}\label{ss:Sigma-tilda}

Since $\Omega_\pm$ might not be almost parallel to $\wh E^{s,u}$, we take instead a finite union of small discs
that are almost parallel to  $\wh E^{s,u}$ and whose union of flow boxes contains $\Omega_\pm$.
Choosing the discs small enough, the transitions in the neck from $\wt\Sigma$
to itself coincide with the transitions from $\Omega_+$
to $\Omega_-$ up to small flow displacements at the beginning and end.
Then we complete $\wt\Sigma$ by adding a disjoint union of finitely many small discs
such that their flow boxes do not intersect the trajectories in the neck that are close to the asymptotic ones.
In other words, we complete the section so that the asymptotic trajectories and nearby ones
remain the same.

We begin introducing some notation.
For each $x\in M$, let $\exp_x:T_xM\to M$ denote the exponential map of $M$ at $x$.
For $U\subset M$ and $I\subset \R$, we let $g_IU=\bigcup_{t\in I}g_t(U)$.
Recall that $Z_x$ is the one-dimensional subspace of $T_xM$ tangent to the geodesic flow.

\medskip
\noindent
{\sc $su$--disc and flow box:} The {\em $su$--disc} at $x\in M$ with radius $\lambda>0$ is the surface
$$
D_\lambda(x)=\{\exp_x(v):v\perp Z_x\text{ and }\|v\|\leq \lambda\}.
$$
The {\em flow box} at $x\in M$ with radius $\lambda>0$ is $g_{[-\lambda,\lambda]}D_\lambda(x)$.

\medskip
Fix $\chi=\chi(r,\ve_0)>0$ small.
We take $\lambda<\chi$ so that $D_\lambda(x)$ is an immersed surface 
and $T_yD_\lambda(x)$ is almost orthogonal to the flow and hence almost
parallel to $\wh E^s_y\oplus \wh E^u_y$ for every $y\in D_\lambda(x)$. 
Now we proceed to construct $\wt\Sigma$.

\medskip
\noindent
{\sc Step 1 (Construction of $\wt\Sigma$ near $\Omega_\pm$):} Choose points
$x_1,\ldots,x_m\in M$ and $0<a<b<\chi$ such that $\{D_b(x_i)\}_{1\leq i\leq m}$
are pairwise disjoint and
\begin{align}\label{eq:step1}
\Omega_+\cup\Omega_-\subset \bigcup_{1\leq i\leq m}g_{[-a,a]}D_a(x_i)\subset \bigcup_{1\leq i\leq m}g_{[-b,b]}D_b(x_i).
\end{align}
In other words, Step 1 ``approximates'' $\Omega_{\pm}$ by finitely many $su$--discs. This can be done e.g.\
by taking a sufficiently fine net of points in $\Omega_-\cup\Omega_+$ and then displacing each of them a small amount
in the flow direction so that their $su$--discs are all disjoint;
see a similar argument in \cite[Lemma 2.7]{Lima-Sarig}.

\medskip
\noindent
{\sc Step 2 (Security neighborhood $\wt N$):} Choose a compact neighborhood 
$N_+\subset \Omega_+$ such that
$\Omega_+^{=}\subset {\rm int}(N_+)$.
For $x\in N_+$, let $t(x)=\inf\{t>0:g_t(x)\in \Omega_-\}\in (0,+\infty]$, and let
$\wt N$ be the closure of $\{g_t(x):x\in N_+\text{ and }0\leq t\leq t(x)\}$.

\medskip
Since $\Omega_+^{=}$ is
the disjoint union of four closed curves, $N_+$ is the disjoint union of four compact sets. 
The set $\wt N$ is a compact neighborhood of all asymptotic vectors with $|s|\leq \ve_0$.

\medskip
\noindent
{\sc Step 3 (Construction of $\wt \Sigma$ far from $\Omega_\pm$):} Choose $x_{m+1},\ldots,x_{m+n}\in M$
such that:
\begin{enumerate}[$\circ$]
\item $\{D_b(x_{i})\}_{m+1\leq i\leq m+n}$ are pairwise disjoint $su$--discs, each of them
disjoint from $\wt N$ and disjoint from $\{D_b(x_{i})\}_{1\leq i\leq m}$; 
\item the disjoint union
$\wt\Sigma(r)\uplus\Omega_0$ 
is a global Poincar\'e section (i.e.\ a cross-section with bounded first return time) for all $\lambda\in[a,b]$, where
$\wt\Sigma(\lambda)=\bigcup_{1\leq i\leq m+n}D_\lambda(x_i)$.
\end{enumerate}

\medskip
Again, Step 3 can be carried out similarly to \cite[Lemma 2.7]{Lima-Sarig}. Note that
all flow trajectories that start in $N_+$ make the transition in the neck without visiting any $D_b(x_{i})$, $m+1\le i\le m+n$. 

\medskip
\noindent
{\sc The security section $\wt\Sigma$:} Define $ \wt\Sigma=\wt\Sigma(b)$.

\medskip
Thus $\wt\Sigma\uplus\Omega_0$ is the largest global Poincar\'e section constructed in Steps 1--3. 
Before continuing, let us introduce some further notation. For $\lambda\in [a,b]$,
let $\wt h_{\lambda}:\wt\Sigma(\lambda)\uplus\Omega_0\to \wt\Sigma(\lambda)\uplus\Omega_0$
be the corresponding Poincar\'e return map. 
Since the return times of $\wt h_{a},\wt h_{b}$ are bounded away from zero and infinity, 
we have $\wt h_{a}=\wt h_{b}^N$ where $N:\wt\Sigma(a)\uplus\Omega_0\to\{1,2,\ldots,N_0\}$ is bounded.
The same holds for every $\lambda\in [a,b]$, namely $\wt h_{\lambda}=\wt h_{b}^{N_{\lambda}}$
where $N_{\lambda}:\wt\Sigma(\lambda)\uplus\Omega_0\to\{1,2,\ldots,N_0\}$ is bounded
(the bound $N_0$ is the same).

For $\lambda\in[a,b]$, define the following objects:
\begin{enumerate}[$\circ$]
\item $\tau_{\lambda}:\wt\Sigma(\lambda)\to\N\cup\{\infty\}$ such that $\wt h_{\lambda}^{\tau_{\lambda}(x)}(x)$
is the first return of $x$  to $\wt\Sigma(\lambda)$.
\item $\mathfs S^+(\lambda)=\partial\wt\Sigma(\lambda)\cup\{\tau_{\lambda}=\infty\}$.
\item $f_{\lambda}:\wt\Sigma(\lambda)\backslash\mathfs S^+(\lambda)\to \wt\Sigma(\lambda)$
the Poincar\'e return map.\footnote{In general, the maps $f_\lambda$ need not be related
to the map $f_0$ introduced in Section \ref{ss:Omega}.}
\end{enumerate}
Note that $\tau_{\lambda}(x)=\infty$ only for asymptotic vectors, and that the flow time function of $f_{\lambda}$
is unbounded exactly when approaching asymptotic vectors. Yet, $f_{a}$ and $f_{b}$
differ by a bounded number of iterates, as we now explain.
Write $\tau=\tau_{a}(x)$, and let $0\leq i,j<\tau$ such that
$\wt h_{b}(x),\ldots,\wt h_{b}^i(x)\in\wt\Sigma(b)\setminus\wt\Sigma(a)$,
$\wt h_{b}^{i+1}(x),\ldots,\wt h_{b}^{\tau-j-1}(x)\in\Omega_0$ and
$\wt h_{b}^{\tau-j}(x),\ldots,\wt h_{b}^{\tau-1}(x)\in\wt\Sigma(b)\setminus\wt\Sigma(a)$,
i.e. $i$ is the last iterate before entering $\Omega_0$ and $\tau-j-1$ is the last iterate 
before leaving $\Omega_0$. Clearly $f_{a}(x)=f_{b}^{i+j+1}(x)$. Observing that 
$\wt h_a(x)=\wt h_b^{i+1}(x)$, it follows that $i+1\leq N_a$. Similarly, $j+1\leq N_a$.
Letting $\ell_0:=2N_0-1$, we conclude that $i+j+1\leq \ell_0$, hence
$f_{a}=f_{b}^\ell$ for some $\ell:\wt\Sigma(a)\backslash\mathfs S^+(a)\to\{1,\ldots,\ell_0\}$.

Finally, observe that for any
$\wt\Sigma(a)\subset X\subset \wt\Sigma(b)$
we can similarly define $\tau_X,\mathfs S^+(X),$ $f_X$, and that 
$f_X=f_{b}^{\ell_X}$ for some $\ell_X:X \backslash \mathfs S^+(X)\to\{1,\ldots,\ell_0\}$
(the bound $\ell_0$ is the same). Hence, controlling pre-iterates of $f_X$ up to order $n_0$ say
follows from controlling pre-iterates of $f_{b}$ up to order $\ell_0n_0$.
In summary, we just need to analyze a bounded number of iterates of a single map.
In the next subsection, we consider a multi-parameter family of such sections $X$ and show that
for some choice of parameters the section $X$ satisfies the required properties.

\begin{remark}\label{conjugacy-extended-f_0}
We can also apply Step 3 above to extend $\Omega_+\cup\Omega_-$ to a
Poincar\'e section $\wt\Omega$. The construction is simpler, since we only require
transversality with the flow direction (and not necessarily almost perpendicularity).
Hence the transition map $f_0$ can be extended to
$f_0:\wt\Omega\backslash \mathfs S^+(\wt\Omega)\to\wt\Omega$, where
$\mathfs S^+(\wt\Omega)$ is the union of the boundary of $\wt\Omega$ and the
points that never return to $\wt\Omega$ under the flow. The transition time function $\Upsilon_0$ can be extended
accordingly.
Similarly to the maps $f_\lambda$, the map $f_0$ is a Poincar\'e return map of $g$ that
captures all flow trajectories not asymptotic to $\gamma$. Hence $f_\lambda$ and
$f_0$ are conjugate, with transition time bounded from above.
\end{remark}

\subsection{Construction of $\wh \Sigma$}\label{ss:Sigma^}

Fix an integer $n_0$.
Let $A\subset \wt\Sigma$ be a connected curve. 
The next lemma is essential to the construction of $\wh \Sigma$, and shows how to perturb $A$ 
to avoid triple intersections of pre-iterates under $f_b$ up to order $\ell_0n_0$.
Since $\wt\Sigma$ is perpendicular to the flow, Lemma \ref{lem:subspaces}(2) implies that 
the stable/unstable subspaces $\wh E^{s/u}$ for the flow project to directions $E^{s/u}$ in $\wt\Sigma$. (For the moment, this is purely notational and no dynamical properties are claimed for $E^{s/u}$.)

\begin{lemma}\label{lem:perturbation}
Given a connected curve $A\subset \wt\Sigma$ transverse to 
$E^s,E^u$ and $\ve>0$,
there is a one-parameter family $\{A(t)\}_{|t|\leq 1}$ of disjoint curves, each of them $\ve$--close
to $A$ in the $C^1$--norm, such that 
\begin{align*}
A(t)\cap f_b^{-i}[A(t)]\cap f_b^{-j}[A(t)]=\emptyset & \\
A(t)\cap f_b^{-i}[A(t)]\cap f_b^{-j}[C]=\emptyset & \\
C\cap f_b^{-i}[A(t)]\cap f_b^{-j}[A(t)]=\emptyset & 
\end{align*}
for all $0<i<j\leq \ell_0n_0$ and $|t|\leq 1$. 
\end{lemma}

To prove Lemma \ref{lem:perturbation}, we provide a combinatorial description of the trajectories of $A$,
according to the transitions in the neck $\mathcal N$ that spend a long time. 
This combinatorial description (decomposition) and some notation (long backward transition
and parameter $\eta>0$) will be only used in this section. 

\medskip
\noindent
{\sc Projection map to $\Omega_\pm$:} The {\em projection map} to $\Omega_\pm$ is the map
$\mathfrak p:g_{[-\chi,\chi]}(\Omega_+\cup\Omega_-)\to \Omega_+\cup\Omega_-$
defined by $\mathfrak p(g^t(x))=x$ for $(x,t)\in (\Omega_+\cup\Omega_-)\times[-\chi,\chi]$.

\medskip
This map allows to localize our analysis inside $\Omega_\pm$. Recall
the transition map $f_0$ studied in Section \ref{ss:explicit}. 
Observe that if $B\subset \Omega_-$ is a curve intersecting $\Omega_-^=$, then $f_0^{-1}(B)\subset \Omega_+$
is the union of two disjoint curves of infinite length, each of them accumulating at $\Omega_+^=$, see Figure~\ref{fig:explosion}.
\begin{figure}[hbt!]
\centering
\def\svgwidth{12.3cm}
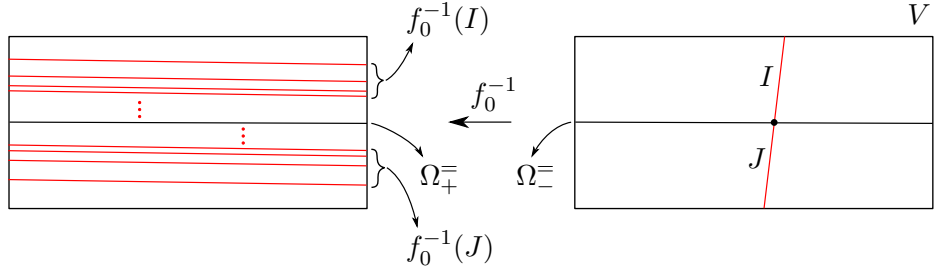
\caption{The pre-iterates of a curve intersecting $\Omega_-^=$ equals two curves of infinite length
accumulating at $\Omega_+^=$.}\label{fig:explosion}
\end{figure}
The proof of this fact is easy. By symmetry, it is enough to prove the analogous result for the forward iterate
of $f_0$, and we know for instance on $\Omega_1$ that $\zeta(\psi)\to\infty$ as 
$\psi\to \psi_0^-$ and as $\psi\to \psi_0^+$ (and similarly on the other three parts of $\Omega_+$ with $\psi_0$ replaced by $-\psi_0$ or $\pm(\pi-\psi_0)$ as appropriate).

The transition of Figure~\ref{fig:explosion} is the only source of unboundedness when considering pre-iterates of $f_{b}$, since
otherwise the return time of the flow is uniformly bounded. To better analyze this phenomenon, fix $\eta>0$ small and define
\[
U=\{x\in \Omega_-:d(x,\Omega_-^=)<\eta\}, \qquad
V=\{x\in \Omega_-:d(x,\Omega_-^=)<2\eta\}.
\]

\noindent
{\sc Long backward transition (LBT):} The point $x\in \wt\Sigma$ has a {\em long backward transition (LBT)}
at time $i$ if $f_{b}^{-i}x\in g_{[-\chi,\chi]}V$ and $f_{b}^{-(i+1)}x\in g_{[-\chi,\chi]}\Omega_+$.
When this happens, we say that the LBT occurs at the point $f_{b}^{-i}x$.

\medskip
\noindent
{\sc Decomposition of $A$:} Given a curve $A\subset \wt\Sigma$ with finite
length and finitely many connected components, write $A=A_0\uplus\cdots\uplus A_{\ell_0n_0}$ where
\begin{align*}
A_i&=\{x\in A:x\text{ makes the first LBT at time }i\}, \ 0\leq i<\ell_0n_0\\
A_{\ell_0n_0}&=A\setminus(A_0\uplus\cdots\uplus A_{\ell_0n_0-1}).
\end{align*}
Note that $A_{\ell_0 n_0}$ is the set of points with first LBT with time at least $\ell_0n_0$, and 
includes points with no LBT. Each $A_i$, $i<\ell_0n_0$, is the disjoint union of finitely
many open pieces of $A$, and $A_{\ell_0n_0}$ is the disjoint union of finitely many pieces of $A$.

Consider $C=g_{[-\chi,\chi]}\Omega_+^=\cap\wt\Sigma$, which is the disjoint union of 
finitely many pieces of curves asymptotic to $\gamma$ under  the map $\wt h_{b}$. 
Decomposing $C=C_0\uplus C_1\uplus\cdots\uplus C_{\ell_0n_0-1}$ as above, the first LBT's associated
to $C$ occur at the set $\wt C=C_0\uplus f_{b}^{-1}(C_1)\uplus \cdots\uplus f_{b}^{-\ell_0n_0+1}(C_{\ell_0n_0-1})$,
equal to the disjoint union of finitely many curves asymptotic to $\gamma$ under the map $\wt h_{b}$.
Then
$$
\mathfrak p[\wt C]=\mathfrak p[C_0]\uplus\mathfrak p[f_{b}^{-1}(C_1)]\uplus\cdots\uplus
\mathfrak p[f_{b}^{-\ell_0n_0+1}(C_{\ell_0n_0-1})]
$$
is the disjoint union of finitely many curves asymptotic to $\gamma$ under the flow.
Let $H=\mathfrak p[\wt C]\cap \Omega_-^=$,
which is a finite set equal to all homoclinic intersections associated to first LBT's.
We claim that all other LBT's accumulate in $H$. The proof is by induction
on the number of LBT's. Assume that $I\subset C_i$ has the second LBT at time $j>i$.
The set $\mathfrak p[f_{b}^{-(i+1)}(C_i)]$ accumulates at $\Omega_+^=$,
and $f_{b}^{-j}(C_i)$ is obtained from $f_{b}^{-(i+1)}(C_i)$ by a uniformly bounded flow time,
hence $\mathfrak p[f_{b}^{-j}(C_i)]$ accumulates on $H$. The claim follows.

For $\eta>0$ small enough, all first LBT's of $C$ are associated
to a point of $H$, i.e.\ the piece $\mathfrak p[f_{b}^{-i}(C_i)]$ not only intersects $V$
but indeed crosses $\Omega_-^=$.
See Figure~\ref{fig:homoclinic} to understand the dynamics of accumulations around $\Omega_-^=$.
The red/green intervals in the left figure are pieces of $\mathfrak p[C_i]/\mathfrak p[C_k]$,
and the vertical red/green curves in the right figure are $\mathfrak p[f_b^{-i}(C_i)]/\mathfrak p[f_b^{-k}(C_k)]$,
equal to the pre-iterates right before the first LBT's. In the figure, they define four homoclinic points.
We also depict an interval $\mathfrak p(I)$ that makes two LBT's. The first LBT generates
the two curves of infinite length in the left figure, both accumulating at $\Omega_-^=$.
Finally, the vertical blue curves in the right figure are $\mathfrak p[f_b^{-j}(I)]$, equal to the pre-iterates right before
the second LBT of $I$. Observe that they accumulate on each of the four homoclinic points.
\begin{figure}[hbt!]
\centering
\def\svgwidth{12.3cm}
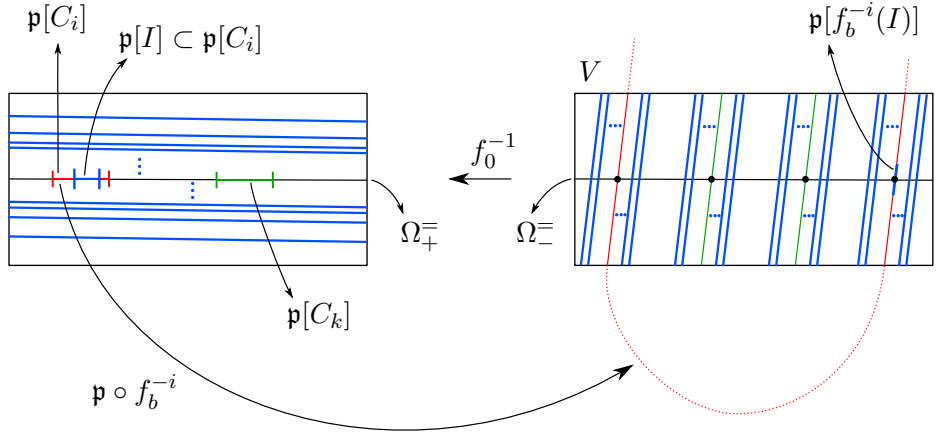
\caption{Dynamics near homoclinic points.}\label{fig:homoclinic}
\end{figure}

Now we are able to prove Lemma \ref{lem:perturbation}.

\begin{proof}[Proof of Lemma \ref{lem:perturbation}]
To obtain a one-parameter family, it is enough to perturb $A$ so that the intersection conditions 
of the statement hold robustly.  We have
\begin{align*}
\{f_b^{-k}(A):0\leq k\leq \ell_0n_0\}&=\underbrace{\left\{f_b^{-k}(A_i):0\leq k\leq i\right\}}_{\mathfs F_A}\cup
\underbrace{\left\{f_b^{-k}(A_i):i<k\leq \ell_0n_0\right\}}_{\mathfs G_A}\\
\{f_b^{-k}(C):0\leq k\leq \ell_0n_0\}&=\underbrace{\left\{f_b^{-k}(C_i):0\leq k\leq i\right\}}_{\mathfs F_C}\cup
\underbrace{\left\{f_b^{-k}(C_i):i<k\leq \ell_0n_0\right\}}_{\mathfs G_C}.
\end{align*}
Let $\mathfs F=\mathfs F_A\cup\mathfs F_C$ and $\mathfs G=\mathfs G_A\cup\mathfs G_C$.
Since $C$ is asymptotic to $\gamma$ under the flow,
$\mathfs F_C\cup\mathfs G_C$ is fixed (does not depend on $A$) and has no double intersections.
Observe that $\mathfs F$ is a finite family of bounded curves, obtained from pieces of $A$ and $C$ 
by uniformly bounded flow time displacements. Let $T_0$ be a bound on such time. 

We start by controlling all possible triple intersections of $A$ and pre-iterates of $A,C$.
There are three possible types of such intersections:
\begin{enumerate}[Type 1:]
\item $A\cap F_1\cap F_2$, where $F_1,F_2\in\mathfs F$.
\item $A\cap F\cap G$, where $(F,G)\in\mathfs F\times\mathfs G$.
\item $A\cap G_1\cap G_2$, where $G_1,G_2\in\mathfs G$.
\end{enumerate}
For each type, we perform finitely many $C^1$ perturbations on $A$ to prevent triple intersections robustly.

\medskip
\noindent
Type 1: We can assume that $F_1\in\mathfs F_A$. If $F_2\in\mathfs F_A$, then
the intersection $A\cap F_1\cap F_2$ is associated to flow displacements of time $\leq T_0$.
By~\eqref{eq:Knieper},
the flow has finitely many closed orbits of length $\leq T_0$. Hence
we can perform an arbitrarily small $C^1$ perturbation of $A$ to destroy these intersections robustly.
If $F_2\in\mathfs F_C$, the same applies to guarantee that $A\cap F_1$ does not intersect
$F_2$ robustly.

\medskip
\noindent
Type 2: The set $A\cap \mathfs F=\bigcup_{F\in\mathfs F}\{A\cap F\}$ is finite, and the set
$A\cap \mathfs G=\bigcup_{G\in\mathfs G}\{A\cap G\}$ is countable with a finite set of accumulation points,
coming from pre-iterates of $\Omega_+^=$. A $C^1$ perturbation of order $O(\ve)$ of $A$ changes
$A\cap \mathfs F$ by $O(\ve)$ inside $A$, and $A\cap\mathfs G$ around its
accumulation points by $o(\ve)$ inside $A$, hence we can destroy all such triple intersections robustly.

\medskip
\noindent
Type 3: We divide this type into three subtypes.
\begin{enumerate}[$\circ$]
\item Type 3.1: $A\cap f_b^{-k}(A_i)\cap f_b^{-m}(A_j)$, with $k>i$ and $m>j$.
\item Type 3.2: $A\cap f_b^{-k}(A_i)\cap f_b^{-m}(C_j)$, with $k>i$, $m>j$ and $k-i\leq m-j$.
\item Type 3.3: $A\cap f_b^{-k}(C_i)\cap f_b^{-m}(A_j)$, with $k>i$, $m>j$ and $k-i\leq m-j$. 
\end{enumerate}
The idea is to push this intersection to $V$. The hardest case is Type 3.1, where all three
sets are simultaneously perturbed. Let us start with it.
Assuming that $k-i\leq m-j$, iterate the intersection
$k-i$ times (the intersection
belongs to $f_b^{-m}(A_j)$ and hence can actually be iterated $m$ times), so that
$f_b^{-i}(A_i)\cap f_b^{-(m-k+i)}(A_j)\neq \emptyset$. We show that a small $C^1$ perturbation 
of $A_i$ makes this intersection empty inside $U$. This is enough for us, since
intersections outside $U$ are associated to uniformly bounded flow times, which can be 
treated as in Type~1. Actually, we show how to guarantee that
\mbox{$\mathfrak p[f_b^{-i}(A_i)]\cap \mathfrak p[f_b^{-(m-k+i)}(A_j)]\cap U= \emptyset$}. The argument is similar to Type 2.
Since the set $\Omega_-^=\cap \mathfrak p[f_b^{-i}(A_i)]$ is finite and all accumulation points of 
$\Omega_-^=\cap \mathfrak p[f_b^{-(m-k+i)}(A_j)]$ are contained in $H$, a $C^1$ perturbation 
of order $O(\ve)$ of the piece $B_i\subset A_i$ such that $\mathfrak p[f_b^{-i}(B_i)]=\mathfrak p[f_b^{-i}(A_i)]\cap U$
changes
$\mathfrak p[f_b^{-i}(A_i)]$ by $O(\ve)$, and $\Omega_-^=\cap \mathfrak p[f_b^{-(m-k+i)}(A_j)]$ around its
accumulation points by $o(\ve)$. Therefore a small perturbation guarantees that
$\mathfrak p[f_b^{-i}(A_i)]\cap \mathfrak p[f_b^{-(m-k+i)}(A_j)]\cap U=\emptyset$.

Types 3.2 and 3.3 are simpler, since $C$ is not perturbed. In Type 3.2, we perform the same argument described above
to have $\mathfrak p[f_b^{-i}(A_i)]\cap \mathfrak p[f_b^{-(m-k+i)}(C_j)]\cap U= \emptyset$.
Again, $\Omega_-^=\cap \mathfrak p[f_b^{-i}(A_i)]$ is finite and all accumulation points of 
$\Omega_-^=\cap \mathfrak p[f_b^{-(m-k+i)}(C_j)]$ are contained in $H$.  A $C^1$ perturbation 
of order $O(\ve)$ of the piece $B_i\subset A_i$ such that $\mathfrak p[f_b^{-i}(B_i)]=\mathfrak p[f_b^{-i}(A_i)]\cap U$
changes $\mathfrak p[f_b^{-i}(A_i)]$ by $O(\ve)$, while $\Omega_-^=\cap \mathfrak p[f_b^{-(m-k+i)}(C_j)]$
remains fixed, so we can make 
$\mathfrak p[f_b^{-i}(A_i)]\cap \mathfrak p[f_b^{-(m-k+i)}(C_j)]\cap U=\emptyset$.

In Type 3.3, once again $\Omega_-^=\cap \mathfrak p[f_b^{-i}(C_i)]$ is finite (and contained in $H$)
and all accumulation points of $\Omega_-^=\cap \mathfrak p[f_b^{-(m-k+i)}(A_j)]$ are contained in $H$.
A $C^1$ perturbation of order $O(\ve)$ of the piece $B_j\subset A_j$ such that
$\mathfrak p[f_b^{-(m-k+i)}(B_j)]=\mathfrak p[f_b^{-(m-k+i)}(A_j)]\cap U$
changes $\mathfrak p[f_b^{-(m-k+i)}(A_j)]$, while $\Omega_-^=\cap \mathfrak p[f_b^{-i}(C_i)]$
remains fixed, so we can make 
$\mathfrak p[f_b^{-i}(C_i)]\cap \mathfrak p[f_b^{-(m-k+i)}(A_j)]\cap U= \emptyset$.

\medskip
To finish the lemma, we deal with triple intersections of $C$ and pre-iterates of $A$. There are also three
types of such intersections:
\begin{enumerate}[Type 1':]
\item $C\cap F_1\cap F_2$, where $F_1,F_2\in\mathfs F_A$.
\item $C\cap F\cap G$, where $(F,G)\in\mathfs F_A\times\mathfs G_A$.
\item $C\cap G_1\cap G_2$, where $G_1,G_2\in\mathfs G_A$.
\end{enumerate}
The analysis of these types is simpler than the previous ones, since $C$ is fixed.

\medskip
\noindent
Type 1': Proceed as in Type 1.

\medskip
\noindent
Type 2': Proceed as in Type 2.

\medskip
\noindent
Type 3': Proceed as in Type 3.1, guaranteeing that, in its notation,
$\mathfrak p[f_b^{-i}(A_i)]\cap \mathfrak p[f_b^{-(m-k+i)}(A_j)]\cap U \neq \emptyset$.
\end{proof}

Using Lemma \ref{lem:perturbation}, we now construct a parametric family of sections.

\medskip
\noindent
{\sc Step 4 (Construction of a parametric family $\wh\Sigma(\vec{t}\,)$ of sections):}
For each $1\leq i\leq m+n$, choose finitely many 
families $\{A_{i,j}(t)\}_{|t|\leq 1}$, $1\le j\le N_i$.
Given 
$\vec t=(t_{i,j})_{1\leq i\leq m+n\atop{1\leq j\leq N_i}}$, we require that
$B_i(\vec t_i)=B_i(t_{i,1},\ldots,t_{i,N_i})$ is a topological disc whose boundary is the polygon defined by
$\{A_{i,j}(t_{i,j})\}_{1\leq j\leq N_i}$ and such that
$D_a(x_i)\subset B_i(\vec t_i)\subset D_b(x_i)$
for $i=1,\ldots,m+n$.
See Figure~\ref{fig:sigma-hat}.
We require that 
$$
\wh\Sigma(\vec t\, )=\biguplus\limits_{1\leq i\leq m+n} B_i(\vec t_i )
$$
defines a cross-section to the flow satisfying
the following conditions:
\begin{enumerate}[i---]
\item[(H1)] $\{A_{i,j}(t)\}_{|t|\leq 1}$ satisfies Lemma \ref{lem:perturbation} for every $1\leq i\leq m+n$,
$1\leq j\leq N_i$.
\item[(H2)] $A_{i_1,j_1}(t)\pitchfork f_{b}^{-k}(A_{i_2,j_2}(t'))$ for all distinct pairs
$(i_1,j_1),(i_2,j_2)\in\{(i,j):1\leq i\leq m+n,1\leq j\leq N_i\}$, $|t|,|t'|\leq 1$ and $0\leq k\leq \ell_0n_0$.
In particular, $A_{i_1,j_1}(t)\cap f_{b}^{-k}(A_{i_2,j_2}(t'))$ 
is at most countable.
\end{enumerate}

\begin{figure}[hbt!]
\centering
\def\svgwidth{12cm}
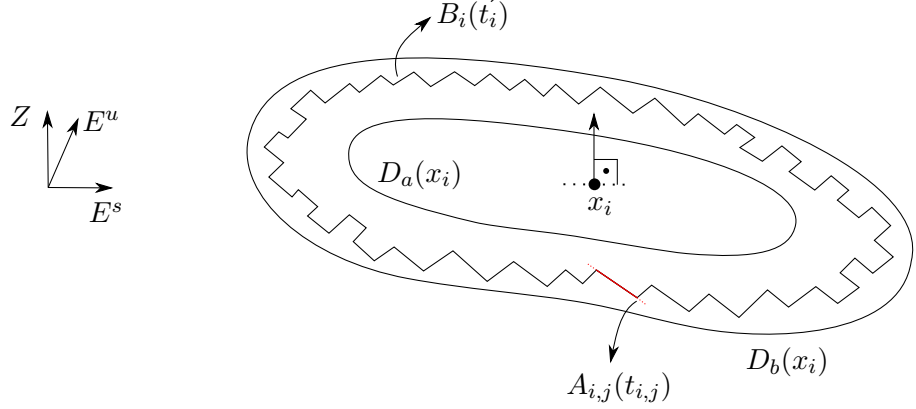
\caption{Construction of $B_i(\vec t_i)$.}\label{fig:sigma-hat}
\end{figure}

\medskip
Here is how we guarantee (H2). Since
$\{A_{i,j}(t)\}_{|t|\leq 1}$ satisfies Lemma \ref{lem:perturbation}, it is
transverse to $E^s,E^u$. The pre-iterates of pieces of curves that 
make no LBT are associated to bounded flow times, hence can be perturbed 
to satisfy (H2). Now, if $A$ makes its first LBT at time $i$, then
up to a compact component the infinite curves composing $f_{b}^{-(i+1)}(A)$ belong to the stable cone,
and so are transverse to every $A_{i,j}(t)$ after a small perturbation. The transversality implies that each 
compact component of $f_{b}^{-k}(A_{i_2,j_2}(t'))$ intersects $A_{i_1,j_1}(t)$ in finitely
many points, therefore the intersection $A_{i_1,j_1}(t)\cap f_{b}^{-k}(A_{i_2,j_2}(t'))$ 
is at most countable.

To finish the construction of $\wh\Sigma$, we show that the space of parameters $\vec t$
such that $\mathfs S^+(\wh\Sigma(\vec t\,))$ has a triple intersection up to the $\ell_0n_0$'th pre-iterate of 
$f_{b}$ has zero Lebesgue measure. For that,
we analyze all possible triple intersections.
Recall that $\mathfs S^+(\wh\Sigma(\vec t\,))=\partial\wh\Sigma(\vec t\,)\cup\{\tau_{\wh\Sigma(\vec t\,)}=\infty\}$.
Since $\{\tau_{\wh\Sigma(\vec t\,)}=\infty\}$ is asymptotic to $\gamma$,
there are not even double intersections associated to it. We have six remaining possibilities for
triple intersections:
\begin{enumerate}[$\circ$]
\item $A_{i,j}(t_{i,j})\cap f_{b}^{-k}[A_{i,j}(t_{i,j})]\cap f_{b}^{-p}[A_{i,j}(t_{i,j})]$:
this intersection is empty, by (H1).
\item $A_{i_1,j_1}(t_{i_1,j_1})\cap f_{b}^{-k}[A_{i_2,j_2}(t_{i_2,j_2})]\cap f_{b}^{-p}[A_{i_3,j_3}(t_{i_3,j_3})]$
with $(i_3,j_3)\neq (i_1,j_1)$ or $(i_3,j_3)\neq (i_2,j_2)$: by (H2), $A_{i_1,j_1}(t_{i_1,j_1})\cap f_{b}^{-k}[A_{i_2,j_2}(t_{i_2,j_2})]$
is at most countable. If we fix all parameters except $t_{i_3,j_3}$,
there are at most countably many choices for $t_{i_3,j_3}$ such that the triple intersection is non-empty.
\item $A_{i,j}(t_{i,j})\cap f_{b}^{-k}[A_{i,j}(t_{i,j})]\cap
f_{b}^{-p}[\{\tau_{\wh\Sigma(\vec t\,)}=\infty\}]$: this intersection is empty, by (H1).
\item $A_{i_1,j_1}(t_{i_1,j_1})\cap f_{b}^{-k}[A_{i_2,j_2}(t_{i_2,j_2})]\cap
f_{b}^{-p}[\{\tau_{\wh\Sigma(\vec t\,)}=\infty\}]$ with $(i_1,j_1)\neq (i_2,j_2)$: the intersection
$A_{i_1,j_1}(t_{i_1,j_1})\cap f_{b}^{-p}[\{\tau_{\wh\Sigma(\vec t\,)}=\infty\}]$ is at most countable,
since every compact component of $f_{b}^{-p}[\{\tau_{\wh\Sigma(\vec t\,)}=\infty\}]$
is transverse to $A_{i_1,j_1}(t_{i_1,j_1})$ and hence intersects it in finitely many points. 
Thus, if we fix all parameters except $t_{i_2,j_2}$,
there are at most countably many choices for $t_{i_2,j_2}$ such that the triple intersection is non-empty.
\item $\{\tau_{\wh\Sigma(\vec t\,)}=\infty\}\cap f_{b}^{-k}[A_{i,j}(t_{i,j})]\cap
f_{b}^{-p}[A_{i,j}(t_{i,j})]$: this intersection is empty, by (H1).
\item $\{\tau_{\wh\Sigma(\vec t\,)}=\infty\}\cap f_{b}^{-k}[A_{i_2,j_2}(t_{i_2,j_2})]\cap
f_{b}^{-p}[A_{i_3,j_3}(t_{i_3,j_3})]$ with $(i_2,j_2)\neq (i_3,j_3)$: this case is similar to the fourth one, since again
$\{\tau_{\wh\Sigma(\vec t\,)}=\infty\}\cap f_{b}^{-k}[A_{i_2,j_2}(t_{i_2,j_2})]$ is at most
countable.
\end{enumerate}

\medskip
\noindent
{\sc The section $\wh\Sigma$:} Define $\wh\Sigma=\wh\Sigma(\vec t\,)$, where $\vec t$ is any
parameter such that $\mathfs S^+(\wh\Sigma(\vec t\,))$ has no triple intersections
up to the $\ell_0n_0$'th pre-iterate under $f_{b}$.

\subsection{Construction of $\Sigma_0$}\label{ss:Sigma}

The final step in the construction, which leads to the section $\Sigma_0$,
is to make small flow displacements in $\wh\Sigma$ so that the Poincar\'e return
time of $\Sigma_0$ is at least $T_\chi>0$ and $\Sigma_0$ is almost perpendicular to the flow direction.
We measure the perpendicularity using a new parameter $\epsilon\ll\chi$.
For each $1\leq i\leq m+n$, let $\mathfrak p_i:g_{[-\chi,\chi]}D_b(x_i)\to D_b(x_i)$ be the flow projection.
Write $\wh\Sigma=\bigcup_{1\leq i\leq m+n}B_i$,
and let $t_{\rm min}$ be the minimal flow time defined by $f_b$.

\medskip
\noindent
{\sc Step 5 (Refinement of $\wh\Sigma$):} For each $1\leq i\leq m+n$, construct a family
$\mathcal Q_i=\{Q\}$ such that:
\begin{enumerate}[$\circ$]
\item each $Q\in\mathcal Q_i$ is contained in a $su$--disc of radius $<\epsilon$
and
$$
B_i\subset \bigcup_{Q\in\mathcal Q_i}g_{\left[-\frac13 t_{\rm min},\frac13 t_{\rm min}\right]}Q;
$$
\item $Q\cap g_{\left[-\frac{1}{100}t_{\rm min},\frac{1}{100}t_{\rm min}\right]}Q'=\emptyset$
for all distinct $Q,Q'\in\mathcal Q_i$.
\item $\mathfrak p_i[\partial Q]$ is transverse to $E^s,E^u$ for all $Q\in\mathcal Q_i$;
\item there are no triple intersections between $\mathfs S^+(\wh\Sigma)$ and
$\bigcup_{1\leq i\leq m+n\atop{Q\in\mathcal Q_i}}\mathfrak p_i[\partial Q]$ up to the 
$\ell_0n_0$'th pre-iterate under $f_{b}$.
\end{enumerate}

\medskip
Observe that the smaller $\epsilon$ is, the more perpendicular is $Q$ to the flow. 
To create $\mathcal Q_i$, first consider a refinement $\wh{\mathfs R_i}=\{\wh R\}$ of $B_i$ by finitely
many compact curves transverse to $E^s,E^u$, see Figure~\ref{fig:sigma}.
Proceeding similarly to the proof of \cite[Lemma 2.7]{Lima-Sarig}, displace each $\wh R$ in the flow direction
to obtain $R$, so that $\mathfs R_i=\{R\}$ is a disjoint family and the displacements of neighbor $\wh R$'s 
differ at least $t_{\rm min}/50$. For each $R\in\mathfs R_i$, choose $y_R\in R$
and apply Step~4 to construct $\overline{R}\subset Q\subset D_{{\rm diam}(R)}(y_R)$ satisfying the above conditions, where $\overline R$ is the flow projection of $R$ to $D_{{\rm diam}(R)}(y_R)$.

\begin{figure}[hbt!]
\centering
\def\svgwidth{12cm}
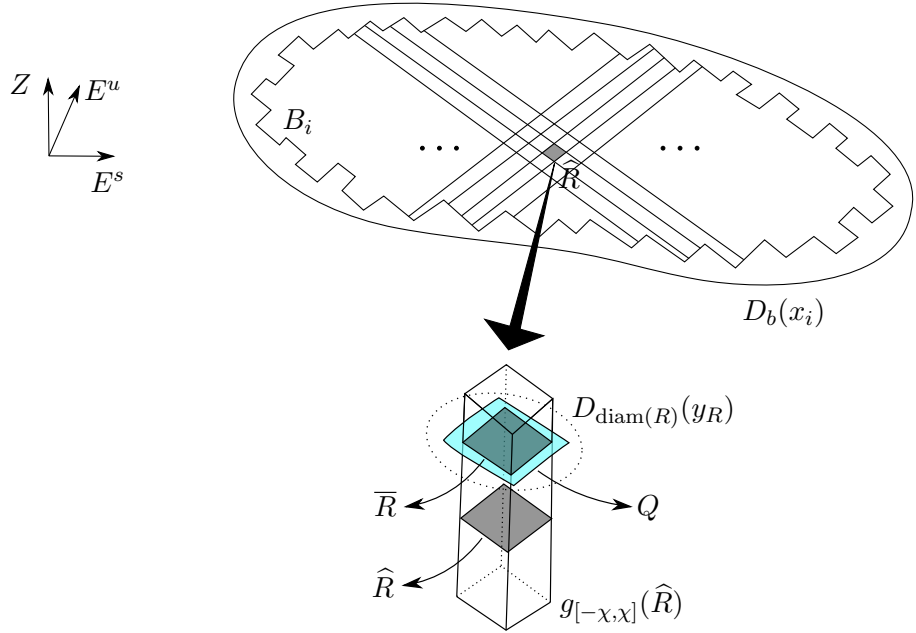
\caption{Construction of $\Sigma_0$: refine $\wh\Sigma$ displacing each component to improve perpendicularity and preserving no triple intersections.}\label{fig:sigma}
\end{figure}

\medskip
\noindent
{\sc The section $\Sigma_0$:} Define $\Sigma_0=\bigcup_{1\leq i\leq m+n\atop{Q\in\mathcal Q_i}} Q$.

\medskip
By the first two conditions in Step 5, the corresponding Poincar\'e return time of $\Sigma_0$
is bounded below by a constant $T_{\chi}>0$ which is independent of $\epsilon$.

\subsection{The first return map $f$}\label{ss:induced-def}


We now define a first return map $f$ which will eventually be shown to satisfy the Chernov axioms.
Recall that, by construction, all flow trajectories intersect $\Sigma_0$ infinitely often except those forward and backward asymptotic to $\gamma$.

\medskip
\noindent
{\sc Return time function:} The {\em return time function} of $\Sigma_0$ is $\tau=\tau_+:\Sigma_0\to (0,\infty]$
such that $\tau(x)=\inf\{t>0:g_t(x)\in\Sigma_0\}$. Define also $\tau_-:\Sigma_0\to [-\infty,0)$
by $\tau_-(x)=\sup\{t<0:g_t(x)\in\Sigma_0\}$.

\medskip
We have $\tau(x)=\infty$ if and only if $x$ is an asymptotic vector, hence $\{\tau=\infty\}$ is a finite
union of compact curves, each of them contained in $g_{[-\chi,\chi]}\Omega_+^=$. 

\medskip
\noindent
{\sc Primary singular sets $\mathfs S_{\rm P}^\pm$:} The {\em primary singular set}
$\mathfs S_{\rm P}=\mathfs S_{\rm P}^+$ is defined as
$$
\mathfs S_{\rm P}=\{x\in\Sigma_0:\tau(x)<\infty\text{ and }g_{\tau(x)}(x)\in\partial\Sigma_0\}\cup\{\tau = \infty\}.
$$
Similarly, the {\em primary singular set} $\mathfs S_{\rm P}^-$ is defined as
$$
\mathfs S_{\rm P}^-=\{x\in\Sigma_0:\tau_-(x)>-\infty\text{ and }g_{\tau_-(x)}(x)\in\partial\Sigma_0\}\cup\{\tau_- = -\infty\}.
$$

Note that $\mathfs S_{\rm P}^\pm$ are closed sets.
%
Proceeding similarly to Section \ref{ss:times}, we partition a neighborhood of $\{\tau=\infty\}$
into homogeneity bands. For that, fix a sufficiently large integer $n_0$
(how large $n_0$ is will depend on a finite number of conditions,
which include the validity of Lemma \ref{lem:BD} and the estimates in Section \ref{sec:CZ-scheme}).

\medskip
\noindent
{\sc Homogeneity bands on $\Sigma_0$:} For each $n\geq n_0$, the {\em homogeneity band} with 
index $n$ is $\mathfs D_n=\mathfs{D}_n^{>}\cup\mathfs{D}_n^{<}$ where 
\begin{align*}
\mathfs{D}_{n}^{>}&=\left\{x\in{\rm int}(\Sigma_0)\cap g_{[-\chi,\chi]}\Omega_+: 1+\tfrac{1}{(n+1)^2}<|c(x)|< 1+\tfrac{1}{n^2}\right\}\\
\mathfs{D}_n^{<}&=\left\{x\in{\rm int}(\Sigma_0)\cap g_{[-\chi,\chi]}\Omega_+: 1-\tfrac{1}{n^2}< |c(x)|<1-\tfrac{1}{(n+1)^2}\right\}.
\end{align*}

\medskip
\noindent
{\sc Secondary singular sets $\mathfs S_{\rm S}^\pm$:} The {\em secondary singular set}
$\mathfs S_{\rm S}=\mathfs S_{\rm S}^+$ is
$$
\mathfs S_{\rm S}=\bigcup_{n\geq n_0}\partial\mathfs D_n.
$$
The {\em secondary singular set} $\mathfs S_{\rm S}^-$ is
$$
\mathfs S_{\rm S}^-=\{g_{\tau(x)}(x):x\in\mathfs S_{\rm S}\}.
$$

\medskip
Let $X_0={\rm int}(\Sigma_0)$, $\mathfs S^+=\mathfs S_{\rm P}\cup\mathfs S_{\rm S}$, and
$\mathfs S^-=\mathfs S_{\rm P}^-\cup\mathfs S_{\rm S}^-$.

\medskip
\noindent
{\sc First return map $f$:} Define the map $f:X_0\backslash\mathfs S^+\to X_0\backslash\mathfs S^-$
to be the first return map of the flow to $\Sigma_0$, i.e.\ $f(x)=g_{\tau(x)}(x)$.

\medskip
The map $f$ has the same regularity of  $g_t$, hence it is $C^2$. It has uniformly bounded
derivatives away from $\{\tau=\infty\}$. Since $\Sigma_0$ is almost perpendicular to the flow direction,
the hyperbolicity properties of $f$ away from $\{\tau=\infty\}$ are almost the same as those of the flow.
We lose control as we approach $\{\tau=\infty\}$, and the homogeneity bands $\mathfs{D}_n$ 
enable us to recover this control.
Since $\Sigma_0$ is obtained from small flow displacements of $\wh\Sigma\subset\wt\Sigma$,
the first return map $f$ is a small perturbation (in the flow direction) of (a restriction of) $f_{b}$ and, for trajectories
near asymptotic vectors, $f_{b}$ is a small perturbation of $f_0$. Hence, we can understand $f$ inside
homogeneity bands by studying $f_0$ inside homogeneity bands. Let us be more specific in the relation
between $f$ and $f_0$. Since we are interested in the transitions in the neck, define
\begin{align*}
&\Sigma_+=\big\{x\in\Sigma_0\cap g_{[-\chi,\chi]}\Omega_+:
||c(x)|-1|<\tfrac{1}{n_ 0^2}\text{ and }fx\in g_{[-\chi,\chi]}\Omega_-\big\}\\
&\Sigma_-=f(\Sigma_+)=\big\{x\in\Sigma_0\cap g_{[-\chi,\chi]}\Omega_-:
||c(x)|-1|<\tfrac{1}{n_ 0^2}\text{ and }f^{-1}x\in g_{[-\chi,\chi]}\Omega_+\big\}
\end{align*}
It is clear that $f\restriction_{\Sigma_+}:\Sigma_+\to\Sigma_-$.

\medskip
\noindent
{\sc Coordinate maps $\mathfrak p_\pm$ and $\mathfrak t_\pm$:} The {\em coordinate maps}
$\mathfrak p_\minus:\Sigma_+\to\Omega_+$ and $\mathfrak t_\minus:\Sigma_+\to [-\chi,\chi]$ are defined by the 
equality $z=g_{\mathfrak t_\minus(z)}[\mathfrak p_\minus(z)]$ for $z\in \Sigma_+$.
The {\em coordinate maps} $\mathfrak p_\plus:\Sigma_-\to\Omega_-$ and
$\mathfrak t_\plus:\Sigma_-\to [-\chi,\chi]$ are defined analogously.

\medskip
The coordinate maps have the same regularity of $g_t$, hence they are $C^2$. 
Since $\Sigma_0$ and $\Omega$ are uniformly transversal to the flow direction, we have
that $\|d\mathfrak p_\pm^{\pm 1}\|\approx 1$.
By the first inclusion of (\ref{eq:step1}), $\mathfrak p_\minus$ is surjective. It is also injective,
because if $x\in \Sigma_+$ then $x,fx$ are uniquely characterized as being the starting/ending point 
of the transition in the neck. Therefore, $\mathfrak p_\minus$ is a bijection. By symmetry,
the same holds for $\mathfrak p_\plus$. Recall the definition of $\mathfs C_n$ in
Section \ref{ss:times}. We note that:
\begin{enumerate}[$\circ$]
\item $\mathfs C_n=\mathfrak p_\minus(\mathfs D_n)$ for all $n\geq n_0$.
\item $f=\mathfrak p_\plus^{-1}\circ f_0\circ\mathfrak p_\minus$ in $\Sigma_+$.
\end{enumerate}

Since $E^{s,u}$ are defined in $\Sigma_\pm,\Omega_\pm$ as the projections of
$\wh E^{s,u}$ onto the respective tangent spaces, the maps $\mathfrak p_\pm$ preserves
these subspaces. Using that $\|d\mathfrak p_{\pm}^{\pm 1}\|\approx 1$, we obtain that
$\|df\restriction_{E^{s,u}_x}\|\approx \|df_0\restriction_{E^{s,u}_{\mathfrak p_\minus(x)}}\|$
for $x\in\Sigma_+$. In the next two subsections, we will estimate $\|df_0\restriction_{E^{s,u}}\|$
inside homogeneity bands and some related bounds.

\subsection{Excursion times in the neck}\label{ss:excursions}

In the notation of Section \ref{ss:explicit},
let $\x=\x(t)$ be a bouncing/crossing geodesic undergoing an excursion in the neck.
Recall from Section~\ref{ss:explicit} that $\Upsilon_0$ (more precisely $2\Upsilon_0$) is the transition time from
$\Omega_+$ to $\Omega_-$.  Similarly, we define
$\Upsilon$ to be the time taken to pass from $\Sigma_+$ to $\Sigma_-$.
Since $f=\mathfrak p_\plus^{-1}\circ f_0\circ\mathfrak p_\minus$ on $\Sigma_+$, we have the relation 
\begin{align}\label{eq:times}
\Upsilon(\x)=-\mathfrak t_\minus(x)+2\Upsilon_0(\x)+\mathfrak t_\plus(f(x)),
\end{align}
where $x$ is the starting point of $\x$ in $\Sigma^+$.
Since $|\mathfrak t_\pm|\le\chi$, Lemma \ref{lem:t1-new} implies the following estimate.

\begin{lemma}\label{lem:t1}
If $\x$ is a geodesic with entry vector in $\mathfs D_n$ then
$\Upsilon(\x)\approx n^{\frac{r-2}{r}}$.  
\end{lemma}

Next, we estimate the tail of $\Upsilon$.
Let ${\rm Leb}$ denote the Lebesgue measure on $\Sigma_+$ in Clairaut coordinates.

\begin{lemma}\label{lem:tails}
${\rm Leb}[\{x\in \Sigma_+: \Upsilon(\x)>n\}]\approx n^{-\frac{2r}{r-2}}$.
\end{lemma}

\begin{proof} 
The push-forward of ${\rm Leb}$ under $\mathfrak p_\minus$ 
is equivalent to the Lebesgue measure of $\Omega_+$,
hence we just need to estimate the Lebesgue measure of 
$\mathfrak p_\minus\{x\in \Sigma_+: \Upsilon(\x)>n\}$. Since $|\mathfrak t_{\pm}|\leq \chi$,
equation \eqref{eq:times} gives that
$$
{\rm Leb}\left[\mathfrak p_\minus\{x\in \Sigma_+: \Upsilon(\x)>n\}\right]\approx
{\rm Leb}\left[\{x\in \Omega_+: 2\Upsilon_0(\x)>n\}\right].
$$

By Lemma~\ref{lem:diffpsi}, ${\rm Leb}[\mathfs C_n^{<}]\approx n^{-3}$ and
${\rm Leb}[\mathfs C_n^{>}]\approx n^{-3}$, therefore ${\rm Leb}[\mathfs C_n]\approx n^{-3}$.
Letting $\mathfs U_k=\bigcup\limits_{\ell>k} \mathfs C_\ell$,
it follows that
$$
{\rm Leb}[\mathfs U_k]\approx \sum_{\ell>k}\ell^{-3}\approx k^{-2}.
$$
By Lemma \ref{lem:t1-new}, there exists $C>1$ independent of $k$
such that $C^{-1}k^{\frac{r-2}{r}}\leq 2\Upsilon_0(\x)\leq Ck^{\frac{r-2}{r}}$ for every geodesic
$\x$ with entry vector in $\mathfs C_k$. For such $k$, we have the following:
\begin{enumerate}[$\circ$]
\item If $k>\left(Cn\right)^{\frac{r}{r-2}}$ then $2\Upsilon_0(\x)\geq C^{-1}k^{\frac{r-2}{r}}>n$.
\item If $k\leq \left(C^{-1}n\right)^{\frac{r}{r-2}}$ then $2\Upsilon_0(\x)\leq  Ck^{\frac{r-2}{r}}\leq n$.
\end{enumerate}
This implies the inclusions
$$
\mathfs U_{\left(Cn\right)^{\frac{r}{r-2}}}\subset\left\{x\in \Omega_+: 2\Upsilon_0(\x)>n\right\}\subset
\mathfs U_{\left(C^{-1}n\right)^{\frac{r}{r-2}}}$$
and, since 
${\rm Leb}\big[\mathfs U_{(C^{\pm 1}n)^{\frac{r}{r-2}}}\big]\approx \big(n^{\frac{r}{r-2}}\big)^{-2}=n^{-\frac{2r}{r-2}}$,
the proof is complete.
\end{proof}

\begin{remark}
From Remark \ref{conjugacy-extended-f_0}, $f$ is conjugate to the extended transition map $f_0$,
and $\Upsilon$ is cohomologous to the extended function $2\Upsilon_0$.
More specifically, if $h:\Sigma_0\to\wt\Omega$ is the conjugacy with $h\circ f=f_0\circ h$
then $\Upsilon-2\Upsilon_0\circ h$ is a coboundary for $f$.  

\end{remark}

\subsection{Hyperbolicity properties of $f$ on $\Sigma_+$}\label{ss:hyperbolicity-f}

We now establish some hyperbolicity properties of $f\restriction_{\Sigma_+}$.
Our reference metric is the $\delta$--Sasaki metric $\|\cdot\|=\|\cdot\|_{\delta-{\rm Sas}}$ 
for a small $\delta>0$. Recall that this metric is equivalent to the Sasaki metric $\|\cdot\|_{\rm Sas}$ 
and also to the Clairaut metric $\|\cdot\|_{\rm C}$.
Among the Chernov axioms, the only one that requires a precise multiplicative constant is (A2.5).
Recall that $\wh E^{s/u}$ are the stable/unstable subspaces for $g_t$, which project to 
directions $E^{s/u}$ on $\Sigma_0$ and $\Omega$. Let $P$ denote such projection. Since $\Sigma_0,\Omega$
are nowhere perpendicular to $\wh E^{s,u}$, we have $\|P^{\pm 1}\|\approx 1$.

\begin{lemma}\label{lem:existence-manifolds}
Lebesgue almost every $x\in\Sigma_0$ has an LSM/LUM $W^{s/u}_x$ for $f$.
\end{lemma}

\begin{proof}
We prove the statement for LSM (the argument for LUM is the same, by time reversion).
Since every $x\in\Sigma_0$ has invariant directions $E^{s/u}$, it is enough to show that there is no 
fast convergence of trajectories to $\mathfs S^+$, i.e.\ that $\tfrac{1}{n}\log d(f^n x,\mathfs S^+)\not\to 0$
for a.e.\ $x\in\Sigma_0$. To see this, let $\alpha>0$ and consider the set
$$
A_{\alpha,n}=\{x\in\Sigma_0: d(f^n x ,\mathfs S^+)<e^{-\alpha n}\}.
$$
By the Borel-Cantelli lemma, it is enough to show that $\sum\limits_{n\geq 1}{\rm Leb}[A_{\alpha,n}]<\infty$ for every $\alpha>0$,
in which case $\limsup A_{\alpha,n}$ has zero Lebesgue measure and so does the set
$\{x\in\Sigma_0:\tfrac{1}{n}\log d(f^n x,\mathfs S^+)\to 0\}$. We have
$$
A_{\alpha,n}=\{x\in\Sigma_0: d(f^n x,\mathfs S_{\rm P})<e^{-\alpha n}\}\cup
\{x\in\Sigma_0: d(f^n x,\mathfs S_{\rm S})<e^{-\alpha n}\}.
$$
If $B$ is the first set in the above union, then $f^n(B)$ is covered using finitely many sets of measure 
$\approx e^{-\alpha n}$, one for each curve of $\mathfs S_{\rm P}$. Hence
${\rm Leb}(f^n(B))\ll e^{-\alpha n}$ and so, by $f$--invariance,
we have ${\rm Leb}(B)\ll e^{-\alpha n}$.
If $C$ is the second set in the above union, then $f^n(C)$ is covered using
$\big\lceil e^{\alpha n/2}\big\rceil$ sets of measure 
$\approx e^{-\alpha n}$, namely one to cover all of $\bigcup_{k\geq e^{\alpha n/2}}\mathfs D_k$
and the others to cover $\mathfs D_{n_0},\ldots,\mathfs D_{\lceil e^{\alpha n/2}\rceil-1}$.
Again by $f$--invariance, we get that  ${\rm Leb}(C)\ll e^{\alpha n/2}\cdot e^{-\alpha n}=e^{-\alpha n/2}$.
Hence 
$$
{\rm Leb}(A_{\alpha,n})\ll e^{-\alpha n}+e^{-\alpha n/2}\ll e^{-\alpha n/2}.
$$
The proof is complete.
\end{proof}


By Lemma \ref{lem:subspaces}(3), there are continuous functions
$x\in\Omega_+\mapsto e^s_x,e^u_x$ such that $e^{s/u}_x\in T_x\Omega_+$ is a unit vector
spanning $E^{s/u}_x$. Since our analysis is local, we 
write $f_0(x)=f_0(\theta,\psi)=(\theta\pm \zeta(\psi),\pm\psi)$
omitting the entry $\pm\ve_0$.
We focus on the case
$f_0(x)=(\theta+ \zeta(\psi),\pm\psi)$ since the other sign is treated the same way.
In this notation, 
$$
(df_0)_x=(df_0)_{(\theta,\psi)}=\begin{bmatrix} 1 & \zeta'(\psi) \\ 0 & \pm 1\end{bmatrix}.
$$
If $\x$ is the geodesic defined by $x\in\Omega_+\setminus\Omega_+^=$, then $(df_0)_x\circ P=P\circ(dg_{2\Upsilon_0(\x)})_x$ and so
$\|(df_0)_x\|\approx \|(dg_{2\Upsilon_0(\x)})_x\|$. By equation (\ref{eq:delta-Sasaki}),
we have $\|df_0 e^s_x\|\ll 1$ for $x\in\Omega_+\setminus\Omega_+^=$.

\begin{lemma}\label{lem:Wu}
The following are true.
\begin{enumerate}[{\rm (1)}]
\item There is a H\"older continuous function $a:\Omega_+\to\R$ such that
$E^u_x$ is spanned by $\begin{bmatrix}a(x) \\ 1\end{bmatrix}$ for all $x\in\Omega_+$.
\item For all $x\in\Omega_+$, there a $C^{1+{\rm Lip}}$ function $\Theta$ such that
$W^u_x$ is locally the graph $\{(\Theta(\psi),\psi)\}$ of $\Theta$. 
\end{enumerate}
\end{lemma}

Part (1) says that $E^u$ is not horizontal in the $(\theta,\psi)$ coordinates.

\begin{proof}
Recall that the definition of $\Omega_+$ depends on the small parameter 
$\chi$; 
hence the homogeneity bands $\mathfs C_n$ are only defined for large $n$.
Let $x\in\Omega_+\setminus\Omega_+^=$.
Writing $e^s_x=\begin{bmatrix}e_1(x) \\ e_2(x)\end{bmatrix}$, we have
$$
df_0 e^s_x = \begin{bmatrix} e_1(x)+\zeta'(x)e_2(x) \\ \pm e_2(x) \end{bmatrix}
$$
where $\zeta'\approx n^{3-\frac{2}{r}}$ in $\mathfs C_n$  by Lemma \ref{lem:pw}. 
 Since $\|df_0 e^s_x\|\ll 1$,
we get that $|e_2|\ll n^{-3+\frac{2}{r}}$ and in particular $e_2\to 0$ as $n\to\infty$.
In other words, $e^s_x\to(1,0)$ as $n\to\infty$,
see Figure~\ref{fig:graph}.
\begin{figure}[hbt!]
\centering
\def\svgwidth{13cm}
\begingroup%
  \makeatletter%
  \providecommand\color[2][]{%
    \errmessage{(Inkscape) Color is used for the text in Inkscape, but the package 'color.sty' is not loaded}%
    \renewcommand\color[2][]{}%
  }%
  \providecommand\transparent[1]{%
    \errmessage{(Inkscape) Transparency is used (non-zero) for the text in Inkscape, but the package 'transparent.sty' is not loaded}%
    \renewcommand\transparent[1]{}%
  }%
  \providecommand\rotatebox[2]{#2}%
  \newcommand*\fsize{\dimexpr\f@size pt\relax}%
  \newcommand*\lineheight[1]{\fontsize{\fsize}{#1\fsize}\selectfont}%
  \ifx\svgwidth\undefined%
    \setlength{\unitlength}{655.43009837bp}%
    \ifx\svgscale\undefined%
      \relax%
    \else%
      \setlength{\unitlength}{\unitlength * \real{\svgscale}}%
    \fi%
  \else%
    \setlength{\unitlength}{\svgwidth}%
  \fi%
  \global\let\svgwidth\undefined%
  \global\let\svgscale\undefined%
  \makeatother%
  \begin{picture}(1,0.40589472)%
    \lineheight{1}%
    \setlength\tabcolsep{0pt}%
    \put(0,0){\includegraphics[width=\unitlength,page=1]{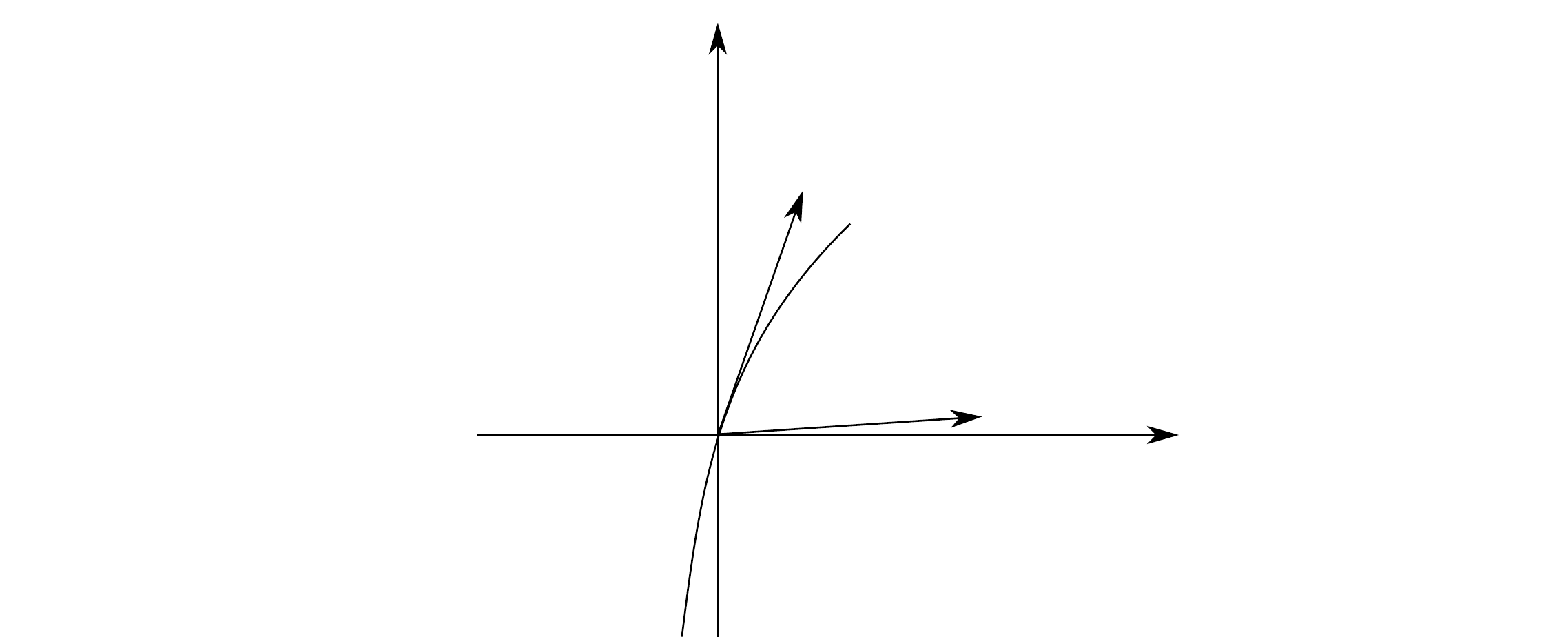}}%
    \put(0.72937188,0.09301643){\color[rgb]{0,0,0}\makebox(0,0)[lt]{\lineheight{1.25}\smash{\begin{tabular}[t]{l}$\theta$\end{tabular}}}}%
    \put(0.42391668,0.37067225){\color[rgb]{0,0,0}\makebox(0,0)[lt]{\lineheight{1.25}\smash{\begin{tabular}[t]{l}$\psi$\end{tabular}}}}%
    \put(0.60254058,0.15709642){\color[rgb]{0,0,0}\makebox(0,0)[lt]{\lineheight{1.25}\smash{\begin{tabular}[t]{l}$e^s_x$\end{tabular}}}}%
    \put(0.47563531,0.2834849){\color[rgb]{0,0,0}\makebox(0,0)[lt]{\lineheight{1.25}\smash{\begin{tabular}[t]{l}$e^u_x$\end{tabular}}}}%
    \put(0.52244056,0.21556556){\color[rgb]{0,0,0}\makebox(0,0)[lt]{\lineheight{1.25}\smash{\begin{tabular}[t]{l}$(\Theta(\psi),\psi)$\end{tabular}}}}%
    \put(0,0){\includegraphics[width=\unitlength,page=2]{graph.pdf}}%
  \end{picture}%
\endgroup%

\caption{We have $e^s_x\to (1,0)$ as $n\to\infty$, hence $W^u_x$ is not horizontal.}\label{fig:graph}
\end{figure}
Since $x\mapsto E^u_x$ is H\"older continuous (Theorem \ref{thm:Gerber-Wilkinson})
and $E^u$ is transverse to $E^s$, part (1) follows. 

For part (2),
recall that by Theorem \ref{thm:Gerber-Wilkinson} the leaves $\wh W^u_x$
are uniformly $C^{1+{\rm Lip}}$. This implies that $W^u_x$ are uniformly $C^{1+{\rm Lip}}$
for $x\in\Omega_+$. Since $W^u_x$ is transverse to the horizontal direction $(1,0)$,
the implicit function theorem implies that we can locally write $W^u_x$ as the graph
$\{(\Theta(\psi),\psi)\}$ of a $C^{1+{\rm Lip}}$ function $\Theta$. 
This completes the proof.
\end{proof}

\begin{lemma}[Growth bounds] \label{lem:growth}
If $x\in\mathfs D_n$ then $\|df\restriction_{E^u_x}\|\approx n^{3-\frac{2}{r}}$.
\end{lemma}

\begin{proof}
Since $\|df\restriction_{E^u_x}\|\approx \|df_0\restriction_{E^u_{\mathfrak p_\minus(x)}}\|$,
we need to prove that $\|df_0\restriction_{E^u_x}\|\approx n^{3-\frac{2}{r}}$
for $x\in \mathfs C_n$. We use the Clairaut metric.
By Lemma \ref{lem:Wu}(1), there exists $a:\Omega_+\to\R$ continuous such that
$E^u_x$ is spanned by $\begin{bmatrix}a(x) \\ 1\end{bmatrix}$ for $x\in\Omega_+$.
If $x=(\theta,\psi)\in\mathfs C_n$, Lemma~\ref{lem:pw} gives that
$df_0\begin{bmatrix}a(x) \\ 1\end{bmatrix}=\begin{bmatrix}a(x)+\zeta'(x) \\ \pm 1\end{bmatrix}$
where $\zeta'(x)\approx n^{3-\frac{2}{r}}$. Hence
\begin{equation} \label{eq:growth}
\|df_0\restriction_{E^u_x}\|_{\rm C}=
\left\|\begin{bmatrix}a(x)+\zeta'(x)\\ \pm1\end{bmatrix}\right\|_{\rm C} \Bigg/ \left\|\begin{bmatrix}a(x)\\ 1\end{bmatrix}\right\|_{\rm C}
=\dfrac{1+|a(x)+\zeta'(x)|}{1+|a(x)|}
\end{equation}
and, since $a$ is bounded, we obtain that
$\|df_0\restriction_{E^u_x}\|_{\rm C}\approx n^{3-\frac{2}{r}}$.
\end{proof}

\begin{lemma}[Distortion bounds]\label{lem:BD}
If $x,\overline{x}\in\mathfs{D}_n^>$ or $x,\overline{x}\in\mathfs{D}_n^<$
with $\overline{x}\in W^u_x$, then 
$$
\big|\log \|df\restriction_{E^u_x}\|-\log \|df\restriction_{E^u_{\overline x}}\|\big| \ll
d(fx, f \overline{x})^\frac13.
$$
\end{lemma}

\begin{proof}
Again, it is enough to prove the estimate for the map $f_0$
in the Clairaut metric. Performing a calculation analogous to the one before
Lemma \ref{lem:Wu} and using the equivalence between metrics, we have
$\|df_0\restriction_{E^u_x}\|_{\rm C}\gg 1$ for $x\in\Omega_+\setminus \Omega_+^=$ and so 
$d_{\rm C}(x,\overline{x})\ll d_{\rm C}(f_0x,f_0\overline{x})$ whenever $\overline x\in W^u_x$.

Write $x=(\theta,\psi)$ and $\overline{x}=(\overline{\theta},\overline{\psi})$.
By equation \eqref{eq:growth},
\begin{align*}
\left|\log \|df_0\restriction_{E^u_x}\|_{\rm C}-\log \|df_0\restriction_{E^u_{\overline x}}\|_{\rm C}\right| & =
\left|\log\left(\dfrac{1+|a(x)+\zeta'(x)|}{1+|a(x)|}\right)-
\log\left(\dfrac{1+|a(\overline{x})+\zeta'(\overline{x})|}{1+|a(\overline{x})|}\right)\right|\\
&=\left|\log\left(\dfrac{1+|a(x)+\zeta'(x)|}{1+|a(\overline{x})+\zeta'(\overline{x})|}\right)-
\log\left(\dfrac{1+|a(x)|}{1+|a(\overline{x})|}\right)\right|\cdot
\end{align*}
Using that $\log\left(\tfrac{|a|}{|b|}\right)=\log\left(1+\tfrac{|a|-|b|}{|b|}\right)\le \tfrac{|a-b|}{|b|}$, we obtain
\begin{align*}
&\left|\log \|df_0\restriction_{E^u_x}\|_{\rm C}-\log \|df_0\restriction_{E^u_{\overline x}}\|_{\rm C}\right| 
 \leq \dfrac{\left||a(x)+\zeta'(x)|-|a(\overline{x})+\zeta'(\overline{x})|\right|}{1+\left|a(\overline{x})+\zeta'(\overline{x})\right|}+\dfrac{\left||a(x)|-|a(\overline{x})|\right|}{1+|a(\overline{x})|}\\
&\le \dfrac{\left|\zeta'(x)-\zeta'(\overline{x})\right|}{1+\left|a(\overline{x})+\zeta'(\overline{x})\right|}+
\dfrac{\left|a(x)-a(\overline{x})\right|}{1+\left|a(\overline{x})+\zeta'(\overline{x})\right|}+
\dfrac{\left|a(x)-a(\overline{x})\right|}{1+\left|a(\overline{x})\right|}\cdot
\end{align*}
Since $a$ is bounded and $n_0$ is large, Lemma \ref{lem:pw} implies that
$|a(\overline{x})+\zeta'(\overline{x})|\geq \tfrac{1}{2}|\zeta'(\overline{x})|$
and so 
\begin{align}\label{eq:distortion}
&\left|\log \|df_0\restriction_{E^u_x}\|_{\rm C}-\log \|df_0\restriction_{E^u_{\overline x}}\|_{\rm C}\right|
\leq 2\dfrac{\left|\zeta'(x)-\zeta'(\overline{x})\right|}{\left|\zeta'(\overline{x})\right|}+2|a(x)-a(\overline{x})|.
\end{align}
This estimate holds for all $x,\overline{x}\in\mathfs{C}_n^{<}$ or $x,\overline{x}\in\mathfs{C}_n^{>}$.
Assuming that $\overline{x}\in W^u_x$, then by Lemma~\ref{lem:Wu}(2) we can write
the latter expression as a function of $\psi$ and apply the mean value theorem to get that
\begin{align*}
\left|\log \|df_0\restriction_{E^u_x}\|_{\rm C}-\log \|df_0\restriction_{E^u_{\overline x}}\|_{\rm C}\right|
& \leq 2\dfrac{\left|\zeta'(\psi)-\zeta'(\overline{\psi})\right|}{\left|\zeta'(\overline{\psi})\right|}+2|a(\psi)-a(\overline{\psi})|\\
&\leq 2\tfrac{\left\|\zeta''\right\|_{\infty}}{\left|\zeta'(\overline{\psi})\right|}|\psi-\overline{\psi}|+2|a(\psi)-a(\overline{\psi})|.
\end{align*}
Also by Lemma~\ref{lem:Wu}(2), the restriction of $a$ to $W^u_x$ is $C^{\rm Lip}$, i.e.\
the map $\psi\mapsto a(\psi)$ is $C^{\rm Lip}$, hence 
$|a(\psi)-a(\overline{\psi})| \ll |\psi-\overline{\psi}|\leq d_{\rm C}(x,\overline{x})\ll d_{\rm C}(f_0x,f_0\overline{x})$.
To estimate the other term, apply Lemma \ref{lem:pw} to get that 
$\tfrac{\left\|\zeta''\right\|_{\infty}}{\left|\zeta'(\overline{\psi})\right|}\ll n^2$ and so
by Lemma~\ref{lem:diffpsi}
$$
\tfrac{\left\|\zeta''\right\|_{\infty}}{\left|\zeta'(\overline{\psi})\right|}|\psi-\overline{\psi}|\ll 
\left[n^2|\psi-\overline{\psi}|^{\frac23}\right]|\psi-\overline{\psi}|^{\frac13}
\ll |\psi-\overline{\psi}|^{\frac13}\leq d_{\rm C}(x,\overline{x})^{\frac13}\ll d_{\rm C}(f_0x,f_0\overline{x})^{\frac13}.
$$
Combining these estimates, we conclude the proof of the lemma.
\end{proof}


\begin{lemma}[Uniform bounds on jacobian of holonomies]\label{lem:ac}
There is a constant $\kappa>0$ such that if $x,\overline{x}\in\mathfs D_n^>$ or $x,\overline{x}\in\mathfs{D}_n^<$
with $\overline{x}\in W^s_x$ then
$$
\big|\log \|df\restriction_{E^u_x}\|-\log \|df\restriction_{E^u_{\overline{x}}}\|\big|\ll d(x,\overline{x})^{\kappa}.
$$
\end{lemma}

\begin{proof}
As usual, we prove the corresponding estimate for $f_0$, taking
the Clairaut metric and assuming that $f_0(\theta,\psi)=(\theta+\zeta,\pm\psi)$.
Let $x=(\theta,\psi)$ and  $\overline{x}=(\overline{\theta},\overline{\psi})$.
Fix $\kappa_0>0$ such that $x\mapsto a(x)$ is $\kappa_0$--H\"older continuous (see Lemma~\ref{lem:Wu}(2)),
i.e. $|a(x)-a(\overline x)|\ll d_{\rm C}(x,\overline x)^{\kappa_0}$. By~\eqref{eq:distortion},
\begin{align*}
\left|\log \|df_0\restriction_{E^u_x}\|_{\rm C}-\log \|df_0\restriction_{E^u_{\overline x}}\|_{\rm C}\right|
& \leq 2\frac{\left|\zeta'(\psi)-\zeta'(\overline{\psi})\right|}{\left|\zeta'(\overline{\psi})\right|}+2|a(x)-a(\overline{x})|\\
&\leq 2\tfrac{\left\|\zeta''\right\|_{\infty}}{\left|\zeta'(\overline{\psi})\right|}|\psi-\overline{\psi}|+2|a(x)-a(\overline{x})|.
\end{align*}
Again,
$\tfrac{\left\|\zeta''\right\|_{\infty}}{\left|\zeta'(\overline{\psi})\right|}|\psi-\overline{\psi}|\ll 
\left[n^2|\psi-\overline{\psi}|^{\frac23}\right]|\psi-\overline{\psi}|^{\frac13}\ll |\psi-\overline{\psi}|^{\frac13}
\leq d_{\rm C}(x,\overline{x})^{\frac13}$. The lemma follows with $\kappa=\min\left\{\frac{1}{3},\kappa_0\right\}$.
\end{proof}

We end this section with an estimate of the variation of $\Upsilon$ on
stable/unstable manifolds.

\begin{lemma} \label{lem:t1var}
If $\x, \overline{\x}$ are both bouncing/crossing geodesics with entry vectors $x,\overline{x}\in\Sigma^+$
such that $\overline{x}\in W^{s/u}_x$ then 
$$
|\Upsilon(\x)-\Upsilon(\overline{\x})|\ll d(x,\overline{x})+d(fx,f\overline{x}).
$$
\end{lemma}

\begin{proof}
For $\chi>0$ small enough, let $\mathfrak p_{\Sigma_0}:g_{[-\chi,\chi]}\Sigma_0\to\Sigma_0$ and 
$\mathfrak t_{\Sigma_0}:g_{[-\chi,\chi]}\Sigma_0\to [-\chi,\chi]$ be the coordinate maps
of the flow box $g_{[-\chi,\chi]}\Sigma_0$.
These maps have the same regularity as the flow $g_t$, hence in particular
are Lipschitz. Since the leaves $\wh W^{s/u}$ are uniformly transverse to the flow direction,
$d(x,y)\approx d(x,\mathfrak p_{\Sigma_0}(y))$ whenever $x\in\Sigma_0$ and
$y\in \wh W^{s/u}_x$ is close to $x$.

Fix $x\in \Sigma^+$ and $\overline{x}\in W^{s/u}_x$. By definition,
$\overline{x}=\mathfrak p_{\Sigma_0}(y)$ for some $y\in \wh W^{s/u}_x$, i.e.\
$y=g_{\mathfrak t_{\Sigma_0}(y)}(\overline{x})$. We also have $fx=g_{\Upsilon(\x)}(x)$ and
$z:=g_{\Upsilon(\x)}(y)\in \wh W^{s/u}_{fx}$. Hence 
$f(\overline{x})=\mathfrak p_{\Sigma_0}(z)=g_{-\mathfrak t_{\Sigma_0}(z)}(z)=
g_{\mathfrak t_{\Sigma_0}(y)+\Upsilon(\x)-\mathfrak t_{\Sigma_0}(z)}(\overline{x})$, so 
$$
\Upsilon(\overline{\x})=\Upsilon(\x)+\mathfrak t_{\Sigma_0}(y)-\mathfrak t_{\Sigma_0}(z).
$$
Since
\begin{enumerate}[$\circ$]
\item $|\mathfrak t_{\Sigma_0}(y)|=|\mathfrak t_{\Sigma_0}(y)-\mathfrak t_{\Sigma_0}(x)|\ll d(x,y)\approx d(x,\overline{x})$ and
\item $|\mathfrak t_{\Sigma_0}(z)|=|\mathfrak t_{\Sigma_0}(z)-\mathfrak t_{\Sigma_0}(fx)|\ll d(z,fx)\approx d(f(\overline{x}),fx)$,
\end{enumerate}
we conclude that
$|\Upsilon(\overline{\x})-\Upsilon(\x)|\leq |\mathfrak t_{\Sigma_0}(y)|+|\mathfrak t_{\Sigma_0}(z)|
\ll d(x,\overline{x})+d(fx,f\overline{x})$.
\end{proof}

\section{The first return map $f$ satisfies the Chernov axioms}\label{sec:CZ-scheme}

In Section \ref{ss:induced-def}, we defined the first return map
$f:X_0\backslash\mathfs S^+\to X_0\backslash\mathfs S^-$, and in Sections \ref{ss:excursions}
and \ref{ss:hyperbolicity-f} we obtained precise estimates for trajectories
that approach the degenerate closed geodesic $\gamma$. Even though the rate of hyperbolicity
of $g_t$ in a neighborhood of $\gamma$ is weak, geodesics spend a long
time during the transition (Lemma \ref{lem:t1}) and so the accumulated hyperbolicity
is large (Lemma \ref{lem:growth}). We now prove that $f$ satisfies 
the Chernov axioms (A1)--(A8) stated in Section \ref{ss:Chernov-axioms-statements}. 

\begin{theorem}\label{thm:f_Chernov}
The first return map $f$ satisfies the Chernov axioms {\rm (A1)--(A8)}. 
\end{theorem}

\begin{proof}
Recall that the parameters $r\geq 4$, $\ve_0>0$ are fixed and we are choosing 
$\chi,\,\delta,\,\eta>0$ small enough, and $n_0\in\N$ large. We consider 
the $\delta$-Sasaki metric, cf.\ Section \ref{sec:geodesic-flows},
which is equivalent to the Clairaut metric on $T^1A$, cf.\ Section \ref{sec:revolution}. 
Let $\Tau:X_0\setminus\mathfs S^+\to(0,\infty)$ be the Poincar\'e return time defined by $f$.
We have $\inf(\Tau)\geq T_{\chi}>0$, cf.\ Section \ref{ss:Sigma}.

\medskip
\noindent
{\sc Verification of (A1):} Recall that $X_0={\rm int}(\Sigma_0)$, cf.\ Section \ref{ss:induced-def}.
Take $\wh X=X=\Sigma_0$, which is a compact Riemannian surface with the metric induced by the
$\delta$--Sasaki metric. We have that $\mathfs S^+,\mathfs S^-$
are closed subsets of $X$, cf.\ Section \ref{ss:induced-def}. The regularity of $f$ is the same of $g_t$,
hence it is a $C^2$ diffeomorphism. This proves axiom (A1).

\medskip
\noindent
{\sc Verification of (A2):} Recall that the directions $E^{s/u}_x$ are the flow projections 
of $\wh E^{s/u}_x$, defined for all $x\in X$. By Lemma \ref{lem:subspaces}(1),
$\{E^{s/u}_x\}$ are $df$--invariant. 
On $\Sigma_+$, $f$ has high rate of hyperbolicity by Lemma \ref{lem:growth}, hence we just need to 
estimate the hyperbolicity of $f$ on $X_0\setminus \Sigma_+$. 
Since each $su$--disc is tangent to $\wh E^{s/u}$ at its center and the $su$--discs used in the construction
of $\Sigma_0$ have radii smaller than $\eta$ (see Step 5 in Section \ref{ss:Sigma}),
there is $C=C(\eta)>1$ with $\lim\limits_{\eta\to 0} C(\eta)=1$ such that
$C^{-1}\|dg_{\Tau(x)}\restriction_{\wh E^{s/u}_x}\|\leq \|df\restriction_{E^{s/u}_x}\|\leq C\|dg_{\Tau(x)}\restriction_{\wh E^{s/u}_x}\|$
for all $x\in X_0\setminus \mathfs S^+$. We estimate the hyperbolicity along $E^u_x$.
The set
$$
Y=\left\{g_t(x):x\in X_0\setminus \Sigma_+\text{ and }0\leq t\leq \Tau(x)\right\}
$$
is at distance $>\ve_0/2$ from $\gamma$. Letting $m_1=\inf(u_+\restriction_Y)$, Proposition \ref{prop:u-continuous} implies that
$0<m_1<\infty$ and so 
equation \eqref{eq:delta-Sasaki} for $E^u$ gives that
\begin{equation*}
\|dg_{\Tau(x)}\restriction_{\wh E^u_x}\|\geq C_{\delta}^{-1}\exp\left[\displaystyle\int_0^{\Tau(x)} u_+(g_tx)dt\right] 
\geq C_{\delta}^{-1}\exp\left[\inf(\Tau)m_1\right]=:\Lambda_u.
\end{equation*}
If $\delta>0$ is small enough then $\Lambda_u>1$. Hence, for $\delta,\eta$ small enough we have
$\|df\restriction_{E^{u}}\|\geq C^{-1}\|dg_{\Tau(x)}\restriction_{\wh E^{u}}\|\geq C^{-1}\Lambda_u>1$.
Arguing similarly with $E^s_x$, we find $\Lambda_s<1$ such that
$\|df\restriction_{E^{s}}\|\leq C\Lambda_s<1$. Choosing 
$1<\Lambda<C^{-1}\min\{\Lambda_s^{-1},\Lambda_u\}$, we have that 
$\|df\restriction_{E^u}\|>\Lambda$ and $\|df^{-1}\restriction_{E^s}\|>\Lambda$.
Finally, choose $\alpha=\alpha(\Lambda)>0$ small enough, and for $x\in X$ define the cones
\begin{align*}
C^s_x&=\{v^s+v^u:v^{s,u}\in E^{s,u}_x\text{ and }\|v^u\|<\alpha \|v^s\|\}\\
C^u_x&=\{v^s+v^u:v^{s,u}\in E^{s,u}_x\text{ and }\|v^s\|<\alpha \|v^u\|\}.
\end{align*}
Hence:
\begin{enumerate}[$\circ$]
\item Condition (A2.1) follows from Lemma \ref{lem:subspaces}(3).
\item Condition (A2.2) follows from Lemma \ref{lem:subspaces}(2).
\item Condition (A2.3) follows from Lemma \ref{lem:subspaces}(4), since $X\cap {\rm Deg}=\emptyset$.
\item Condition (A2.4) follows from Lemma \ref{lem:subspaces}(1).
\item Condition (A2.5) follows because $\|df\restriction_{E^u}\|,\|df^{-1}\restriction_{E^s}\|>\Lambda$ and $\alpha>0$ is small.
\end{enumerate}

\medskip
\noindent
{\sc Verification of (A3).} Recall that $\mathfs S^+=\mathfs S_{\rm P}\cup\mathfs S_{\rm S}$, where
\begin{align*}
\mathfs S_{\rm P}&=\{x\in\Sigma_0:\tau(x)<\infty\text{ and }g_{\tau(x)}(x)\in\partial\Sigma_0\}\cup\{\tau = \infty\},\\
\mathfs S_{\rm S}&=\bigcup_{n\geq n_0}\partial \mathfs D_n.
\end{align*}
By construction, $\partial\Sigma_0$ is transverse to $E^u$, hence by $df$--invariance the set 
$\{x\in\Sigma_0:\tau(x)<\infty\text{ and }g_{\tau(x)}(x)\in\partial\Sigma_0\}$ is as well. We also know that
$\{\tau=\infty\}$ is contained in the stable manifold of ${\rm Deg}$, hence it is transverse to $E^u$.
This shows that $\mathfs S_{\rm P}$ is transverse to $E^u$, but also $\mathfs S_{\rm S}$ since 
the curves $\partial\mathfs D_n$ converge in the $C^1$ norm to $\{\tau=\infty\}$.
Another way to see this later property is observing that, in Clairaut coordinates, $\{\tau=\infty\}$
is contained in $\{\psi=\psi_0\}$ on $\Omega_1$ (with suitable modifications on the remainder of $\Omega_+$) while the two curves composing $\partial\mathfs D_n$ are of the form
$\{\psi=\psi_n\}$ with $\lim\limits_{n\to\infty}\psi_n=\psi_0$.

\medskip
\noindent
{\sc Verification of (A4).} The Liouville measure $\mu$ on $M$ induces a Liouville measure $\mu_{\Sigma_0}$
on $\Sigma_0$,
invariant under $f$. We consider the ergodicity of $f^n$ with respect to $\mu_{\Sigma_0}$.
Since the invariant manifolds of ${\rm Deg}$ have zero Liouville measure, the suspension
flow $(f,\Tau)$ induced by $f$ is isomorphic to the flow $g_t$, hence ergodic.
The suspension flows $(f,\Tau)$ and $(f^n,\Tau_n)$ are isomorphic ($\Tau_n$ denotes the $n$--th Birkhoff sum),
hence $(f^n,\Tau_n)$ is ergodic and so $f^n$ is ergodic.

\medskip
\noindent
{\sc Verification of (A5).} The leaves $W^{s/u}_x$ are flow projections of the leaves
$\wh W^{s/u}_x$ which, by Theorem \ref{thm:Gerber-Wilkinson}, are uniformly $C^{1+\rm{Lip}}$.
Hence the same holds for the leaves $W^{s/u}_x$.

\medskip
\noindent
{\sc Verification of (A6).} We claim that $\psi(s)\approx s^{\frac{1}{3}}$ satisfies (A6).
To see that, observe that if $x,y$ belong to the same connected component of 
$\Sigma_0\setminus\Sigma_+$ then
$$
\left|\log\|Df\restriction_{E^u_x}\|-\log\|Df\restriction_{E^u_y}\|\right|\ll d(fx,fy)^{\frac{1}{3}},
$$
since outside $\Sigma_+$ the map $f$ has the same regularity as $g_t$. Inside 
$\Sigma_+$, Lemma \ref{lem:BD} gives the same estimate if both $x,y\in {\mathfs D}_n^>$
or $x,y\in {\mathfs D}_n^<$. If $W$ is a LUM and $x,y$ belong to a same connected 
component of $W\cap \mathfs S_{n-1}$,
then for all $0\leq k<n$ either $f^kx,f^ky$ belong to the same connected component of 
$\Sigma_0\setminus\Sigma_+$ or $f^kx,f^ky\in {\mathfs D}_{n_k}^>$ or $f^kx,f^ky\in {\mathfs D}_{n_k}^<$
for some $n_k\geq n_0$. Since we also have $d(f^kx,f^ky)\leq \Lambda^{k-n}d(f^n x,f^n y)$ for
$0\leq k<n$, we conclude that
\begin{align*}
&\,\big|\log\|Df^n\restriction_{E^u_x}\|-\log\|Df^n\restriction_{E^u_y}\|\big|
\le \sum_{k=0}^{n-1}\big|\log\|Df\restriction_{E^u_{f^kx}}\|-\log\|Df\restriction_{E^u_{f^ky}}\|\\
& \ll \sum_{k=1}^n d(f^kx,f^ky)^\frac13 \ll \sum_{k=1}^n \Lambda^{\frac13 (k-n)}d(f^n x,f^n y)^\frac13
\ll d(f^n x,f^n y)^{\frac13}.
\end{align*}

\medskip
\noindent
{\sc Verification of (A7).} As in the verification of (A6),
if $x,\overline x$ belong to the same connected component of $\Sigma_0\setminus \Sigma_+$ or
both $x,\overline x\in\mathfs D_n^>$ or both $x,\overline x\in\mathfs D_n^<$ then 
$$
\big|\log\|Df\restriction_{E^u_x}\|-\log\|Df\restriction_{E^u_{\overline x}}\|\big|\leq d(x,y)^\kappa.
$$
Indeed, the first case holds because the restriction of $f$ to $\Sigma_0\setminus\Sigma_+$ has
bounded $C^2$ norm, and the latter cases follow from Lemma \ref{lem:ac}.
By symmetry, it is enough to verify (A7) for unstable holonomies, so 
let $W_1,W_2$ be sufficiently small and close enough
LSM's, and let $H:W_1\to W_2$ be the (unstable) holonomy map. By classical
Pesin theory (see e.g.~\cite[Theorem 8.6.13]{Barreira-Pesin-Non-Uniform-Hyperbolicity-Book}),
the Jacobian $JH$ of $H$ is given by the equation 
$$
\log JH(x)=\sum_{i=0}^\infty 
\Big(\log \big\|Df\restriction_{E^u_{f^{-i}x}}\big\|  -\log \big\|Df\restriction_{E^u_{f^{-i}H(x)}}\big\|\Big).
$$
By the uniform hyperbolicity of $f$, we have $d(f^{-i}x,f^{-i}H(x))\le \Lambda^{-i} d(x,H(x))$ for all $i\geq 0$
and, since $f^{-i}x,f^{-i}H(x)$ belong to the same connected component of $\Sigma_0\setminus \Sigma_+$ or
are both in the same $\mathfs D_n^>$ or $\mathfs D_n^<$, we conclude that 
\begin{align*}
&\, |\log JH(x)| \le \sum_{i=0}^\infty \Big|\log \big\|Df\restriction_{E^u_{f^{-i}x}}\big\|  -\log \big\|Df\restriction_{E^u_{f^{-i}H(x)}}\big\|\Big|\\
&\ll \sum_{i=0}^\infty d(f^{-i}x,f^{-i}H(x))^\kappa \ll \sum_{i=0}^\infty \Lambda^{-\kappa i}d(x,H(x))^\kappa\ll 1.
\end{align*}

\medskip
\noindent
{\sc Verification of (A8).} Let $W$ be a LUM. By (A3), $E^u$ is transverse to the boundaries of
homogeneity bands and so $W\cap \mathfs S_1$ is at most countable.
When it is countable, we can write $W\cap \mathfs S_1=\{x_n,x_{n+1},\ldots\}$ 
where $|c(x_n)|=1\pm\tfrac{1}{n^2}$ and $x_n\to x_\infty\in \Sigma_0\cap g_{[-\delta,\delta]}\Omega_+^=$.
Writing $x_n=(\ve_n,\theta_n,\psi_n)$ in Clairaut coordinates, we obtain that 
$\rho(x_n,x_\infty)\approx |\psi_n-\psi_\infty|\approx |c(x_n)-c(x_\infty)|\approx n^{-2}$, which proves (A8.1).  
For (A8.2), note that by Lemma \ref{lem:growth}
$$
\theta_0:=\liminf_{\delta\to 0}\sup_{|W|<\delta}\sum_{n\geq n_0} \tfrac{1}{\Lambda_n} 
\approx \sum_{n\geq n_0}\tfrac{1}{n^{3-\frac{2}{r}}}\approx  \tfrac{1}{n_0^{2-\frac{2}{r}}}<1,
$$
for $n_0$ large enough, so (A8.2) follows.
Finally, (A8.3) follows from our construction of $\Sigma_0$, which was made to prevent triple intersections
of $\mathfs S_{\rm P}$ up to the $n_0$'th pre-iterate under $f$. In the terminology of Section \ref{ss:Chernov-axioms-statements}, this gives that $K_{P,n_0}=3$, which is obviously smaller than $\min\{\theta_0^{-1},\Lambda\}^{n_0}$
for $n_0$ large enough.
\end{proof}

\section{Proof of Theorem \ref{thm:main} and statistical limit laws}\label{sec:conclusion}

In this section, we prove the results mentioned in the introduction, Theorem~\ref{thm:main} and Remark~\ref{rmk:main} as well as various statistical limit laws.

In the previous sections, we constructed a first return map
$f:\Sigma_0\to\Sigma_0$
that satisfies the Chernov axioms. The return time function
$\tau:\Sigma_0\to(0,\infty)$ 
defined in Section \ref{ss:induced-def} is bounded below but not above.
We have $f=g_\tau$ (i.e.\ $f(x)=g_{\tau(x)}(x)$).

By Theorem \ref{thm:Chernov-Zhang}, $f$ is modelled by a Young tower with exponential tails:
there is a subset $Y\subset \Sigma_0$ with ${\rm Leb}(Y)>0$ and a function $\sigma:Y\to\N$ such that $F=f^{\sigma}:Y\to Y$
is ``nice'' (uniformly hyperbolic with product structure and bounded distortion) and ${\rm Leb}(\sigma>n)\to 0$ exponentially quickly as $n\to\infty$. 
Note that
$$
F=f^\sigma = (g_\tau)^\sigma = g_\varphi
\quad\text{where}
\quad
\varphi=\sum_{\ell=0}^{\sigma-1}\tau\circ f^\ell.
$$
Hence, we have shown that the geodesic flow $g_t$ is modelled by a suspension flow over $F:Y\to Y$ with roof function $\varphi$.

Next we estimate the tails of $\varphi$.
Recalling the definition of $\Sigma_+$ in Section \ref{ss:induced-def},
$\tau=\Upsilon$ is unbounded on $\Sigma_+$, while
$\tau$ is bounded on $\Sigma_0\setminus\Sigma_+$. By Lemma~\ref{lem:tails},
$\mu_{\Sigma_0}(\tau>n)\approx {\rm Leb}(\Upsilon>n)\approx n^{-(a+1)}$ where
 $a=\tfrac{r+2}{r-2}$.
Since $\sigma$ has exponential tails, a standard argument
(see for example~\cite{Markarian-polynomial,Chernov-Zhang}) shows that 
$\mu_{\Sigma_0}(\varphi>n)\ll (\log n)^{a+1}n^{-(a+1)}$.
In particular, $\mu_{\Sigma_0}(\varphi>n)\ll n^{-(a+1-\epsilon)}$ for any $\epsilon>0$.

To apply the recent work of \cite{Balint-et-al-2019},
we also require the following ``bounded H\"older constants'' property for $\varphi$.
\begin{lemma} \label{lem:phi}
We have $|\varphi(x)-\varphi(\overline x)|\ll d(x,\overline x)$ for
$x,\overline x\in Y$ with $\overline x\in W^s_x$, and
$|\varphi(x)-\varphi(\overline x)|\ll d(F x,F \overline x)$ for
$x,\overline x\in Y$ with $\overline x\in W^u_x$.
\end{lemma}

\begin{proof}
This is essentially Lemma \ref{lem:t1var}.
Assume first that $\overline x\in W^u_x$. If $x,\overline x\in\Sigma_+$
then $\tau=\Upsilon$, and so by Lemma~\ref{lem:t1var} we have that 
$|\tau(x)-\tau(\overline x)|\ll d(fx,f\overline x)$.  
Since the same estimate is trivially true for $x,\overline x\in \Sigma_0\setminus\Sigma_+$,
it follows that $|\tau(x)-\tau(\overline x)|\ll d(fx,f\overline x)$ for 
all $x,\overline x\in \Sigma_0$ with $\overline x\in W^u_x$.
By the hyperbolicity of $f$, this estimate implies that $|\varphi(x)-\varphi(\overline x)|\ll d(F x,F \overline x)$
for all $x,\overline x\in Y$ with $\overline x\in W^u_x$.
The argument along stable manifolds is similar.
\end{proof}

In the terminology of~\cite[Section~6]{Balint-et-al-2019}, we have shown that $g_t$ is a ``Gibbs-Markov flow''. The condition (H) in~\cite[Section~6]{Balint-et-al-2019} follows from Lemma~\ref{lem:phi} by~\cite[Lemma~8.3]{Balint-et-al-2019}.\footnote{Condition (8.2) in~\cite[Lemma~8.3]{Balint-et-al-2019} is stated more generally in terms of a separation time $s$. It is standard in the Young tower set up that the estimate in terms of the metric $d$ in Lemma~\ref{lem:phi} is stronger, e.g.\ apply condition~(7.3) from~\cite{Balint-et-al-2019} with $n=1$.}
The main remaining hypothesis in~\cite{Balint-et-al-2019} is 
``absence of approximate eigenfunctions''.  This can be verified using ideas from~\cite[Section~8.4]{Balint-et-al-2019}.  The general setup there applies by Lemma~\ref{lem:phi}.
Since $g_t$ is a geodesic flow and hence has a contact structure, it follows 
from~\cite[Remark~8.11]{Balint-et-al-2019} that absence of approximate eigenfunctions is automatic.

By~\cite[Theorem~6.4]{Balint-et-al-2019}, we can now deduce decay of correlations for the geodesic flow $g_t$ with rate $t^{-(a-\epsilon)}$ as claimed in Theorem~\ref{thm:main}
 for a certain class of observables. However, these observables
belong to a regularity class defined in terms of the abstract suspension flow over $F$ with unbounded roof function $\varphi$.
In order to work with sufficiently smooth observables on the underlying phase space $M$, it is necessary to introduce a new Poincar\'e map $g$ with bounded roof function.

\begin{remark} In the corresponding decay questions for billiards, we would often take $g$ to be the billiard collision map. However, for the geodesic flow, there is no such natural candidate for $g$.
\end{remark}

To construct $g$,
we adjoin the cross-section $\Omega_0$ to $\Sigma_0$.
 Recall that $\Sigma_0$ is constructed
by small flow displacements of $\wh\Sigma$. Since $\wh\Sigma\cup\Omega_0$ is a global
Poincar\'e section for the geodesic flow $g_t$, the same holds for $\Sigma=\Sigma_0\cup\Omega_0$.
Let $g:\Sigma\to\Sigma$ be the corresponding Poincar\'e return map.
Also, let $h:\Sigma_0\to(0,\infty)$ be
the first return time to $\Sigma$ under $g_t$.
Then $g=g_h$ and $h$ is bounded above and below.

Let $\partial_t\phi=\frac{d}{dt}(\phi\circ g_t)|_{t=0}$ denote the derivative in the flow direction. An observable $\phi:M\to\R$ is ``sufficiently smooth'' in Theorem~\ref{thm:main} if $\partial_t^j\phi$ is H\"older for $j=0,\dots,k$, for some $k$ sufficiently large independent of $\epsilon$ and $\phi$.  (Actually, it suffices that $\phi$ and $\psi$ are both H\"older and that one of them is sufficiently smooth.)

In particular, when $r\ge 4$ is an even integer, an observable $\phi$ is sufficiently smooth if it is $C^k$ for $k$ sufficiently large. Otherwise, the flow is not smooth (nor is $M$) and we require in addition that the observable is sufficiently flat at the degenerate geodesic~$\gamma$.

\begin{proof}[Proof of Theorem~\ref{thm:main}]
We are now in the situation of~\cite[Sections 7.1 and~7.2]{Balint-et-al-2019} (where $g_t$ and $g$ are called $T_t$ and $f$, and there is no counterpart of our $f$).
Most of the assumptions therein follow from the existence of the Young tower, and the other assumptions~(7.2), (7.4), (7.5) are immediate (see~\cite[Remark~7.2]{Balint-et-al-2019} for extra information).
The remaining assumptions of~\cite[Corollary~8.1]{Balint-et-al-2019} 
(i.e.\ condition~(H) and absence of approximate eigenfunctions) have been dealt with above. Hence we conclude from~\cite[Corollary~8.1]{Balint-et-al-2019}
 the desired polynomial decay for H\"older observables that are sufficiently smooth. 
\end{proof}

Turning to decay of correlations for the global Poincar\'e map $g$, we can write
$F=g^{\varphi^*}$ where $\varphi^*:Y\to\N$ is the return time to $Y$ under $g$. Hence 
\[
g_\varphi=F=g^{\varphi^*}=(g_h)^{\varphi^*}
\]
so it follows that $\varphi=\sum_{\ell=0}^{\varphi^*-1}h\circ g^\ell$.
Since $h$ is bounded above and below, $\varphi^*$ has the same tails as $\varphi$, hence
\cite[Proposition~5.1]{BMTapp} gives that
$$
\frac{1}{n^{a+1}\log n}\ll \mu_{\Sigma_0}(\varphi^*>n)\ll \frac{(\log n)^{a+1}}{n^{a+1}}\cdot
$$ 
This says that $g:\Sigma\to\Sigma$ is modelled by a Young tower with polynomial tails with tail rate essentially of order $n^{-(a+1)}$. 
The upper and lower bounds in Remark~\ref{rmk:main} now follow from~\cite{Young-polynomial} and~\cite[Theorem~7.4(a)]{BMTapp} respectively.

\subsection{Statistical limit laws}

Since $g:\Sigma\to\Sigma$ is modelled by a Young tower with polynomials tails,
it is possible to read off numerous statistical limit laws for H\"older observables on $\Sigma$. Many of these pass over to the flow.
The results in this subsection do not rely on Theorem~\ref{thm:main} and hence are not restricted to sufficiently smooth observables.

Since $a=\frac{r+2}{r-2}>1$, the return time function $\varphi^*:Y\to\N$ lies in $L^2$. Hence
it follows from~\cite[Corollary~2.1]{MV16} 
that the central limit theorem (CLT) holds for the map $g$. Namely,
let $\phi:\Sigma\to\R$ be H\"older with $\int_\Sigma \phi\,d\mu_\Sigma=0$.
Then $n^{-1/2}\sum_{j=0}^{n-1}\phi\circ g^j$ converges in distribution (with respect to $\mu_\Sigma$) to a (typically nondegenerate) normal distribution.
(At the level of the one-sided Young tower obtained by quotienting stable leaves, the CLT
was proved by~\cite[Theorem~4]{Young-polynomial}.)
By~\cite{Z07}, the convergence in distribution can equally be taken with respect to Lebesgue measure on $\Sigma$.

Statistical limit laws for the flow $g_t$ follow by inducing 
from those for maps, see e.g.\ \cite{MT04,MZ15,KellyM16}. 
Here it is convenient to apply~\cite[Theorem~5.5]{BM18}. (The roof function $\varphi$ and return times $\tau$ and $\sigma$ are denoted by $H$, $h$ and $\tau$ respectively in~\cite{BM18}.) Again $\varphi\in L^2$. The underlying assumptions on $Y$ and conditions~(5.3)--(5.4) at the beginning of~\cite[Section~5]{BM18} are automatic consequences of the fact that $Y$ is the base of a Young tower. The assumptions on $g_t$ in conditions~(5.1)--(5.2) of~\cite{BM18} follow from Proposition~\ref{prop:u-continuous} and~\eqref{eq:delta-Sasaki}.
Finally, the assumptions on $\varphi$ in conditions~(5.1)--(5.2) of~\cite{BM18} were verified in Lemma~\ref{lem:phi}.
By~\cite[Theorem~5.5]{BM18} the CLT holds for the flow. 
Namely, let $\phi:M\to\R$ be H\"older with $\int_M\phi\,d\mu=0$.
Then $t^{-1/2}\int_0^t \phi\circ g_s\,ds$ converges in distribution (with respect to $\mu$ or Lebesgue measure) to a (typically nondegenerate) normal distribution.

A refinement of the CLT is the functional CLT or weak invariance principle (WIP).
Given $\phi:M\to\R$ H\"older with $\int_M\phi\,d\mu=0$, we define 
$W_n(t)=n^{-1/2}\int_0^{nt} \phi\circ g_s\,ds$. Then $W_n$ converges weakly (with respect to $\mu$ or Lebesgue measure) in $C([0,1])$ to Brownian motion $W$
by~\cite[Theorem~5.5]{BM18}.

Finally, we briefly mention applications to homogenization of deterministic fast-slow systems where the aim is to prove convergence,
as the time separation goes to infinity,
 to a stochastic differential equation driven by the Brownian motion $W$.
See~\cite{CFKMZ} for a recent survey.
Using rough path theory, it is sufficient~\cite{KellyM16,KellyM17} to check that the fast dynamics satisfies certain statistical limit laws. As we now explain, our geodesic flow examples $g_t$ satisfy all of these requirements for all $r\ge4$.

First, we require a multidimensional version of the WIP (for observables $v:M\to\R^d$). Again this holds for $g$ by~\cite{MV16} and for the flow 
by~\cite{MZ15}. However, the WIP does not suffice to specify stochastic integrals, and for this one requires the so-called iterated WIP. Once more, this holds for $g$ by~\cite[Corollary~2.3]{MV16} and for the flow by~\cite[Theorem~5.5]{BM18}. The remaining ingredients needed for homogenization are moment bounds and iterated moment bounds. 
Since $h$ is bounded above and below, such bounds for the flow follow by~\cite[Proposition~7.5]{KellyM16} from the corresponding bounds for the map $g$. For two-sided Young towers, such as we have here, optimal bounds for moments and iterated moments were very recently obtained by~\cite{FV}; in particular they hold in the full range 
$r\in [4,\infty)$.

\appendix

\section{On two theorems of Gerber \& Wilkinson}\label{sec:appendix}

In this appendix we show how to adapt the results of Gerber \& Wilkinson
\cite{Gerber-Wilkinson} to prove Theorem \ref{thm:Gerber-Wilkinson}. More specifically,
we show how to obtain \cite[Theorems I and II]{Gerber-Wilkinson} for surfaces with degenerate
closed geodesic,  see Section \ref{sec:our-surface} for the definition.

In \cite{Gerber-Wilkinson}, Theorems I and II are proved for $C^r$ metrics of nonpositive curvature,
where $r\geq 4$ is an integer, under two assumptions on the surface:
\begin{enumerate}[(1)]
\item If $\gamma$ is a geodesic that is not closed, then there is no infinite time interval
$I$ for which $K(\gamma(t))=0$, for all $t\in I$.
\item If $\gamma$ is a closed geodesic, then there is a $t$ such that $K$ vanishes
to order at most $r-3$.
\end{enumerate}
See the statements of the theorems and the remark in \cite[p.~43]{Gerber-Wilkinson}.
It is clear that surfaces with degenerate closed geodesic satisfy assumption (1), regardless of
the value of $r\in [4,\infty)$ (integer or not). But assumption (2), for non-integer $r$, makes no sense.
Our goal is to check that, even though surfaces with degenerate closed geodesic do not satisfy
(2), all estimates of Gerber \& Wilkinson remain true, and so does Theorem \ref{thm:Gerber-Wilkinson}.
The reason is that (2) is used to obtain a control on how the curvature approaches zero:
if $\gamma$ is a closed geodesic of zero curvature, then there are constants $C_1,C_2>0$ such that
\begin{equation}\label{eq:practical-GW}
-C_1 \textrm{dist}(p,\gamma)^{r-2}\leq K(p)\leq-C_2\textrm{dist}(p,\gamma)^{r-2}
\end{equation} 
in a neighborhood of $\gamma$. This estimate does hold for surfaces with degenerate closed geodesic,
as we now explain. Let $S$ be such a surface.
The region containing the closed geodesic $\gamma$ with zero curvature is the surface of revolution $\mathcal N$,
which we call the {\em neck}. The profile function is $\xi(s)=1+|s|^r$, hence by the curvature
formula given in Section \ref{sec:revolution} we have:
\begin{enumerate}[(1)']
\item[(2)'] There are constants $C_1, C_2>0$ such that
$$
-C_1 |s|^{r-2}\leq K(s,\theta)\leq -C_2 |s|^{r-2},
$$
where $(s,\theta)$ are the Clairaut coordinates on the neck $\mathcal N$.
\end{enumerate}

It is clear that (2)' is (\ref{eq:practical-GW}) in our context. In the sequel, we check that
\cite[Theorems I and II]{Gerber-Wilkinson} hold under assumptions (1) and (2)'.  
We warn the reader that, while \cite{Gerber-Wilkinson} uses Fermi coordinates
$(s,a)$, we will maintain our use of the Clairaut coordinates $(s,\theta,\psi)$. 
Let $\mathcal{H}^-$ and $\mathcal{H}^+$ be
the stable and unstable horocycle foliations of $S$. 

\begin{theorem}\label{t.II}
Let $S$ be a surface with degenerate closed geodesic. Then:
\begin{enumerate}[i.]
\item[\rm (i)] The leaves of $\mathcal{H}^-$ and $\mathcal{H}^+$ are uniformly $C^{1+{\rm Lip}}$.
\item[\rm (ii)] The tangent distributions $T\mathcal{H}^-$ and $T\mathcal{H}^+$ are H\"older continuous. 
\end{enumerate}
\end{theorem}


The proof of Theorem \ref{t.II} requires two general lemmas \cite[Lemmas 3.1 and 3.2]{Gerber-Wilkinson},
which we reproduce below (only the items that we explicitly refer to are listed).

\begin{lemma}[Lemma 3.1 of \cite{Gerber-Wilkinson}]\label{l.Lemma31}
Let $K, K_0, K_1:[A,B]\to\mathbb{R}$ be continuous functions and $u, u_0, u_1:[A,B]\to\mathbb{R}$
be solutions of the Riccati equations $u'+u^2+K=0$, $u_i'+u_i^2+K_i=0$, $i=0,1$. Let $y=u_1-u_0$
and $\widehat{j}_i(t)=\exp\left[-\int_t^B u_i(\tau)\,d\tau\right]$, $i=0,1$, and let $j_0,j_1$
be solutions of the Jacobi equations $j_i''+K_i j_i=0$, $i=0,1$. Then the following hold:
\begin{enumerate}[{\rm (i)}]
\item[{\rm (ii)}] $y(B) = y(A) \widehat{j}_0(A) \widehat{j}_1(A) + \int_A^B[K_0(t)-K_1(t)]\widehat{j}_0(t) \widehat{j}_1(t) dt$.
\item[{\rm (iii)}] $\widehat{j}_i(B)=1$ and $\widehat{j}_i$ satisfies the Jacobi equation $\widehat{j}_i''+K_i\widehat{j}_i=0$ for $i=0,1$; moreover, if $u_1\geq 0$ on $[A,B]$, then $0\leq \widehat{j}_1\leq 1$ on $[A,B]$ 
and $\widehat{j}_1'(A)\leq 1/(B-A)$.
\item[{\rm (v)}] If $K$ is nonpositive and $u(A)\geq 0$, then 
$$
u(B)\geq \frac{u(A)}{(B-A)u(A)+1}
$$
and this estimate is an equality whenever $K$ is identically zero.
\item[{\rm (vi)}] If $0\leq j_0(A)\leq j_1(A)$, $0\leq j_0'(A)\leq j_1'(A)$ and $K_1(t)\leq K_0(t)\leq 0$ for all $t\in[A,B]$, then $j_0(B)\leq j_1(B)$. 
\end{enumerate}
\end{lemma}

\begin{lemma}[Lemma 3.2 of \cite{Gerber-Wilkinson}]\label{l.Lemma32} 
Let $f:S\to\mathbb{R}$ be a nonpositive $C^2$ function on a $C^2$ compact surface $S$. 
Define $L:=\sup\left\{\left|\frac{d^2}{dt^2}f(\sigma(t))\right|:\sigma\text{ geodesic and }t\in\mathbb R\right\}$.
Then 
$$|f(p)-f(q)|\leq \sqrt{2L}\sqrt{-f(p)}d(p,q) + \frac{L}{2}d(p,q)^2$$ 
for all $p, q\in S$. 
\end{lemma}

Given $v\in T^1S$, let $k_-(v)$ and $k_+(v)$ be the curvature at $v$ of the stable and unstable horocycles.
Recall from Section \ref{sec:geodesic-flows} that $k_\pm=u_\pm$, and that $k_\pm=0$ only at
$T^1\gamma$. The next result is a lower bound on the 
curvatures of horocycles, which is \cite[Lemma 3.3]{Gerber-Wilkinson} in our context.

\begin{lemma}\label{l.Lemma33lower}
Let $S$ be a surface with degenerate closed geodesic. There is a constant $C_3>0$
with the following property: if $v=(s,\theta,\psi)\in T^1\mathcal N$ in Clairaut coordinates,
then
$$
k_{\pm}(v)\geq C_3\max\{|s|^{(r-2)/2}, |\psi|^{(r-2)/r}\}.
$$
\end{lemma}

\begin{proof}
In \cite{Gerber-Wilkinson}, this lemma is proved in Section 4. It considers a geodesic
that visits the flat region of $S$ in the time interval $[-T,0]$, and decomposes $[-T,0]$ 
according to whether an estimate as in assumption (2)' holds or not.
In our case, the estimate always holds, hence we do not decompose $[-T,0]$.
The other ingredient is \cite[Lemma 4.2]{Gerber-Wilkinson},
a lemma due to K. Burns, which holds in our case in Clairaut coordinates, also due to
assumption (2)'. 
\end{proof}

Using the above lemma, we can control the curvature of the horocycles in terms of the Gaussian curvature
at the basepoint. This estimate is \cite[Lemma 3.4]{Gerber-Wilkinson} in our context.
 
\begin{lemma}\label{l.Lemma34} Let $S$ be a surface with degenerate closed geodesic.
There is a constant $C_4>0$ such that for any $v\in T^1 S$ with basepoint $p\in S$ it holds
$$
k_{\pm}(v)\geq C_4\sqrt{-K(p)}.
$$ 
\end{lemma} 

\begin{proof}
We first assume that $v=(p,\psi)=(s,\theta,\psi)\in T^1\mathcal N$. By assumption (2)',
$-K(p)\leq C_1|s|^{r-2}$. By Lemma \ref{l.Lemma33lower}, we have
$k_{\pm}(v)\geq C_3C_1^{-1/2}\sqrt{-K(p)}$ for $v\in T^1\mathcal N$.
Now, since $k_{\pm}(v)$ is continuous and positive outside $T^1\mathcal N$,
there is $C>0$ such that 
$$
k_{\pm}(v)\geq C\sqrt{-K(p)}
$$
for $v\not\in T^1\mathcal N$. Taking $C_4=\min\{C_3C_1^{-1/2},C\}$, the proof is complete. 
\end{proof}  

As done in \cite[pp. 51 and 52]{Gerber-Wilkinson}, Lemmas \ref{l.Lemma31}, \ref{l.Lemma32}
and \ref{l.Lemma34} imply the next result.

\begin{lemma}\label{l.Lemma35}
Let $S$ be a surface with degenerate closed geodesic.
There are constants $C_5,C_6>0$ with the following property.
Let $\gamma_0,\gamma_1$ be geodesics, let $K_i(t)=K(\gamma_i(t))$ and $u_i:[A,B]\to\mathbb{R}$
be a solution of the Riccati equation $u_i'+u_i^2+K_i=0$, $i=0,1$. 
If $u_0(t)\geq k_+(\gamma_0(t))$ for all $A\leq t\leq B$ and $u_1(A)\geq 0$, then 
$$|u_1(B)-u_0(B)|\leq C_5\varepsilon + C_6(B-A)\varepsilon^2+|u_1(A)-u_0(A)|\widehat{j}_0(A)\widehat{j}_1(A),$$ 
where $\varepsilon:=\max\{d(\gamma_0(t),\gamma_1(t)): t\in[A,B]\}$ and 
$\widehat{j}_i(t):=\exp\left(-\int_t^B u_i(\tau)\,d\tau\right)$ as defined in Lemma~\ref{l.Lemma31}, $i=0,1$. 
\end{lemma} 

\begin{proof}[Proof of part {\rm (i)} of Theorem \ref{t.II}]
Let $S$ be a surface with degenerate closed geodesic, and let
$\gamma_0,\gamma_1$ be geodesics on the same unstable horocycle $\mathcal{W}\subset S$.
We estimate $|k_+(\gamma_0'(0))-k_+(\gamma_1'(0))|$ in terms of
$\varepsilon:=d(\gamma_0(0),\gamma_1(0))$. Since
$t\mapsto d(\gamma_0(0),\gamma_1(t))$ is convex ($K\leq 0$)
and $d(\gamma_0(t),\gamma_1(t))\to 0$ as $t\to+\infty$, we have 
$d(\gamma_0(t),\gamma_1(t))\leq\varepsilon$ for all $t\leq 0$.
Let $u_i$ be the unstable Riccati solutions along $\gamma_i$, $i=0, 1$.
Applying Lemma \ref{l.Lemma35} with $A=-1/\varepsilon$ and $B=0$, we have 
$$
|k_+(\gamma_1'(0))-k_+(\gamma_0'(0))|=|u_1(0)-u_0(0)|\leq C_5\varepsilon + C_6\varepsilon+|u_1(A)-u_0(A)|\widehat{j}_0(A)\widehat{j}_1(A). 
$$ 
Since $u_i\geq 0$, we have $0\leq \widehat{j}_i\leq 1$ and, by Lemma \ref{l.Lemma31}(iii),
$\widehat{j}_i'(A) \leq \frac{1}{B-A}=\varepsilon$, hence 
$$
|k_+(\gamma_1'(0))-k_+(\gamma_0'(0))|\leq C_5\varepsilon + C_6\varepsilon+ \max_{i=0,1} u_i(A) \widehat{j}_i(A) \leq (C_5+C_6+1)\varepsilon. 
$$ 
\end{proof}

Next, we prove part (ii) of Theorem \ref{t.II}. We will need two auxiliary lemma, the first being
\cite[Lemma 3.6]{Gerber-Wilkinson} in our context. Recall that 
$k_+,k_-$ are the curvatures of the stable and unstable horocycles.

\begin{lemma}\label{l.Lemma36}
Let $S$ be a surface with degenerate closed geodesic.
There is a constant $C_7>0$ such that for all $v\in T^1 S$ it holds
$$
\frac{1}{C_7} k_+(v)\leq k_-(v)\leq C_7 k_+(v).
$$ 
\end{lemma} 

\begin{proof}
As in Lemma \ref{l.Lemma34}, we divide the proof into two cases. Start assuming $v\in T^1\mathcal N$.
Assumptions (1) and (2)' allow to apply a result of Gerber \& Nitica \cite[Theorem 3.1]{Gerber-Nitica},
and obtain the upper bound
$$
k_{\pm}(v)\leq C\max\{|s|^{(r-2)/2}, |\psi|^{(r-2)/r}\}
$$ 
for $v=(s,\theta,\psi)\in T^1\mathcal N$ in Clairaut coordinates. 
This and Lemma \ref{l.Lemma33lower} imply that 
$$
\frac{1}{C'} k_+(v)\leq k_-(v)\leq C' k_+(v).
$$ 
for some $C'>0$.
The case $v\not\in T^1\mathcal N$ follows as in the proof of Lemma \ref{l.Lemma34}.
\end{proof} 

The second auxiliary lemma is a simple estimate on solutions of the Riccati equation.

\begin{lemma}[Lemma 3.7 of \cite{Gerber-Wilkinson}]\label{l.Lemma37}
Let $K:[A,B]\to(-\infty,0]$ continuous. If $u_0,u_1$ are solutions of the Riccati equation
$u'+u^2+K=0$ and $u_1(t)\geq u_0(t)>0$ for $t\in [A,B]$, then 
$$
\frac{\exp\left[\int_A^B u_1(t)\, dt\right]}{\exp\left[\int_A^B u_0(t)\, dt\right]}\leq\frac{u_1(A)}{u_0(A)}\cdot
$$
\end{lemma}

\begin{proof}[Proof of part {\rm (ii)} of Theorem \ref{t.II}] 
We wish to show that
$|k_+(v_0)-k_+(v_1)|\leq C d(v_0,v_1)^{\alpha}$ 
for all $v_0,v_1\in T^1S$. As in \cite{Gerber-Wilkinson}, we divide the proof into five steps.

\medskip
\noindent
{\sc Step 1.} It is enough to show that $|k_+(v_0)-k_+(v_1)|\leq C d(v_0,v_1)^{\alpha}$ 
for $v_0,v_1$ with the same basepoint. Indeed, given $v_0,v_1$ with basepoints $p_0,p_1$,
let $v_1'\in T_{p_1}^1S$ be the vector spanning a geodesic negatively asymptotic to the geodesic spanned by $v_0$.
By part (i), we have $|k_+(v_1')-k_+(v_0)|\leq Cd(p_0,p_1)$. 
Since $S$ has nonpositive curvature, Busemann functions are (uniformly) $C^2$ and so 
$d(v_1',v_1)\leq Cd(v_0,v_1)$. 
Hence, if $|k_+(v_1')-k_+(v_1)|\leq Cd(v_1',v_1)^{\alpha}$ then 
$$|k_+(v_0)-k_+(v_1)|\leq Cd(p_0,p_1)+Cd(v_0,v_1)^{\alpha}\leq 2Cd(v_0,v_1)^{\alpha}.$$ 

\medskip
\noindent
{\sc Step 2.}
Given $p\in S$ and $v_0, v_1\in T_p^1S$, let $\omega$ be the angle between $v_0$ and $v_1$.
We can assume that $|\omega|<\omega_0$ for $\omega_0$ small.
Let $v_r\in T_p^1S$, $0\leq r\leq 1$, be the continuous family of unit vectors making angle $r\omega$ with $v_0$,
and let $\gamma_r$ be the variation of geodesics with $\gamma_r(0)=p$ and $\gamma_r'(0)=-v_r$. Define 
$$
T:=\max\{T_0:\text{ the curve }r\in[0,1]\mapsto\gamma_r(t) \textrm{ has length}\leq \sqrt{\omega}
\textrm{ for all } 0\leq t\leq T_0\},
$$ 
and consider the scalar function $j_r(t)$ associated to the perperdicular Jacobi field generated by the variation of geodesics $\gamma_r$. By definition, $j_r(0)=0$, $j_r'(0)=\omega$, and 
$$\int_0^1 j_r(T)\,dr = \sqrt{\omega}.$$ 
Comparing $j_r$ with the solution of the Jacobi equation in zero curvature (cf.~Lemma \ref{l.Lemma31}(vi)), we also have $j_r(T)\geq\omega T$. 
Therefore $T\leq 1/\sqrt{\omega}$. Similarly, comparing with the case of constant curvature $K_{\min}=\inf K$,
we obtain that $j_r(T)\leq \frac{\omega}{\sqrt{-K_{\min}}}\sinh(\sqrt{-K_{\min}}T)$. If 
$\omega_0$ is small enough, then $T>1$. Now, as in \cite[pp. 55]{Gerber-Wilkinson}, 
applying Lemma~\ref{l.Lemma35} and Lemma~\ref{l.Lemma36} we obtain 
\begin{eqnarray*}
|k_+(v_0)-k_+(v_1)|&\leq& C_8\sqrt{\omega} + C_9\left(\exp\left[-\int_1^T u_0(t) dt\right]\right)^{\beta}
\left(\exp\left[-\int_1^T u_1(t) dt\right]\right)^{\beta}
\end{eqnarray*} 
for constants $\beta,C_8,C_9>0$.

\medskip
\noindent
{\sc Step 3.}
We estimate $\exp\left[\int_1^T w_0(t) dt\right]$, where $w_0=j_0'/j_0$
satisfies $w_0'+w_0^2+K\circ\gamma_0=0$ with $w_0(0)=\infty$ and $w_0(t)>0$ for $t>0$. 
Proceeding as in \cite[pp. 56--57]{Gerber-Wilkinson}, there is a constant $C_{10}>0$ 
such that
$$\exp\left[-\int_1^T w_0(t) dt\right]\leq C_{10}\sqrt{\omega}.$$

\medskip
\noindent
{\sc Step 4.}
Assume that $k_+(v_0)\neq 0$. Since $w_0(1)>u_0(1)$, Lemma \ref{l.Lemma37} implies that 
$$
\frac{\exp\left[\int_1^T w_0(t)\, dt\right]}{\exp\left[\int_1^T u_0(t)\, dt\right]}\leq\frac{w_0(1)}{u_0(1)}\cdot
$$ 
By Lemma \ref{l.Lemma31}(v), we have $u_0(1)\geq \tfrac{u_0(0)}{u_0(0)+1}$ and so
$\tfrac{w_0(1)}{u_0(1)}\leq \tfrac{w_0(1)(u_0(0)+1)}{u_0(0)}=\tfrac{w_0(1)(u_0(0)+1)}{k_+(v_0)}$.
Since $K$ is bounded from below, $u_0(0)$ and $w_0(1)$ are bounded from above
and so there is a constant $C_{11}>0$ such that 
$$\exp\left[-\int_1^T u_0(t)\, dt\right]\leq \frac{C_{11}}{k_+(v_0)}\exp\left[-\int_1^T w_0(t)\, dt\right].$$

\medskip
\noindent
{\sc Step 5.}
By Steps 3 and 4, if $k_+(v_i)\neq 0$ then
$$
\exp\left(-\int_1^T u_i(t)\, dt\right)^{\beta}\leq \frac{C_{12}}{k_+(v_i)^{\beta}}\omega^{\beta/2}.
$$ 
for some $C_{12}>0$.
By Step 2, 
$$|k_+(v_0)-k_+(v_1)|\leq C_8\sqrt{\omega} + C_9C_{12}\min\{k_+(v_0), k_+(v_1)\}^{-\beta} \omega^{\beta/2}.$$ 
Recalling that $\omega:=\textrm{dist}(v_0,v_1)$, we conclude the proof, since:
\begin{enumerate}[$\circ$]
\item if $\min\{k_+(v_0), k_+(v_1)\}\leq\omega^{1/4}$, then $|k_+(v_0)-k_+(v_1)|\leq k_+(v_0)+k_+(v_1)\leq 2\omega^{1/4}$; 
\item if $\min\{k_+(v_0), k_+(v_1)\}>\omega^{1/4}$, then $|k_+(v_0)-k_+(v_1)|\leq C_8\sqrt{\omega} + C_9C_{12}\omega^{\beta/4}$.
\end{enumerate} 
\end{proof}

\bibliographystyle{alpha}
\bibliography{bibliography}{}

\end{document}